\documentclass[12pt,a4paper]{amsart}
\usepackage{amsxtra,amssymb,eucal}
\usepackage[cmtip,arrow]{xy}
\usepackage{pb-diagram,pb-xy}

\calclayout

\newcommand{\+}{\nobreakdash-}
\renewcommand{\:}{\colon}
\renewcommand{\.}{\mskip.5\thinmuskip}

\newcommand{\rarrow}{\longrightarrow}
\newcommand{\ot}{\otimes}
\newcommand{\ocn}{\odot}
\newcommand{\Ocn}{\circledcirc}
\newcommand{\wot}{\mathbin{\widehat\otimes}}

\DeclareMathOperator{\Hom}{Hom}
\DeclareMathOperator{\Ext}{Ext}
\DeclareMathOperator{\Tor}{Tor}

\DeclareMathOperator{\Cohom}{Cohom}
\DeclareMathOperator{\SemiExt}{SemiExt}
\DeclareMathOperator{\SemiTor}{SemiTor}

\newcommand{\Hot}{\mathsf{Hot}}

\newcommand{\C}{\mathcal C}
\newcommand{\D}{\mathcal D}
\newcommand{\E}{\mathcal E}

\newcommand{\I}{\mathcal I}
\newcommand{\J}{\mathcal J}
\newcommand{\K}{\mathcal K}
\renewcommand{\L}{\mathcal L}
\newcommand{\M}{\mathcal M}
\newcommand{\N}{\mathcal N}

\renewcommand{\P}{\mathfrak P}
\newcommand{\Q}{\mathfrak Q}
\newcommand{\fI}{\mathfrak I}
\newcommand{\fJ}{\mathfrak J}
\newcommand{\fK}{\mathfrak K}
\newcommand{\fL}{\mathfrak L}
\newcommand{\fR}{\mathfrak R}
\newcommand{\fT}{{\boldsymbol{\mathfrak T}}}

\newcommand{\bS}{{\boldsymbol{\mathcal S}}}
\newcommand{\bK}{{\boldsymbol{\mathcal K}}}
\newcommand{\bM}{{\boldsymbol{\mathcal M}}}
\newcommand{\bN}{{\boldsymbol{\mathcal N}}}
\newcommand{\bP}{{\boldsymbol{\mathfrak P}}}
\newcommand{\bQ}{{\boldsymbol{\mathfrak Q}}}

\newcommand{\Z}{\mathbb Z}
\newcommand{\boL}{\mathbb L}
\newcommand{\boR}{\mathbb R}

\newcommand{\gr}{{\mathrm{gr}}}
\renewcommand{\ss}{{\mathrm{ss}}}
\newcommand{\rop}{{\mathrm{op}}}

\newcommand{\sop}{{\mathsf{op}}}
\newcommand{\co}{{\mathsf{co}}}
\newcommand{\ctr}{{\mathsf{ctr}}}
\newcommand{\abs}{{\mathsf{abs}}}
\newcommand{\si}{{\mathsf{si}}}
\renewcommand{\b}{{\mathsf{b}}}
\newcommand{\fd}{{\mathsf{fd}}}
\newcommand{\qfc}{{\mathsf{qfc}}}
\newcommand{\qfp}{{\mathsf{qfp}}}
\newcommand{\acp}{{\mathsf{acp}}}
\newcommand{\caap}{{\mathsf{caap}}}

\newcommand{\proj}{{\mathsf{proj}}}
\newcommand{\inj}{{\mathsf{inj}}}
\newcommand{\sipr}{{\mathsf{sipr}}}
\newcommand{\siin}{{\mathsf{siin}}}
\newcommand{\wcpr}{{\mathsf{wcpr}}}
\newcommand{\wcin}{{\mathsf{wcin}}}

\newcommand{\sA}{\mathsf{A}}
\newcommand{\sB}{\mathsf{B}}
\newcommand{\sD}{\mathsf{D}}

\newcommand{\oc}{\mathbin{\text{\smaller$\square$}}}
\newcommand{\bu}{{\text{\smaller\smaller$\scriptstyle\bullet$}}}
\newcommand{\lrarrow}{\.\relbar\joinrel\relbar\joinrel\rightarrow\.}

\newcommand{\vect}{{\operatorname{\mathsf{--vect}}}}
\newcommand{\modl}{{\operatorname{\mathsf{--mod}}}}
\newcommand{\modr}{{\operatorname{\mathsf{mod--}}}}
\newcommand{\comodl}{{\operatorname{\mathsf{--comod}}}}
\newcommand{\comodr}{{\operatorname{\mathsf{comod--}}}}
\newcommand{\comodrqfc}{{\operatorname{\mathsf{comod_\qfc--}}}}
\newcommand{\comodrqfcinj}
 {{\operatorname{\mathsf{comod_\qfc^\inj--}}}}
\newcommand{\contra}{{\operatorname{\mathsf{--contra}}}}
\newcommand{\simodl}{{\operatorname{\mathsf{--simod}}}}
\newcommand{\simodr}{{\operatorname{\mathsf{simod--}}}}
\newcommand{\sicntr}{{\operatorname{\mathsf{--sicntr}}}}
\newcommand{\smooth}{{\operatorname{\mathsf{--smooth}}}}
\newcommand{\discr}{{\operatorname{\mathsf{--discr}}}}
\newcommand{\discrR}{{\operatorname{\mathsf{discr--}}}}
\newcommand{\pscomp}{{\operatorname{\mathsf{--pscomp}}}}

\newcommand{\dqfc}{{\operatorname{\mathsf{-qfc}}}}
\newcommand{\dqfp}{{\operatorname{\mathsf{-qfp}}}}
\newcommand{\qfcd}{{\operatorname{\mathsf{qfc-}}}}

\newcommand{\Section}[1]{\bigskip\section{#1}\medskip}
\theoremstyle{plain}
\newtheorem{thm}{Theorem}[section]

\newtheorem{lem}[thm]{Lemma}
\newtheorem{prop}[thm]{Proposition}
\newtheorem{cor}[thm]{Corollary}
\theoremstyle{definition}
\newtheorem{ex}[thm]{Example}
\newtheorem{exs}[thm]{Examples}
\newtheorem{rem}[thm]{Remark}

\begin{document}

\title{Smooth duality and co-contra correspondence}

\author{Leonid Positselski}

\address{Institute of Mathematics of the Czech Academy of Sciences,
\v Zitn\'a~25, 115~67 Prague~1, Czech Republic; and
\newline\indent Laboratory of Algebra and Number Theory, Institute for
Information Transmission Problems, Moscow 127051, Russia}

\email{posic@mccme.ru}

\begin{abstract}
 The aim of this paper is to explain how to get a complex of smooth
representations out of the dual vector space to a smooth representation 
of a $p$\+adic Lie group, in natural characteristic.
 The construction does not depend on any finiteness/admissibility
assumptions.
 Imposing such an assumption, one obtains an involutive duality on
the derived category of complexes of smooth modules with admissible
cohomology modules.
 The paper can serve as an introduction to the results about
representations of locally profinite groups contained in the author's
monograph on semi-infinite homological algebra~\cite{Psemi}.
\end{abstract}

\maketitle

\tableofcontents

\section{Introduction}
\medskip

\subsection{{}}
 Dualizing modules produces modules.
 If $M$ is an $B$\+$A$\+bimodule and $V$ is a left $B$\+module, then
the group of all $B$\+linear maps $\Hom_B(M,V)$ is naturally
a left $A$\+module.
 When the ring $A$ is isomorphic to its opposite ring $A^\rop$
(say, $A$ is commutative or endowed with a Hopf algebra antipode),
this means that we started from a right $A$\+module and came back
to a right $A$\+module.

 The situation gets more complicated when one starts from a module of
a particular class, like a torsion, discrete, or smooth module, and
wants to obtain a module from the same class after the dualization.
 If $\M$ is a torsion abelian group, then the Pontryagin dual group
$\M\spcheck=\Hom_\Z(\M,\mathbb Q/\Z)$ is no longer torsion,
generally speaking.
 There is, however, a covariant derived functor producing
a \emph{two-term complex} of torsion abelian groups out of
the group~$\M\spcheck$ (see, e.~g., \cite[Sections~1.4\+-1.6]{Prev}
and the introduction to~\cite{PMat}).
 This construction does not depend on the topology on
the group $\M\spcheck$, but only on its abelian group structure.

 The general philosophy is that dualizing comodules produces
contramodules, and one can get back from a contramodule to
a complex of comodules using the \emph{derived comodule-contramodule
correspondence} constructions~\cite{Psemi,Pkoszul}.
 Various kinds of torsion, discrete, or smooth modules are viewed
as species of comodules, in one or another
sense~\cite[Sections~0.1\+-0.2]{Prev}.
 Oversimplifying matters a bit, one can say that for every abelian
category of torsion, discrete, or smooth modules there is a much
less familiar, but no less interesting, related abelian category of
contramodules.
 To be more careful, one has to distinguish between left and
right co/contramodules.

 In the simplest example of the coalgebra $\C$ dual to the algebra
$\C\spcheck=k[[t]]$ of formal power series in one variable over
a field~$k$, the $\C$\+comodules are just the $k[t]$\+modules with
a locally nilpotent action of~$t$.
 The $\C$\+contramodules also form a full subcategory in
$k[t]$\+modules; it is described in~\cite[Section~A.1.1]{Psemi}
or~\cite[Sections~1.3 and~1.5\+-1.6]{Prev}.

 For any coalgebra $\C$ over a field~$k$, the dual vector space
$\N\spcheck$ to a right $\C$\+comodule $\N$ is a left
$\C$\+contramodule.
 So is the vector space $\Hom_k(\N,V)$ for an arbitrary $k$\+vector
space~$V$.
 Applying the co-contra correspondence, one obtains a nonpositively
cohomologically graded complex of left $\C$\+comodules
$\boL\Phi_\C(\Hom_k(\N,V))$ out of~$\Hom_k(\N,V)$;
in particular, the complex $\boL\Phi_\C(\N\spcheck)$.
 In the case of $\C\spcheck=k[[t]]$, this will be a two-term
complex; in the case of $\C\spcheck=k[[t_1,\dotsc,t_n]]$,
an $(n+1)$\+term complex.
 When $n=\infty$, it will be sometimes an acyclic complex!
 (See~\cite[Sections~0.2.6\+-0.2.7]{Psemi}.)

\subsection{{}} \label{loc-profinite-group-semialgebra-subsecn}
 In this paper, we are interested in representations of locally
profinite groups~$G$.
 To a profinite group $H$ and a field~$k$, one assigns the coalgebra
$\C=k(H)$ of locally constant $k$\+valued functions on~$H$.
 Discrete $H$\+modules over~$k$ are then the same thing as
$\C$\+comodules.
 To a locally profinite group $G$ with a compact open subgroup
$H\subset\nobreak G$, one assigns a $\C$\+\emph{semialgebra} $\bS=k(G)$
of compactly supported locally constant functions on~$G$.
 Smooth $G$\+modules over~$k$ are the same thing as
\emph{$\bS$\+semimodules}~\cite[Section~E.1]{Psemi},
\cite[Example~2.6]{Prev}.

 More specifically, a $\C$\+semialgebra is an associative algebra
object in the tensor category of $\C$\+$\C$\+bicomodules with
respect to the \emph{cotensor product} operation~$\oc_\C$.
 In other words, a semialgebra $\bS$ over $\C$ is
a $\C$\+$\C$\+bicomodule endowed with a \emph{semimultiplication}
morphism $\bS\oc_\C\bS\rarrow\bS$ and a \emph{semiunit} morphism
$\C\rarrow\bS$ satisfying the conventional axioms.
 In particular, given a locally profinite group $G$ with
a compact open subgroup $H$, one can decompose the multiplication
map $G\times G\rarrow G$ into the composition $G\times G\rarrow
G\times_HG\rarrow G$, where $G\times_HG$ is the quotient of
$G\times G$ by the equivalence relation $(g'h,g'')\sim(g',hg'')$,
where $g'$, $g''\in G$ and $h\in H$.
 Then the pullback of compactly supported locally constant functions
with respect to the quotient map $G\times G\rarrow G\times_HG$
identifies $k(G\times_HG)$ with $\bS\oc_\C\bS\subset\bS\ot_k\bS$, and
the pushforward of such functions with respect to the multiplication
map $G\times_HG\rarrow G$ provides the semimultiplication morphism
$\bS\oc_\C\bS\rarrow\bS$.

\subsection{{}} \label{pseudo-and-contra-subsecn}
 The opposite category to the category of $k$\+vector spaces is
identified with the category of \emph{linearly compact} or
\emph{pseudo-compact} topological vector spaces.
 The identification is provided by the functor assigning to
a vector space $V$ is dual vector space $V\spcheck$; the inverse
functor assigns to a topological vector space $W$ the vector
space of all continuous linear functions $W\rarrow k$.
 Similarly, the opposite category to the category $\comodr\C$ of
right comodules over a coalgebra $\C$ is identified with
the category $\C\spcheck\pscomp$ of \emph{pseudo-compact left modules}
over the \emph{pseudo-compact algebra}~$\C\spcheck$.
 In particular, when $\C=k(H)$ for a profinite group $H$, one has
$$
 \C\spcheck\.=\.k[[H]]\.=\.\varprojlim\nolimits_{U\subset H} k[H/U],
$$
where the projective limit of the group algebras $k[H/U]$ is taken
over all the open normal subgroups $U\subset H$.

 One can also consider the conventional category of left modules
$\C\spcheck\modl$ over the algebra $\C\spcheck$ (viewed as
an abstract $k$\+algebra with the topology forgotten).
 Then there is the forgetful (or in other terms, the dualization)
functor
$$
 (\comodr\C)^\sop\simeq\C\spcheck\pscomp\lrarrow\C\spcheck\modl,
$$
which can be also constructed as the composition
$$
 (\comodr\C)^\sop=(\discrR\C\spcheck)^\sop
 \lrarrow(\modr\C\spcheck)^\sop\lrarrow\C\spcheck\modl
$$
of the fully faithful embedding of $\comodr\C$ as the full
subcategory of discrete right modules in $\modr\C\spcheck$ and
the dualization functor $N\longmapsto N\spcheck\:(\modr\C\spcheck)^\sop
\rarrow\C\spcheck\modl$.

 The category of left $\C$\+contramodules $\C\contra$ stands ``in
between'' the category of pseudo-compact left $\C\spcheck$\+modules
$\C\spcheck\pscomp$ and the category of abstract left
$\C\spcheck$\+modules $\C\spcheck\modl$.
 In other words, the forgetful functor $\C\spcheck\pscomp\rarrow
\C\spcheck\modl$ factorizes naturally into the composition
$$
 \C\spcheck\pscomp\lrarrow\C\contra\lrarrow\C\spcheck\modl.
$$
 All the three categories $\C\spcheck\pscomp$, \ $\C\contra$, and
$\C\spcheck\modl$ are abelian, and both the natural (forgetful)
functors between them are exact.

 A $\C$\+contramodule does not remember the topology of
a pseudo-compact $\C\spcheck$\+mod\-ule; still it carries
more structure than that of an abstract $\C\spcheck$\+module.
 This intermediate structure is that of \emph{infinite summation
operations} with zero-converging families of coefficients
in~$\C\spcheck$ \cite[Sections~1.1, 2.1, and~2.3]{Prev}.
 Sometimes (e.~g., when $\C\spcheck$ is a quotient algebra of
the algebra of formal power series in a finite set of variables),
this structure can be uniquely recovered from
the $\C\spcheck$\+module structure, so $\C\contra$ is
a full subcategory in $\C\spcheck\modl$ \cite[Remark~A.1.1]{Psemi},
\cite[Theorem~B.1.1]{Pweak}, \cite[Theorem~C.5.1]{Pcosh},
\cite[Sections~1.6 and~2.2]{Prev}.
 Generally speaking, it can not and it is~not.

\subsection{{}} \label{co-contra-subsecn}
 There are enough injective objects in the abelian category $\C\comodl$
and enough projective objects in the abelian category $\C\contra$.
 Moreover, the additive categories of injective left
$\C$\+comodules and projective left $\C$\+contramodules are
naturally equivalent~\cite[Section~1.2]{Prev},
\cite[Section~0.2.6]{Psemi}.
 Generally speaking, this allows to construct a natural triangulated
equivalence between the \emph{coderived category} of left
$\C$\+comodules and the \emph{contraderived category} of left
$\C$\+contramodules~\cite[Section~0.2.6]{Psemi},
\cite[Sections~5.1\+-5.2]{Pkoszul},
\begin{equation} \label{co-contra}
 \boR\Psi_\C\:\sD^\co(\C\comodl)\simeq\sD^\ctr(\C\contra):\!\boL\Phi_\C.
\end{equation}
 This equivalence of exotic derived categories can sometimes assign
an acyclic complex to an irreducible
co/contramodule~\cite[Section~0.2.7]{Psemi}.

 When the coalgebra $\C$ has finite homological dimension~$n$,
there is no difference between the co/contraderived and
the conventional derived categories, that is
$$
 \sD^\co(\C\comodl)=\sD(\C\comodl) \quad\text{and}\quad
 \sD^\ctr(\C\contra)=\sD(\C\contra)
$$
(see~\cite[Section~4.5]{Pkoszul} and~\cite[Remark~2.1]{Psemi}).
 Hence the triangulated equivalence
\begin{equation} \label{fhd-co-contra}
 \boR\Psi_\C\:\sD(\C\comodl)\simeq\sD(\C\contra):\!\boL\Phi_\C.
\end{equation}
 The left derived functor $\boL\Phi_\C$ takes a $\C$\+contramodule to
an (at most) $(n+1)$\+term complex of $\C$\+comodules, and
the right derived functor $\boR\Psi_\C$ takes a $\C$\+comodule to
an $(n+1)$\+term complex of $\C$\+contramodules.

\subsection{{}} \label{hom-V-subsecn}
 Given a $k$\+coalgebra $\C$ and a $\C$\+semialgebra $\bS$, for any
right $\bS$\+semimodule $\bN$ and a $k$\+vector space $V$
the vector space $\Hom_k(\bN,V)$ has a natural structure of
left \emph{$\bS$\+semicontramodule}~\cite[Sections~0.3.2
and~0.3.5]{Psemi}, \cite[Section~2.6]{Prev}.
 In particular, given a locally profinite group $G$ and a field~$k$,
for any smooth $G$\+module $\bN$ over~$k$ the vector space
$\Hom_k(\bN,V)$ has a natural structure of \emph{$G$\+contramodule}
over~$k$.

 The $G$\+contramodules over~$k$ are the same thing as
the $\bS$\+semicontramodules for the semialgebra $\bS=k(G)$ over
the coalgebra $\C=k(H)$.
 One has to be careful: the semialgebra $\bS$ depends on the choice of
a compact open subgroup $H\subset G$ (because the coalgebra $\C$ does),
but the notion of a $G$\+contramodule over~$k$ does not depend on this
choice~\cite[Section~E.1]{Psemi}, \cite[Example~2.6]{Prev}.
 One has to choose a compact open subgroup in $G$ if one wants to
interpret smooth $G$\+modules as semimodules and $G$\+contramodules
as semicontramodules over some semialgebra.

 In particular, the $H$\+contramodules over~$k$ are the same thing
as the $\C$\+contramod\-ules for the coalgebra $\C=k(H)$.

\subsection{{}}
 For any $\C$\+semialgebra $\bS$ (satisfying certain adjustness
conditions which always hold for semialgebras arising from
locally profinite groups), there is a natural triangulated equivalence
between certain exotic derived categories of the abelian categories of
left $\bS$\+semimodules and left $\bS$\+semicontramodules.
 These are called the \emph{semiderived categories} and denoted
by $\sD^\si(\bS\simodl)$ and $\sD^\si(\bS\sicntr)$; the natural
equivalence between them is called the \emph{derived
semimodule-semicontramodule correspondence}~\cite[Sections~0.3.7
and~6.3]{Psemi} (cf.~\cite[Section~3.5]{Prev}).

 The triangulated equivalence $\sD^\si(\bS\simodl)\simeq
\sD^\si(\bS\sicntr)$ forms a commutative diagram with the triangulated
equivalence~\eqref{co-contra} and the forgetful functors from
semi(contra)modules to co(ntra)modules,
\begin{equation} \label{semico-semicontra}
\begin{diagram}
\node{\boR\Psi_\bS\:\sD^\si(\bS\simodl)}\arrow{s}\arrow[2]{e,=}
\node[2]{\sD^\si(\bS\sicntr):\!\boL\Phi_\bS}\arrow{s} \\
\node{\boR\Psi_\C\:\sD^\co(\C\comodl)}\arrow[2]{e,=}
\node[2]{\sD^\ctr(\C\contra):\!\boL\Phi_\C}
\end{diagram}
\end{equation}

 The functor $\boR\Psi_\bS$ is the right derived functor of
the functor
$$
 \Psi_\bS\:\bS\simodl\lrarrow\bS\sicntr, \qquad
 \Psi_\bS(\bM)=\Hom_\bS(\bS,\bM)
$$
of homomorphisms in the category of left $\bS$\+semimodules.
 The functor $\boL\Phi_\bS$ is the left derived functor of
the functor
$$
 \Phi_\bS\:\bS\sicntr\lrarrow\bS\simodl,
$$
which is constructed using the operation of \emph{contratensor product}
of right $\bS$\+semimod\-ules and left $\bS$\+semicontramodules,
$$
 \Phi_\bS(\bP)=\bS\Ocn_\bS\bP.
$$
 The contratensor product operation is in some sense dual to
the functor $\Hom^\bS$ of homomorphisms in the category of left
$\bS$\+semicontramodules: for any right $\bS$\+semi\-module $\bN$,
left $\bS$\+semicontramodule $\bP$, and $k$\+vector space $V$
one has
$$
 \Hom_k(\bN\Ocn_\bS\bP,\>V)\simeq\Hom^\bS(\bP,\Hom_k(\bN,V)).
$$
 It follows that the functor $\Phi_\bS$ is left adjoint to
the functor~$\Psi_\bS$.

\subsection{{}}
 In particular, when the coalgebra $\C$ has finite homological
dimension~$n$, there is no difference between the semiderived and
the conventional derived categories of $\bS$\+semi(contra)modules,
$$
 \sD^\si(\bS\simodl)=\sD(\bS\simodl) \quad\text{and}\quad
 \sD^\si(\bS\sicntr)=\sD(\bS\sicntr).
$$
 Hence the commutative diagram of triangulated equivalences and
triangulated forgetful functors
\begin{equation} \label{fhd-semico-semicontra}
\begin{diagram}
\node{\boR\Psi_\bS\:\sD(\bS\simodl)}\arrow{s}\arrow[2]{e,=}
\node[2]{\sD(\bS\sicntr):\!\boL\Phi_\bS}\arrow{s} \\
\node{\boR\Psi_\C\:\sD(\C\comodl)}\arrow[2]{e,=}
\node[2]{\sD(\C\contra):\!\boL\Phi_\C}
\end{diagram}
\end{equation}
 The left derived functor $\boL\Phi_\bS$ takes an $\bS$\+semicontramodule
to an $(n+1)$\+term complex of $\bS$\+semimodules, and
the right derived functor $\boR\Psi_\bS$ takes an $\bS$\+semimodule to
an $(n+1)$\+term complex of $\bS$\+semicontramodules.

 Let us emphasize that no homological dimension condition on
the semialgebra $\bS$ is imposed here, but only on the coalgebra~$\C$.

\subsection{{}} \label{duality-adjunction-introd}
 Given a $\C$\+semialgebra $\bS$ and a $k$\+vector space $V$,
we compose the dualization functor from Section~\ref{hom-V-subsecn}
$$
 \Hom({-},V)\:\sD^\si(\simodr\bS)^\sop\lrarrow\sD^\si(\bS\sicntr)
$$
with the derived semimodule-semicontramodule correspondence
functor~\eqref{semico-semicontra}
$$
 \boL\Phi_\bS\:\sD^\si(\bS\sicntr)\lrarrow
 \sD^\si(\bS\simodl),
$$
obtaining a contravariant triangulated functor
\begin{equation} \label{composite-duality}
 \Delta_\bS^V\:\sD^\si(\simodr\bS)^\sop\lrarrow\sD^\si(\bS\simodl).
\end{equation}

 Similarly one constructs a triangulated functor in the opposite
direction
$$
 \Delta_{\bS^\rop}^V\:\sD^\si(\bS\simodl)^\sop
 \lrarrow\sD^\si(\simodr\bS).
$$
 Consider the simplest (and most interesting) case when $V=k$.
 In what sense and under what restrictions can one expect the two
functors $\Delta_\bS^k$ and $\Delta_{\bS^\rop}^k$ to be
mutually inverse?

 The answer is, first of all, that this \emph{cannot} be literally true
in this form.
 It is instructive to consider the trivial case when $\bS=\C=k$ is
the ground field.
 Then $\bS\simodl=\simodr\bS$ is the category of $k$\+vector spaces,
$\sD^\si(\bS\simodl)=\sD^\si(\simodr\bS)$ is the derived category
of $k$\+vector spaces, and the functor $\Delta_k^k$ takes a complex
of $k$\+vector spaces $U^\bu$ to the complex $U^\bu{}\spcheck=
\Hom_k(U^\bu,k)$.
 Obviously, unless all the cohomology spaces of the complex $U^\bu$
are finite-dimensional, $U^\bu$ is \emph{not} isomorphic
to~$U^\bu{}\spcheck{}\spcheck$.
 So some finiteness conditions have to be imposed at this point.

 What can one say without imposing any finiteness conditions?
 The answer is that the contravariant functors $\Delta_\bS^V$ and
$\Delta_{\bS^\rop}^V$ are \emph{right adjoint to each other}.
 In other words, for any two objects $\bM^\bu\in\sD^\si(\bS\simodl)$
and $\bN^\bu\in\sD^\si(\simodr\bS)$ there is a natural adjunction
isomorphism of $k$\+vector spaces of morphisms
$$
 \Hom_{\sD^\si(\bS\simodl)}(\bM^\bu,\Delta_\bS^V\bN^\bu)\simeq
 \Hom_{\sD^\si(\simodr\bS)}(\bN^\bu,\Delta_{\bS^\rop}^V\bM^\bu)
$$
induced by the natural adjunction morphisms
$\bM^\bu\rarrow\Delta_\bS^V\Delta_{\bS^\rop}^V(\bM^\bu)$ in
$\sD^\si(\bS\simodl)$ and
$\bN^\bu\rarrow\Delta_{\bS^\rop}^V\Delta_\bS^V(\bN^\bu)$ in
$\sD^\si(\simodr\bS)$.

\subsection{{}} \label{coalgebra-duality}
 A general discussion of finiteness and quasi-finiteness conditions
on coalgebras, comodules, and contramodules can be found in
the paper~\cite[Section~2]{Pmc}, the preprint~\cite[Section~2]{Ppc},
and the references therein.
 The most general and convenient condition for our purposes is that
the coalgebra $\C$ should be \emph{left} and
\emph{right quasi-cocoherent}~\cite{Ppc}.

 In this case, the full subcategories of \emph{quasi-finitely
copresented left $\C$\+comodules} $\C\comodl_\qfc\subset\C\comodl$,
quasi-finitely copresented right $\C$\+comodules $\comodrqfc\C\subset
\comodr\C$, and \emph{quasi-finitely presented left $\C$\+contramodules}
$\C\contra_\qfp\subset\C\contra$ are closed under the kernels,
cokernels, and extensions in $\C\comodl$, \,$\comodr\C$, and
$\C\contra$, respectively.
 So $\C\comodl_\qfc$, \,$\comodrqfc\C$, and $\C\contra_\qfp$ are
abelian categories.
 The functor assigning to a right $\C$\+comodule $\N$ its dual left
$\C$\+contramodule $\P=\N\spcheck=\Hom_k(\N,k)$ restricts to
an anti-equivalence of abelian categories
$$
 (\comodrqfc\C)^\sop\simeq\C\contra_\qfp,
$$
(so any quasi-finitely presented left $\C$\+contramodule acquires
a natural pseudo-compact left $\C\spcheck$\+module structure);
while the equivalence between the additive categories of injective
left $\C$\+comodules and projective left $\C$\+contramodules
$\C\comodl^\inj\simeq\C\contra^\proj$ restricts to an equivalence
between the categories of quasi-finitely cogenerated injective left
$\C$\+comodules and quasi-finitely generated projective left
$\C$\+contramodules, {\hbadness=1250
$$
 \C\comodl_\qfc^\inj\simeq\C\contra_\qfp^\proj.
$$

 Now} let us assume that the coalgebra $\C$ is not only left and right
quasi-cocoherent, but also has finite homological dimension.
 Then the derived equivalence~\eqref{fhd-co-contra} restricts to
an equivalence between the full subcategories $\sD^\b(\C\comodl_\qfc)
\subset\sD(\C\comodl)$ and $\sD^\b(\C\contra_\qfp)\subset
\sD(\C\contra)$.
 So one obtains triangulated equivalences
\begin{equation} \label{co-contra-duality}
 \sD^\b(\comodrqfc\C)^\sop\simeq\sD^\b(\C\contra_\qfp)
 \simeq\sD^\b(\C\comodl_\qfc),
\end{equation}
where the composition $\sD^\b(\comodrqfc\C)\rarrow
\sD^\b(\C\comodl_\qfc)$ is the restriction of the functor
$\Delta_\C^k$.
 Thus the restrictions of the adjoint functors $\Delta_\C^k$ and
$\Delta_{\C^\rop}^k$ to the full subcategories $\sD^\b(\comodrqfc\C)
\subset\sD(\comodr\C)$ and $\sD^\b(\C\comodl_\qfc)\subset\sD(\C\comodl)$
are mutually inverse anti-equivalences.

\subsection{{}} \label{semialgebra-duality}
 Finally, let $\bS$ be a semialgebra over a coalgebra~$\C$.
 Consider the full subcategories $\bS\simodl_{\C\dqfc}\subset\bS\simodl$
and $\simodr\bS_{\qfcd\C}\subset\simodr\bS$ consisting of all
the left/right $\bS$\+semimodules that are quasi-finitely copresented
\emph{as\/ $\C$\+comodules}.
 Similarly, consider the full subcategory $\bS\sicntr_{\C\dqfp}\subset
\bS\sicntr$ consisting of al the left $\bS$\+semicontramodules that are
quasi-finitely presented \emph{as\/ $\C$\+contramodules}.
 Then the functor assigning to a right $\bS$\+semimodule $\bN$ its
dual left $\bS$\+semicontramodule $\bP=\bN\spcheck$ restricts to
an anti-equivalence of categories
$$
 (\simodr\bS_{\qfcd\C})^\sop\simeq\bS\sicntr_{\C\dqfp}
$$
(so there is a natural pseudo-compact topology on any
$\bS$\+semicontramodule that is quasi-finitely presented as
a $\C$\+contramodule).
 The two equivalent categories here are abelian whenever the coalgebra
$\C$ is right quasi-cocoherent.

 Now suppose that the coalgebra $\C$ is left and right quasi-cocoherent
of finite homological dimension.
 Consider the full subcategories $\sD^\b_{\C\dqfc}(\bS\simodl)\subset
\sD(\bS\simodl)$ and $\sD^\b_{\qfcd\C}(\simodr\bS)\subset
\sD(\simodr\bS)$ in the derived categories of left and right
$\bS$\+semimodules formed by the bounded complexes of semimodules with
the \emph{cohomology semimodules} that are \emph{quasi-finitely
copresented as\/ $\C$\+comodules}, i.~e., the underlying $\C$\+comodules
of the cohomology $\bS$\+semimodules are quasi-finitely copresented.
 Similarly, consider the full subcategory
$\sD^\b_{\C\dqfp}(\bS\sicntr)\subset\sD(\bS\sicntr)$ in the derived
category of left $\bS$\+semicontramodules formed by the bounded
complexes of semicontramodules with the \emph{cohomology
semicontramodules} that are \emph{quasi-finitely presented as\/
$\C$\+contramodules}, i.~e., the underlying $\C$\+contramodules of
the cohomology $\bS$\+semicontramodules are quasi-finitely presented.

 Then it follows from the discussion in Section~\ref{coalgebra-duality}
that the derived equivalence in the upper line of
the diagram~\eqref{fhd-semico-semicontra} restricts to an equivalence
between the full subcategories $\sD^\b_{\C\dqfc}(\bS\simodl)\subset
\sD(\bS\simodl)$ and $\sD^\b_{\C\dqfp}(\bS\sicntr)\subset
\sD(\bS\sicntr)$.
 Hence one obtains a commutative diagram of triangulated equivalences
and triangulated forgetful functors
\begin{equation} \label{semico-semicontra-duality}
\begin{diagram}
\node{\sD^\b_{\qfcd\C}(\simodr\bS)^\sop}\arrow{s}\arrow[2]{e,=}
\node[2]{\sD^\b_{\C\dqfp}(\bS\sicntr)}\arrow{s}\arrow[2]{e,=}
\node[2]{\sD^\b_{\C\dqfc}(\bS\simodl)}\arrow{s} \\
\node{\sD^\b(\comodrqfc\C)^\sop}\arrow[2]{e,=}
\node[2]{\sD^\b(\C\contra_\qfp)}\arrow[2]{e,=}
\node[2]{\sD^\b(\C\comodl_\qfc)}
\end{diagram}
\end{equation}
where the composition $\sD^\b_{\qfcd\C}(\simodr\bS)\rarrow
\sD^\b_{\C\dqfc}(\bS\simodl)$ is the restriction of the functor
$\Delta_\bS^k$.
 In other words, the restrictions of the adjoint functors $\Delta_\bS^k$
and $\Delta_{\bS^\rop}^k$ to the full subcategories
$\sD^\b_{\qfcd\C}(\simodr\bS)\subset\sD(\simodr\bS)$ and
$\sD^\b_{\C\dqfc}(\bS\simodl)\subset\sD(\bS\simodl)$ are
mutually inverse anti-equivalences.

\subsection{{}}
 Specializing~\eqref{co-contra} to the case of the coalgebra $\C=k(H)$
for a profinite group $H$, we obtain a natural triangulated
equivalence between the coderived category of discrete $H$\+modules
and the contraderived category of $H$\+contramodules over~$k$,
\begin{equation} \label{H-co-contra}
 \boR\Psi_H\:\sD^\co(H\discr_k)\simeq\sD^\ctr(H\contra_k):\!\boL\Phi_H.
\end{equation}
 When the profinite group $H$ has finite $k$\+cohomological dimension
(that is, the homological dimension of the abelian category of
discrete $H$\+modules over~$k$ is finite), this reduces
to an equivalence between the conventional derived
categories~\eqref{fhd-co-contra}
\begin{equation} \label{fhd-H-co-contra}
 \boR\Psi_H\:\sD(H\discr_k)\simeq\sD(H\contra_k):\!\boL\Phi_H.
\end{equation}

 In the case of the semialgebra $\bS=k(G)$ over the coalgebra $\C=k(H)$
corresponding to a locally profinite group $G$ with a compact open
subgroup $H$, we obtain a natural triangulated equivalence between
the semiderived categories of the abelian categories of smooth
$G$\+modules and $G$\+contramodules over~$k$,
\begin{equation} \label{G-semico-semicontra}
 \boR_H\Psi_G\:\sD^\si_H(G\smooth_k)\simeq\sD^\si_H(G\contra_k)
 :\!\boL_H\Phi_G.
\end{equation}
 The subindices $H$ are necessary, because the semiderived categories
depend on the choice of a compact open subgroup $H\subset G$, even
though the abelian categories of smooth $G$\+modules and
$G$\+contramodules do not.
 Moreover, the triangulated equivalences~\eqref{H-co-contra}
and~\eqref{G-semico-semicontra} form a commutative diagram with
the forgetful functors (remembering only the action of $H$ and
forgetting the action of the rest of~$G$)~\eqref{semico-semicontra}
\begin{equation} \label{G-H-semico-semicontra-diagram}
\begin{diagram}
\node{\boR_H\Psi_G\:\sD^\si_H(G\smooth_k)}\arrow{s}\arrow[2]{e,=}
\node[2]{\sD^\si_H(G\contra_k):\!\boL_H\Phi_G}\arrow{s} \\
\node{\boR\Psi_H\:\sD^\co(H\discr_k)}\arrow[2]{e,=}
\node[2]{\sD^\ctr(H\contra_k):\!\boL\Phi_H}
\end{diagram}
\end{equation}

\subsection{{}} \label{fhd-subsecn}
 The situation simplifies considerably when a locally profinite group
$G$ has a compact open subgroup $H$ of finite $k$\+cohomological
dimension.
 Notice that every open subgroup $H'\subset H$ then also has finite
$k$\+cohomological dimension; so the group $G$ has a base of
neighborhoods of zero formed by compact open subgroups of finite
$k$\+cohomological dimension.

 Furthermore, the $H$\+semiderived categories of smooth $G$\+modules
and $G$\+contra\-modules over~$k$ coincide with the conventional
derived categories when $H$ has finite $k$\+cohomological dimension,
$$
 \sD^\si_H(G\smooth_k)=\sD(G\smooth_k) \quad\text{and}\quad
 \sD^\si_H(G\contra_k)=\sD(G\contra_k).
$$
 Moreover, the derived functors $\boR_H\Psi_G$ and $\boL_H\Phi_G$
do not depend on the choice of a compact open subgroup $H\subset G$
of finite $k$\+cohomological dimension.
 So we come to the commutative diagram of triangulated equivalences
and forgetful functors~\eqref{fhd-semico-semicontra}
\begin{equation} \label{fhd-G-H-semico-semicontra-diagram}
\begin{diagram}
\node{\boR\Psi_G\:\sD(G\smooth_k)}\arrow{s}\arrow[2]{e,=}
\node[2]{\sD(G\contra_k):\!\boL\Phi_G}\arrow{s} \\
\node{\boR\Psi_H\:\sD(H\discr_k)}\arrow[2]{e,=}
\node[2]{\sD(H\contra_k):\!\boL\Phi_H}
\end{diagram}
\end{equation}
where the upper line does not depend on the choice of~$H$.

 If the locally profinite group $G$ has a compact open subgroup $H$
of $k$\+cohomological dimension~$n$, then the homological dimension
of the functors $\boR\Psi_G$ and $\boL\Phi_G$ does not exceed~$n$.
 In other words, the functor $\boL\Phi_G$ takes a $G$\+contramodule
over~$k$ to a nonpositively cohomologically graded $(n+1)$\+term
complex of smooth $G$\+modules over~$k$, while the functor $\boR\Psi_G$
takes a smooth $G$\+module over~$k$ to a nonnegatively cohomologically
graded $(n+1)$\+term complex of $G$\+contramodules over~$k$.

 More generally, for any locally profinite group $G$ admitting
a compact open subgroup of finite $k$\+cohomological dimension,
and for any open subgroup $G'\subset G$, one has a commutative diagram
of triangulated equivalences and forgetful functors
\begin{equation} \label{fhd-G-G-prime-semico-semicontra}
\begin{diagram}
\node{\boR\Psi_G\:\sD(G\smooth_k)}\arrow{s}\arrow[2]{e,=}
\node[2]{\sD(G\contra_k):\!\boL\Phi_G}\arrow{s} \\
\node{\boR\Psi_{G'}\:\sD(G'\smooth_k)}\arrow[2]{e,=}
\node[2]{\sD(G'\contra_k):\!\boL\Phi_{G'}}
\end{diagram}
\end{equation}
where the triangulated forgetful functors are induced by the exact
forgetful functors between the abelian categories
$$
 G\smooth_k\lrarrow G'\smooth_k \quad\text{and}\quad
 G\contra_k\lrarrow G'\contra_k.
$$

\subsection{{}} \label{dedualizing-tilting}
 In fact, the situation considered in Section~\ref{fhd-subsecn} stands
in the intersection of several derived covariant duality theories.

 Let $k$ be a field and $G$ be a locally profinite group having
a compact open subgroup $H\subset G$ of finite
$k$\+cohomological dimension~$n$.
 Let $\bS=k(G)$ be the related semialgebra over the coalgebra
$\C=k(H)$.
 Then the $\bS$\+$\bS$\+bisemimodule $\bS$, viewed as a one-term
complex, is a dedualizing complex for the pair of semialgebras
$(\bS,\bS)$ \cite[Examples~4.2]{Pmc}, so~\cite[Theorem~4.3]{Pmc}
applies.

 Furthermore, in the same assumptions, the smooth $G$\+module
$\bS=k(G)$ is an $n$\+tilting object in the Grothendieck abelian
category $\sA=G\smooth_k$ in the sense of, e.~g.,
\cite[Section~6]{NSZ} and~\cite[Definition~2.1 and Theorem~2.2]{FMS}.
 The tilting heart $\sB$ is equivalent to the abelian category
of $G$\+contramodules $G\contra_k$.
 Either approach (see also Section~\ref{derived-equivalence-subsecn}
below) allows to construct the triangulated equivalences
\begin{equation}
 \sD^\star(G\smooth_k)\simeq\sD^\star(G\contra_k)
\end{equation}
for all the bounded or unbounded, conventional or absolute derived
categories with the symbols $\star=\b$, $+$, $-$, $\varnothing$,
$\abs+$, $\abs-$, or~$\abs$.
 We refer to the paper~\cite[Corollary~5.6 and Example~10.10]{PS} for
the details.

\subsection{{}} \label{admissibly-co-presented-subsecn}
 Let $G$ be a locally profinite group, $H\subset G$ be a compact open
subgroup, and $k$~be a field.
 As usually, we say that a smooth $G$\+module $\bM$ over~$k$ is
\emph{admissible} if, for every compact open subgroup $F\subset G$,
the subspace of $F$\+invariant elements $\bM^F\subset\bM$ is
finite-dimensional.
 Equivalently, this means that the underlying smooth $H$\+module
of the smooth $G$\+module $\bM$ is quasi-finitely cogenerated
as a $k(H)$\+comodule.
 Hence the latter property of a smooth $G$\+module $\bM$ over~$k$
does not depend on the choice of a compact open subgroup $H\subset G$.

 Dual-analogously, for any $G$\+contramodule $\bP$ over~$k$ and
a compact open subgroup $F\subset G$ we consider the maximal quotient
contramodule of $\bP$ on which the contraaction of $F$ is trivial,
and denote it by~$\bP_F$.
 This is the contramodule version of the space of coinvariants of $F$
in~$\bP$.
 We say that a $G$\+contramodule $\bP$ over~$k$ is
\emph{contraadmissible} if, for every compact open subgroup
$F\subset G$, the vector space $\bP_F$ is finite-dimensional.
 Equivalently, this means that the underlying $H$\+contramodule
of the $G$\+contramodule $\bP$ is quasi-finitely generated as
a $k(H)$\+contramodule.
 It follows that the latter property of a $G$\+contramodule $\bP$
over~$k$ does not depend on the choice of a compact open subgroup
$H\subset G$.

 A smooth $G$\+module $\bM$ over~$k$ is said to be \emph{admissibly
copresented} if its underlying smooth $H$\+module is quasi-finitely
copresented as a $k(H)$\+comodule, i.~e., in other words,
the underlying smooth $H$\+module of $\bM$ is the kernel of
a morphism between two admissible injective smooth $H$\+modules
over~$k$.
 One can check that this property of a smooth $G$\+module over~$k$
does not depend on the choice of a compact open subgroup $H\subset G$.
 We denote the full subcategory of admissibly copresented smooth
$G$\+modules by $k(G)\simodl_{k(H)\dqfc}=G\smooth_{k,\acp}\subset
G\smooth_k$.

 Similarly, a $G$\+contramodule $\bP$ over~$k$ is said to be
\emph{contraadmissibly presented} if its underlying $H$\+contramodule
is quasi-finitely presented as a $k(H)$\+contramodule, i.~e.,
the underlying $H$\+contramodule of $\bP$ is the cokernel of
a morphism between two contraadmissible projective $H$\+contramodules
over~$k$.
 Once again, one can check that this property of a $G$\+contramodule
over~$k$ does not depend on the choice of a compact open subgroup
$H\subset G$.
 We denote the full subcategory of contraadmissibly presented
$G$\+contramodules by $k(G)\sicntr_{k(H)\dqfp}=G\contra_{k,\caap}\subset
G\contra_k$.

 The functor assigning to a smooth $G$\+module $\bM$ over~$k$ its dual
$G$\+semicontramod\-ule $\bP=\bM\spcheck=\Hom_k(\bM,k)$ restricts to
an anti-equivalence of categories
$$
 (G\smooth_{k,\acp})^\sop\simeq G\contra_{k,\caap}.
$$
 The two equivalent categories here are abelian whenever the coalgebra
$\C=k(H)$ is quasi-cocoherent.
 The latter property does not depend on the choice of a compact open
subgroup $H$ in a given locally profinite group~$G$.

\subsection{{}} \label{loc-profinite-duality-subsecn}
 Let $G$ be a locally profinite group and $k$~be a field.
 Suppose that the coalgebra $\C=k(H)$ is quasi-cocoherent for some
(equivalently, for all) compact open subgroup $H\subset G$.
 Suppose further that there exists a compact open subgroup $H\subset G$
of finite $k$\+cohomological dimension.

 As in Section~\ref{semialgebra-duality}, we consider the full
subcategory $\sD^\b_\acp(G\smooth_k)\subset\sD(G\smooth_k)$ in
the derived category of smooth $G$\+modules over~$k$ formed by
the bounded complexes of smooth $G$\+modules \emph{with admissibly
copresented cohomology $G$\+modules}.
 Similarly, we consider the full subcategory $\sD^\b_\caap(G\contra_k)
\subset\sD(G\contra_k)$ in the derived category of $G$\+contramodules
over~$k$ formed by the bounded complexes of $G$\+contramodules
\emph{with contraadmissibly presented cohomology $G$\+contramodules}.
 Then the commutative diagram of triangulated equivalences and
triangulated forgetful functors~\eqref{semico-semicontra-duality}
takes {the form \hbadness=1875
\begin{equation} \label{loc-profinite-duality-diagram}
\begin{diagram}
\node{\sD^\b_\acp(G\smooth_k)^\sop}\arrow{s}\arrow[2]{e,=}
\node[2]{\sD^\b_\caap(G\contra_k)}\arrow{s}\arrow[2]{e,=}
\node[2]{\sD^\b_\acp(G\smooth_k)}\arrow{s} \\
\node{\sD^\b(H\discr_{k,\acp})^\sop}\arrow[2]{e,=}
\node[2]{\sD^\b(H\contra_{k,\caap})}\arrow[2]{e,=}
\node[2]{\sD^\b(H\discr_{k,\acp})}
\end{diagram}
\end{equation}
where $H\subset G$ is} any chosen compact open subgroup and
$H\discr_{k,\acp}$ is a notation for the abelian category of
admissibly copresented discrete/smooth $H$\+modules over~$k$.

 The composition $\sD^\b_\acp(G\smooth_k)^\sop\rarrow
\sD^\b_\acp(G\smooth_k)$ is an involutive auto-anti-equivalence
of the triangulated category $\sD^\b_\acp(G\smooth_k)$ which
can be obtained as the restriction of the functor $\Delta_\bS^k=
\Delta_G^k=\Delta_{\bS^\rop}^k$, where $\bS=k(G)\simeq\bS^\rop$, to
the full subcategory $\sD^\b_\acp(G\smooth_k)\subset\sD(G\smooth_k)$.
 
 Here the upper line of
the diagram~\eqref{loc-profinite-duality-diagram} does not depend on
the choice of a compact open subgroup $H\subset G$, because the upper
line of the diagram~\eqref{fhd-G-H-semico-semicontra-diagram} does
not depend on it.
 More generally, for any locally profinite group $G$ admitting
a compact open subgroup $H$ of finite $k$\+cohomological dimension
with a quasi-cocoherent coalgebra $k(H)$, and for any open subgroup
$G'\subset G$, we have a commutative diagram of triangulated
equivalences and forgetful functors
(cf.~\eqref{fhd-G-G-prime-semico-semicontra}) 
\begin{equation} \label{G-G-prime-duality-comparison}
\begin{diagram}
\node{\sD^\b_\acp(G\smooth_k)^\sop}\arrow{s}\arrow[2]{e,=}
\node[2]{\sD^\b_\caap(G\contra_k)}\arrow{s}\arrow[2]{e,=}
\node[2]{\sD^\b_\acp(G\smooth_k)}\arrow{s} \\
\node{\sD^\b_\acp(G'\smooth_k)^\sop}\arrow[2]{e,=}
\node[2]{\sD^\b_\caap(G'\contra_k)}\arrow[2]{e,=}
\node[2]{\sD^\b_\acp(G'\smooth_k)}
\end{diagram}
\end{equation}

\subsection{{}} \label{nonnatural-subsecn}
 Let $H$ be a profinite group of the proorder not divisible by
the characteristic of a field~$k$.
 (In particular, $H$ can be an arbitrary profinite group and $k$ a field
of characteristic~$0$.)
 Then the $k$\+cohomological dimension of $H$ is zero, $n=0$.
 In other words, the coalgebra $\C=k(H)$ is cosemisimple.

 Denote by $I_\alpha$ the irreducible discrete representations of $H$
over~$k$.
 Then an arbitrary smooth representation of $H$ over~$k$ has
the form
$$
 \M = \bigoplus\nolimits_\alpha V_\alpha\ot_k I_\alpha,
$$
where $V_\alpha$ are some $k$\+vector spaces.
 At the same time, an arbitrary $H$\+contramodule over~$k$
has the form~\cite[Lemma~A.2.2]{Psemi}
$$
 \P = \prod\nolimits_\alpha V_\alpha\ot_k I_\alpha.
$$
 The derived equivalence $\sD(H\discr_k)\simeq\sD(H\contra_k)$ from
the diagram~\eqref{fhd-G-H-semico-semicontra-diagram} is induced
by the equivalence of (semisimple) abelian categories
$$
 \Psi_H\:H\discr_k\.\simeq\. H\contra_k :\!\Phi_H
$$
taking the discrete $H$\+module $\bigoplus_\alpha V_\alpha\ot I_\alpha$
to the $H$\+contramodule $\prod_\alpha V_\alpha\ot I_\alpha$ and back.
 The functor $\Phi_H$ can be also described as taking
an $H$\+contramodule $\P$ over~$k$ to its $H$\+submodule $\M\subset\P$
formed by all the $H$\+discrete vectors (i.~e., vectors whose
stabilizers are open in~$H$).

 Furthermore, any cosemisimple coalgebra is left and right
co-Noetherian~\cite[Section~2]{Pmc}, hence quasi-co-Noetherian,
and therefore quasi-cocoherent~\cite[Section~2]{Ppc}.
 In the situation at hand, the full subcategory of all admissibly
copresented (\,$=$~admissible) discrete $H$\+modules
$H\discr_{k,\acp}\subset H\discr_k$ consists of all the representations
$\M=\bigoplus_\alpha V_\alpha\ot_k I_\alpha$ such that $V_\alpha$ is
a finite-dimensional $k$\+vector space for every~$\alpha$.
 Similarly, the full subcategory of all contraadmissibly presented
(\,$=$~contraadmissible) $H$\+contramodules
$H\contra_{k,\caap}\subset H\contra_k$ consists of
all the $H$\+contramodules $\P=\prod_\alpha V_\alpha\ot_k I_\alpha$
such that all the vector spaces $V_\alpha$ are finite-dimensional.

 Now let $G$ be a locally profinite group containing a compact
open subgroup $H\subset G$ of the proorder not divisible
by~$\operatorname{char}k$.
 Then the assertions of
Sections~\ref{fhd-subsecn}\+-\ref{dedualizing-tilting} are applicable.
 In particular, the smooth $G$\+module $\bS$ is a projective generator
of the abelian category $G\smooth_k$.
 Moreover, the derived equivalence $\sD(G\smooth_k)\simeq
\sD(G\contra_k)$ from
the diagram~\eqref{fhd-G-H-semico-semicontra-diagram} is induced
by an equivalence of abelian categories
$$
 \Psi_G\:G\smooth_k\.\simeq\. G\contra_k :\!\Phi_G.
$$
 The functor $\Phi_G$ takes a $G$\+contramodule $\bP$ over $k$ to
its submodule $\bM=\Phi_G(\bP)\subset\bP$ of all $G$\+smooth vectors.
 The functor $\Psi_G$ can be described by the rule
$$
 \bP=\Psi_G(\bM)=\Hom_{k[G]}(\bS,\bM),
$$
where $\bS=k(G)$ is the smooth $G$\+$G$\+bimodule of all compactly
supported locally constant functions $G\rarrow k$.

 Restricting the triangulated equivalences and triangulated forgetful
functors from the diagram~\eqref{loc-profinite-duality-diagram} to
the full subcategories of one-term complexes, one obtains a diagram of
abelian category equivalences and forgetful functors
\begin{equation} \label{nonnatural-duality}
\begin{diagram}
\node{G\smooth_{k,\acp}^\sop}\arrow{s}\arrow[2]{e,=}
\node[2]{G\contra_{k,\caap}}\arrow{s}\arrow[2]{e,=}
\node[2]{G\smooth_{k,\acp}}\arrow{s} \\
\node{H\discr_{k,\acp}^\sop}\arrow[2]{e,=}
\node[2]{H\contra_{k,\caap}}\arrow[2]{e,=}
\node[2]{H\discr_{k,\acp}}
\end{diagram}
\end{equation}

\subsection{{}} \label{natural-subsecn}
 Let $G$ be a $p$\+adic Lie group.
 If $\operatorname{char}k\ne p$ (the ``nonnatural characteristic''
case), then $G$ contains a compact open subgroup $H$ of the proorder
not divisible by $\operatorname{char}k$, so the assertions of
Section~\ref{nonnatural-subsecn} apply.

 In this paper, we are mostly interested in the case
$\operatorname{char}k=p$ (the ``natural characteristic'' case).
 According to~\cite[Definition~4.1, Theorem~4.5, and
Theorem~8.32]{DSMS} and~\cite[Corollaire~(1)]{Ser} (see
also~\cite[Theorem~3.1 and the discussion in the preceding
paragraph]{Koh}), $G$ contains an open subgroup $H$ that is
a pro-$p$-group of finite cohomological dimension.
 Hence the assertions of
Sections~\ref{fhd-subsecn}\+-\ref{dedualizing-tilting} apply to
$H$ and~$G$.

 Furthermore, the pro-$p$-group $H$ is finitely generated.
 As we show in this paper, it follows that the the forgetful functor
$$
 H\contra_k\.=\.\C\contra\lrarrow\C\spcheck\modl
$$
from the category of $\C$\+contramodules to the category of abstract
$\C\spcheck$\+modules is fully faithful.
 So is the forgetful functor
$$
 G\contra_k\lrarrow G\modl_k
$$
from the category of $G$\+contra\-modules over~$k$ to the category of
abstract $G$\+modules over~$k$.

 It also follows that the derived functor $\boL\Phi_G$ in
the diagram~\eqref{fhd-G-H-semico-semicontra-diagram} is simply
the left derived functor of the functor of tensor product
$$
 \Phi_G\:G\contra_k\lrarrow G\smooth_k, \qquad
 \Phi_G(\bP)=\bS\ot_{k[G]}\bP.
$$
 Alternatively, the functor $\Phi_G$ can be described as
$$
 \Phi_G(\bP)=\bS\ot_\fT\bP,
$$
where $\fT=\Hom_{k[G]}(\bS,\bS)^\rop$ is the opposite ring to the ring
of endomorphisms of the left $G$\+module $\bS$ over~$k$.
 The functor $\boL\Phi_G\:\sD(G\contra_k)\rarrow\sD(G\smooth_k)$
agrees with the conventional derived tensor product functor
$\bS\ot_\fT^\boL{-}$.

 Similarly, the derived functor $\boR\Psi_G$ is simply the right
derived functor of
$$
 \Psi_G\:G\smooth_k\lrarrow G\contra_k, \qquad
 \Psi_G(\bM)=\Hom_{k[G]}(\bS,\bM)=\Hom_{\fT^\rop}(\bS,\bM).
$$
 The functor $\boR\Psi_G\:\sD(G\smooth_k)\rarrow\sD(G\contra_k)$
agrees with the conventional $\boR\Hom_{\fT^\rop}(\bS,{-})$.

 Finally, the ring $k[[H]]=k(H)\spcheck$ is left and right
Noetherian~\cite[Corollary~7.25 and Exercise~7.6]{DSMS}, hence
the coalgebra $\C=k(H)$ is left and right Artinian and consequently
co-Noetherian~\cite[Section~2]{Pmc}.
 It follows that $\C$ is quasi-co-Noetherian and
quasi-cocoherent~\cite[Section~2]{Ppc}.
 A discrete $H$\+module is admissibly copresented if and only if it
is admissible, and if and only if the related $\C$\+comodule is
Artinian.
 An $H$\+contramodule is contraadmissibly presented if and only if it
is contraadmissible, and if and only if the related $\C$\+contramodule
is co-Artinian~\cite[Section~2]{Pmc}.

 Thus the right self-adjoint contravariant triangulated functor
$\Delta_G^k=\Delta_{k(G)}^k\:\allowbreak\sD(G\smooth_k)^\sop\rarrow
\sD(G\smooth_k)$ restricts to an involutive triangulated
auto-anti-equivalence of the full subcategory $\sD^\b_\acp(G\smooth_k)
\subset\sD(G\smooth_k)$ of bounded complexes of smooth $G$\+modules
with admissible cohomology $G$\+modules, and we have the commutative
diagrams~(\ref{loc-profinite-duality-diagram}\+-%
\ref{G-G-prime-duality-comparison}) of triangulated equivalences and
triangulated forgetful functors, as in
Section~\ref{loc-profinite-duality-subsecn}.

\subsection{{}}
 Before we finish this introduction, let us mention a result
demonstrating that the semiderived category $\sD^\si_H(G\smooth_k)$
is a natural triangulated category to assign to a locally profinite
group $G$ with a compact open pro-$p$ subgroup $H$ of a more general
kind than $p$\+adic Lie groups.
 The question becomes nontrivial when the pro-$p$-group $H$ has
infinite cohomological dimension (and $\operatorname{char}k=p$).

 For any coalgebra $\C$ over a field~$k$, the coderived category
of left $\C$\+comodules $\sD^\co(\C\comodl)$ is compactly generated
by its full subcategory of finite complexes of finite-dimensional
comodules.
 This full subcategory is equivalent to the bounded derived category
$\sD^\b(\C\comodl_\fd)$ of the abelian category of finite-dimensional
left $\C$\+comodules~\cite[Sections~3.11 and/or~4.6 and~5.5]{Pkoszul}.

 In particular, let $H$ be a profinite group and $k$~be a field.
 Denote by $I_\alpha$ the irreducible discrete representations of
$H$ over~$k$.
 Then the objects $I_\alpha\in H\discr_k\subset\sD^\co(H\discr_k)$
form a set of compact generators of the coderived category of
discrete representations $\sD^\co(H\discr_k)$.
 When $H$ is a pro-$p$-group and $\operatorname{char}k=p$, there is
a unique irreducible discrete representation of $H$ over~$k$,
namely, the trivial representation $I_0=k$.
 So $I_0$ is a single compact generator of the coderived category
$\sD^\co(H\discr_k)$.

 For a semialgebra $\bS$ over $\C$, one can consider the induction
functor
$$
 \operatorname{ind}_\C^\bS\:\C\comodl\lrarrow\bS\simodl,
 \qquad \operatorname{ind}_\C^\bS(\M)=\bS\oc_\C\M.
$$
 The functor $\operatorname{ind}_\C^\bS$ is left adjoint to
the forgetful functor $\bS\simodl\rarrow\C\comodl$.
 The functor $\operatorname{ind}_\C^\bS$ is exact, so it induces
a triangulated functor
$$
 \operatorname{ind}_\C^\bS\:\sD^\co(\C\comodl)\lrarrow
 \sD^\si(\bS\simodl).
$$
left adjoint to the forgetful functor $\sD^\si(\bS\simodl)\rarrow
\sD^\co(\C\comodl)$.
 The triangulated forgetful functor preserves coproducts, so
the triangulated induction functor takes compact objects to
compact objects.

 Denote by $I_\alpha$ the irreducible left $\C$\+comodules.
 Then the objects $\operatorname{ind}_\C^\bS(I_\alpha)\in
\bS\simodl\subset\sD^\si(\bS\simodl)$ are compact generators of
the triangulated category $\sD^\si(\bS\simodl)$.
 This follows from the fact that the triangulated forgetful functor
$\sD^\si(\bS\simodl)\rarrow\sD^\co(\C\comodl)$ is, by the definition
of the semiderived category, conservative, i.~e., it takes nonzero
objects to nonzero objects.

 Now let $G$ be a locally profinite group with a compact open
subgroup $H\subset G$.
 Specializing to the case $\C=k(H)$ and $\bS=k(G)$, we obtain
the functor of compactly supported smooth induction
$$
 \operatorname{ind}_H^G\:H\discr_k\lrarrow G\smooth_k
$$
left adjoint to the forgetful functor $G\smooth_k\rarrow H\discr_k$.
 Let $I_\alpha$ be the irreducible discrete representations of $H$
over~$k$.
 Then the objects $\operatorname{ind}_H^G(I_\alpha)\in
G\smooth_k\subset\sD^\si_H(G\smooth_k)$ are compact generators of
the semiderived category $\sD^\si_H(G\smooth_k)$.
 When $H$ is a pro-$p$-group and $\operatorname{char}k=p$,
the smooth $G$\+module $\operatorname{ind}_H^G(I_0)$ is a single
compact generator of the category $\sD^\si_H(G\smooth_k)$.

 Now we can define the Hecke DG\+algebra.
 The full subcategories of bounded below complexes in
$\sD(G\smooth_k)$, \ $\sD^\si_H(G\smooth_k)$, and
$\sD^\co(G\smooth_k)$ are all equivalent, so it does not matter
in (a DG\+enhancement of) which of these triangulated categories such
a DG\+algebra is computed.
 As in~\cite[Section~3]{Schn}, it suffices to pick a right injective
resolution $\J^\bu$ of the object $\operatorname{ind}_H^G(I_0)$ in
the abelian category $G\smooth_k$, and set
$A^\bu=\Hom_{k[G]}(\J^\bu,\J^\bu)^\rop$.

 The same arguments as in~\cite{Schn} provide a triangulated
equivalence
\begin{equation}
 \sD^\si_H(G\smooth_k)\simeq\sD(A^\bu\modl)
\end{equation}
between the $H$\+semiderived category of smooth $G$\+modules over~$k$
and the derived category of left DG\+modules over the Hecke
DG\+algebra~$A^\bu$.
 This is a generalization of the Schneider derived
equivalence~\cite[Theorem~9]{Schn} to the case of an arbitrary locally
profinite group $G$ with a compact open pro-$p$ subgroup $H\subset G$.

\subsection{{}}
 In conclusion, let us mention that many results of this paper
can be extended easily to the case of a discrete commutative coefficient
ring~$k$ of finite homological dimension (in place of a field).
 Indeed, the exposition in~\cite[Sections~E.1--E.3]{Psemi} is given
in precisely this generality.
 In order not to intimidate the reader, we prefer to work over a field
in this paper.
 For the same reason, we do not go into a discussion of the absolute
derived categories.

\bigskip

\noindent\textbf{Acknowlegdements.}
 This paper grew out of the author's participation in the workshop
``Geometric methods in the mod~$p$ local Langlands correspondence''
in Pisa on June~6--10, 2016.
 Listening to the talk of Jan Kohlhaase on smooth duality at
the workshop and subsequently reading his preprint~\cite{Koh} were
particularly strong influences.
 I~would like to thank the workshop organizers for inviting me,
and Centro di Ricerca Matematica Ennio De~Giorgi of Scuola Normale
Superiore di~Pisa for excellent working conditions at the workshop.
 Last but not least, I~am indebted to an anonymous reviewer for many
questions and suggestions which helped to improve and expand
the exposition and make it (hopefully) more accessible to
a wider audience.
 In particular, Sections~\ref{duality-adjunction-subsecn}
and~\ref{admissibility-secn}\+-\ref{involutive-secn} were written in
responce to a question asked by the reviewer.
 The author's research was supported by the Israel Science Foundation
grant~\#\,446/15, by the Grant Agency of the Czech Republic
under the grant P201/12/G028, and by research plan RVO:~67985840.

\Section{Discrete $H$-Modules and $H$-Contramodules}
\label{profinite-group-secn}

\subsection{Comodules and contramodules}
\label{comodules-contramodules-subsecn}
 Let $k$ be a field.
 We will denote the category of (discrete or abstract) $k$\+vector spaces
by $k\vect$.
 A complete, separated topological $k$\+vector space $W$ is said to be
\emph{pseudo-compact} if its open vector subspaces of finite codimension
form a base of neighborhoods of zero in it.
 In particular, a finite-dimensional topological vector space is
pseudo-compact if and only if it is discrete, and a pseudo-compact
topological vector space is discrete if and only if it is
finite-dimensional.
 The category of pseudo-compact topological vector spaces and continuous
linear maps between them will be denoted by $k\pscomp$.

 For any discrete $k$\+vector space $V$, the dual vector space
$V\spcheck=\Hom_k(V,k)$ carries a natural pseudo-compact topology in
which the annihilators of finite-dimensional subspaces in $V$ form
a base of neighborhoods of zero.
 Conversely, to any pseudo-compact $k$\+vector space $W$ one assigns
the discrete vector space $V$ of all continuous linear functions
$W\rarrow k$ (where the topology on~$k$ is discrete).
 This correspondence defines an anti-equivalence between the categories
of discrete and pseudo-compact $k$\+vector spaces,
\begin{equation} \label{linear-duality}
 (k\vect)^\sop\simeq k\pscomp.
\end{equation}

 The \emph{completed tensor product} of two pseudo-compact vector spaces
$W_1$ and $W_2$ is defined as the projective limit
$$
 W_1\wot_k W_2 = \varprojlim\nolimits_{U_1,U_2} W_1/U_1\ot_k W_2/U_2,
$$
where $U_1$ ranges over the open vector subspaces in $W_1$ and $U_2$
over the open vector subspaces in $W_2$.
 The $k$\+vector space $W_1\wot_k W_2$ is endowed with the topology
of projective limit of discrete finite-dimensional vector spaces
$W_1/U_1\ot_k W_2/U_2$, making $W_1\wot_k W_2$ a pseudo-compact
topological vector space.
 The anti-equivalence of categories~\eqref{linear-duality} takes
the tensor product of discrete $k$\+vector spaces to the completed
tensor product of pseudo-compact $k$\+vector spaces; so it is
an anti-equivalence of tensor categories.

 A (coassociative and counital) \emph{coalgebra} $\C$ over~$k$ is
$k$\+vector space endowed with a \emph{comultiplication} map
$\mu\:\C\rarrow\C\ot_k\C$ and a \emph{counit} map
$\varepsilon\:\C\rarrow k$ satisfying the conventional
coassociativity and counitality axioms.
 A \emph{left\/ $\C$\+comodule} $\M$ is a $k$\+vector space endowed
with a \emph{left coaction} map $\nu_\M\:\M\rarrow\C\ot_k\M$, and
a \emph{right\/ $\C$\+comodule} $\N$ is a $k$\+vector space endowed
with a \emph{right coaction} map $\nu_\N\:\N\rarrow\N\ot_k\C$.
 The coassociativity and counitality axioms have to be satisfied in
both cases.
 We refer to~\cite{Swe} or~\cite{Prev} for the details.

 The \emph{Sweedler notation}~\cite{Swe} for the comultiplication in
a coalgebra $\C$ has the form
$$
 \mu(c)=\sum\nolimits_i c^{(1)}_i\ot c^{(2)}_i=c^{(1)}\ot c^{(2)},
$$
where $c$, $c^{(1)}_i$, $c^{(2)}_i\in\C$, while $c^{(1)}\ot c^{(2)}$
is a simplified symbolic form of the notation.
 Similarly, the Sweedler notation for the left coaction map~$\nu_\M$
is
$$
 \nu_\M(x)=\sum\nolimits_i x^{(-1)}_i\ot x^{(0)}_i = x^{(-1)}\ot x^{(0)},
$$
where $x$, $x^{(0)}_i\in\M$ and $x^{(-1)}_i\in\C$; and the notation
for the right coaction map~$\nu_\N$ is
$$
 \nu_\N(y)=\sum\nolimits_i y^{(0)}_i\ot y^{(1)}_i = y^{(0)}\ot y^{(1)},
$$
where $y$, $y^{(0)}_i\in\N$ and $y^{(1)}_i\in\C$.
 In every one of these cases, the index~$i$ actually ranges over
a finite set; and this index is often omitted, together with
the summation sign over~$i$, for the sake of brevity.

 The anti-equivalence of tensor categories~\eqref{linear-duality}
transforms coalgebra objects in the category $k\vect$ into algebra
objects in the category $k\pscomp$ and back.
 In other words, for any coassociative, counital coalgebra $\C$
over~$k$, the dual pseudo-compact vector space $\C\spcheck$ acquires
an associative, unital \emph{pseudo-compact algebra} structure given by
a pseudo-compact multiplication map $m=\mu\spcheck\:
\C\spcheck\wot_k\C\spcheck\rarrow\C\spcheck$ and a unit map
$e=\varepsilon\spcheck\:k\rarrow\C\spcheck$.
 The datum of a map~$m$ is equivalent to the datum of a conventional
multiplication $\C\spcheck\times\C\spcheck\rarrow\C\spcheck$ that is
continuous as a function of two variables in the pseudo-compact topology,
$k$\+linear, associative, and unital with the unit element $e(1)\in\C$.
 In other words, the pseudo-compact algebra $\C\spcheck$ has
an underlying structure of a discrete or nontopological associative
algebra (as well as an underlying structure of a pseudo-compact
vector space).

 There still remains a choice between the two opposite multiplications
on the pseudo-compact vector space $\C\spcheck$: which one of the two
opposite pseudo-compact algebras is to be denoted by $\C\spcheck$, and
which one by $\C\spcheck{}^\rop$?
 (See the discussion in~\cite[Section~1.4]{Prev}.)
 As in~\cite{Prev}, we make the slightly nonstandard choice of
defining the multiplication in $\C\spcheck$ by the formula
$$
 \langle fg, c \rangle =
 \langle f,c^{(2)}\rangle \langle g,c^{(1)}\rangle,
 \qquad f,g \in\C\spcheck, \ c\in\C,
$$
where $\langle\,,\,\rangle$ denotes the natural pairing $\C\spcheck\times
\C\rarrow k$.

 Then, for any coalgebra $\C$ over~$k$, the anti-equivalence of tensor
categories~\eqref{linear-duality} transforms left $\C$\+comodules $\M$
into pseudo-compact right $\C\spcheck$\+modules $\M\spcheck$ and right
$\C$\+comodules $\N$ into pseudo-compact left
$\C\spcheck$\+modules~$\N\spcheck$.
 So the anti-equivalence of tensor categories~\eqref{linear-duality}
induces an anti-equivalence between the categories of right
$\C$\+comodules and pseudo-compact left $\C\spcheck$\+modules,
\begin{equation} \label{linear-duality-over-C}
 (\comodr\C)^\sop\simeq\C\spcheck\pscomp,
\end{equation}
and similarly on the other side.
 What is important for us, though, is that the right $\C\spcheck$\+module
$\M\spcheck$ can be obtained by applying the duality functor $\Hom_k({-},k)$
to a natural left $\C\spcheck$\+module structure on~$\M$.
 The natural left $\C\spcheck$\+module structure on a left $\C$\+comodule
$\M$ is defined by the composion
$$
 \C\spcheck\ot_k\M\lrarrow\C\spcheck\ot_k\C\ot_k\M\lrarrow\M
$$
of the map $\C\spcheck\ot_k\M\rarrow\C\spcheck\ot_k\C\ot_k\M$ induced by
the coaction map $\M\rarrow\C\ot_k\M$ and the map $\C\spcheck\ot_k
\C\ot_k\M\rarrow\M$ induced by the pairing $\C\spcheck\times\C\rarrow k$.
 The natural right $\C\spcheck$\+module structure on a right
$\C$\+comodule $\N$ is constructed similarly.

 So we have two functors
\begin{gather}
 \C\comodl\lrarrow\C\spcheck\modl, \\
 \comodr\C\lrarrow\modr\C\spcheck.
 \label{right-co-forgetful}
\end{gather}
 These functors are always fully faithful, and identify the category of
(left or right) $\C$\+comodules with the category of \emph{discrete}
(left or right) $\C\spcheck$\+modules with respect to the pseudo-compact
topology on $\C\spcheck$, i.~e., the $\C\spcheck$\+modules every vector
in which has an open annihilator in~$\C\spcheck$ \cite[Section~2.1]{Swe}.

 A \emph{left\/ $\C$\+contramodule} $\P$ is a $k$\+vector space 
endowed with a \emph{left contraaction} map $\pi_\P\:\Hom_k(\C,\P)
\rarrow\P$.
 The map $\pi_\P$ must satisfy the \emph{contraassociativity} and
\emph{contraunitality} axioms: namely, the two maps
$\Hom_k(\C,\Hom_k(\C,\P))=\Hom_k(\C\ot_k\C,\>\P)\rightrightarrows
\Hom_k(\C,\P)$, one of them induced by the comultiplication
map~$\mu$ and the other one by the contraaction map~$\pi_\P$,
should have equal compositions with the contraaction map~$\pi_\P$,
$$
 \Hom_k(\C,\Hom_k(\C,\P))\.=\.\Hom_k(\C\ot_k\C,\>\P)
 \.\rightrightarrows\. \Hom_k(\C,\P)\rarrow\P,
$$
and the composition of the map $\P\rarrow\Hom_k(\C,\P)$ induced by
the counit map~$\varepsilon$ with the contraaction map~$\pi_\P$
should be equal to the identity map~$\operatorname{id}_\P$,
$$
 \P\lrarrow\Hom_k(\C,\P)\lrarrow\P.
$$
 Here the identification $\Hom_k(V,\.\Hom_k(U,W))\simeq
\Hom_k(U\ot_kV,\>W)$ is presumed in the contraassociativity axiom
(using $\Hom_k(U,\.\Hom_k(V,W))\simeq\Hom_k(U\ot_kV,\>W)$ would lead
to the definition of a \emph{right $\C$\+contramodule}).

 Given a right $\C$\+comodule $\N$ and a $k$\+vector space $V$,
the vector space $\P=\Hom_k(\N,V)$ has a natural left
$\C$\+contramodule structure with the contraaction map
$\pi_\P\:\Hom_k(\C,\P)\rarrow\P$ constructed as the composition
\begin{equation} \label{right-co-left-contra}
 \Hom_k(\C,\Hom_k(\N,V))\.\simeq\.\Hom_k(\N\ot_k\C,\>V)
 \lrarrow\Hom_k(\N,V)
\end{equation}
of the natural isomorphism of Hom spaces and the map induced
by the right coaction map $\nu_\N\:\N\rarrow\N\ot_k\C$.

 The category of left $\C$\+comodules $\C\comodl$ is abelian with
exact functors of infinite direct sum and an injective
cogenerator $\C$, while the category of left $\C$\+contramodules
$\C\contra$ is abelian with exact functors of infinite product and
a projective generator~$\C\spcheck$.
 More generally, a \emph{cofree} left $\C$\+comodule is
a $\C$\+comodule of the form $\C\ot_kV$, and a \emph{free}
left $\C$\+contramodule is a $\C$\+contramodule of the form
$\Hom_k(\C,V)$, where $V$ is a $k$\+vector space.
 For any left $\C$\+comodule $\M$, left $\C$\+comodule morphisms
$\M\rarrow\C\ot_kV$ correspond bijectively to $k$\+linear maps
$\M\rarrow V$; and similarly, for any left $\C$\+contramodule $\P$,
left $\C$\+contramodule morphisms $\Hom_k(\C,V)\rarrow\P$ correspond
bijectively to $k$\+linear maps $V\rarrow\P$.
 The injective objects of $\C\comodl$ are precisely the direct
summands of the cofree left $\C$\+comodules, and the projective
objects of $\C\contra$ are precisely the direct summands of
free left $\C$\+contramodules~\cite[Section~1.2]{Prev}.

 Every left $\C$\+contramodule $\P$ has a natural underlying structure
of a left $\C\spcheck$\+module.
 The left action of $\C\spcheck$ in $\P$ is given by the composition
$$
 \C\spcheck\ot_k\P\lrarrow\Hom_k(\C,\P)\lrarrow\P
$$
of the natural inclusion $\C\spcheck\ot_k\P\hookrightarrow\Hom_k(\C,\P)$
(whose image is the vector subspace of all linear operators $\C\rarrow\P$
of finite rank in the vector space of arbitrary such operators) and
the left contraaction map $\Hom_k(\C,\P)\rarrow\P$.
 So there is a faithful forgetful functor
\begin{equation} \label{contra-forgetful}
 \C\contra\lrarrow\C\spcheck\modl.
\end{equation}

 Following the above discussion, for any right $\C$\+comodule $\N$
the vector space $\N\spcheck=\Hom_k(\N,k)$ has a natural left
$\C$\+contramodule structure.
 On the other hand, $\N\spcheck$ has a natural pseudo-compact topology
and a structure of a pseudo-compact left $\C\spcheck$\+module.
 Moreover, all the pseudo-compact left $\C\spcheck$\+modules come from
right $\C$\+comodules $\N$ in this way
(see~\eqref{linear-duality-over-C}).
 Thus the functor of forgetting the pseudo-compact topology of
a pseudo-compact left $\C\spcheck$\+module (assigning to
a pseudo-compact left $\C\spcheck$\+module its underlying nontopological
left $\C\spcheck$\+module) factorizes naturally through the category of
left $\C$\+contramodules,
\begin{equation} \label{two-forgetful-functors}
 \C\spcheck\pscomp\lrarrow\C\contra\lrarrow\C\spcheck\modl.
\end{equation}

 Let us say a few words about what the functor $\C\spcheck\pscomp\rarrow
\C\contra$ does.
 It forgets the topology of a pseudo-compact $\C\spcheck$\+module while
keeping a remnant of it in the form of the \emph{infinite summation
operations} with zero-convergent families of elements of $\C\spcheck$ as
the coefficients.
 We refer to~\cite[Sections~1.3, 2.1, and~2.3]{Prev} for a discussion of
the contramodule infinite summation operations.
 Subsequently, the functor $\C\contra\rarrow\C\spcheck\modl$ forgets
the infinite summation operations, too, keeping only the conventional
action of the $k$\+algebra $\C\spcheck$ in a module over it.
 (See also the discussion in Section~\ref{pseudo-and-contra-subsecn}
of the introduction.)

 Both the functors in~\eqref{two-forgetful-functors} are faithful, but
neither one of them is full, generally speaking.
 In other words, the action of these functors on morphisms is injective,
but not, in general, surjective.
 Furthermore, neither one of the two functors is surjective on
the isomorphism classes of objects (not even of objects
finite-dimensional over~$k$).
 We refer to~\cite[Section~A.1.2]{Psemi} for a detailed discussion.

 The following trivial example is helpful to keep in mind: when $\C=k$
is the trivial coalgebra, one has $\C\spcheck\pscomp=k\pscomp$ and
$\C\contra=k\vect=\C\spcheck\modl$.
 So the functor $\C\contra\rarrow\C\spcheck\modl$ is an equivalence of
categories in this case, but the functor $\C\spcheck\pscomp\rarrow
\C\contra$ is very far from being an equivalence.
 More generally, the functor $\C\spcheck\pscomp\rarrow\C\contra$ is
\emph{never} an equivalence of categories for a nonzero coalgebra $\C$,
while the functor $\C\contra\rarrow\C\spcheck\modl$ is an equivalence
for any finite-dimensional coalgebra $\C$, as well as in some more
interesting cases covered by
Theorem~\ref{fg-coalgebra-contra-fully-faithful} below.

\subsection{The full-and-faithfulness theorem}
\label{full-and-faithful-subsecn}
 A coalgebra~$\C$ is called \emph{conilpotent} if it has a unique
irreducible (say, left) comodule, whose dimension over~$k$ is
equal to~$1$.
 Alternatively, a (coassociative) coalgebra without counit $\D$ is
called conilpotent if for every element $x\in\D$ there exists
an integer $m\ge0$ such that $x$ is annihilated by the iterated
comultiplication map $\D\rarrow \D^{\ot m+1}$.
 A \emph{coaugmented} coalgebra $\C$ is a (coassociative and
counital) coalgebra endowed with a morphism of coalgebras
(the coaugmentation) $\gamma\:k\rarrow\C$.
 A coaugmented coalgebra $\C$ is called conilpotent if the coalgebra
without counit $\C_+=\C/\gamma(k)$ is conilpotent~\cite{PV}.
 This definition is equivalent to the previous
one~\cite[Section~9.1]{Swe}.
 A conilpotent (counital) coalgebra $\C$ always has a unique
coaugmentation.

 The \emph{cohomology algebra} $H^*(\C)$ of a coaugmented coalgebra
$\C$ can be defined as the Yoneda Ext-algebra $\Ext^*_\C(k,k)$
computed in the abelian category of left $\C$\+comodules $\C\comodl$.
 More explicitly, one has, in particular, $H^1(\C)=\ker(\C_+\to
\C_+\ot_k\C_+)$; this vector space can be interpreted as
the \emph{space of cogenerators} of the conilpotent coalgebra~$\C$.
 We refer to~\cite{PV}, \cite{Pbogom}, and~\cite[Section~5]{Pqf}
for further details.

\begin{thm} \label{fg-coalgebra-contra-fully-faithful}
 Let\/ $\C$ be a conilpotent coalgebra such that the $k$\+vector space
$H^1(\C)$ is finite-dimensional.
 Then the forgetful functor\/ $\C\contra\rarrow\C\spcheck\modl$
is fully faithful.

 Moreover, for any dense subring $R$ in the topological ring\/
$\C\spcheck$, the forgetful functor\/ $\C\contra\rarrow R\modl$
is fully faithful.
\end{thm}

\begin{proof}
 Given two left $\C$\+contramodules $\P$ and $\Q$ and a left
$R$\+module morphism $f\:\P\rarrow\Q$, we have to show that $f$~is
a $\C$\+contramodule morphism.
 Composing $f$ with the contraaction morphism $\Hom_k(\C,\P)
\rarrow\P$ and replacing $\P$ with $\Hom_k(\C,\P)$, we can assume
that $\P$ is a free left $\C$\+contramodule, $\P=\Hom_k(\C,V)$
for some $k$\+vector space~$V$.
 Then the composition
\begin{equation} \label{restrict-to-generators}
 V\lrarrow\Hom_k(\C,V)\overset{f}\lrarrow\Q
\end{equation}
of the $k$\+linear map $V\rarrow\Hom_k(\C,V)$ induced by the counit
$\varepsilon\:\C\rarrow k$ with the $R$\+module morphism
$f\:\P=\Hom_k(\C,V)\rarrow\Q$ extends uniquely to a left
$\C$\+contramodule morphism $f'\:\Hom_k(\C,V)\rarrow\Q$.
 Replacing $f$ with $f-f'$, we can assume that
the composition~\eqref{restrict-to-generators} vanishes.
 Then we have to show that $f=0$.
 Furthermore, replacing $\Q$ with its $\C$\+subcontramodule
generated by $f(\P)$, we can assume that $\Q$ has no
proper subcontramodules containing~$f(\P)$.

 For any left $\C$\+contramodule $\fL$, denote by $\fL^+\subset\fL$
the image of the composition of maps $\Hom_k(\C_+,\fL)\rarrow
\Hom_k(\C,\fL)\rarrow\fL$.
 Then $\fL^+$ is a $\C$\+subcontramodule in $\fL$ and $\fL/\fL^+$ is
the maximal quotient contramodule of $\fL$ with a trivial
$\C$\+contramodule structure (the latter notion being defined in terms
of the coaugmentation of~$\C$).
 According to the contramodule Nakayama
lemma~\cite[Lemma~A.2.1]{Psemi} (cf.~\cite[Lemma~2.1]{Prev},
\cite[Lemma~1.3.1]{Pweak}, \cite[Lemma~D.1.2]{Pcosh}),
\,$\fL/\fL^+\ne0$ whenever $\fL\ne0$.

 Let us first discuss the case when $R=\C\spcheck$.
 Let $\C_+\rarrow U$ be a $k$\+linear map from $\C_+$ to
a finite-dimensional $k$\+vector space $U$ such that the composition
$H^1(\C)\rarrow\C_+\rarrow U$ is injective.
 Since $\C$ is conilpotent, the composition $\C_+\rarrow\C\ot_k\C_+
\rarrow\C\ot_k U$ is then also an injective left $\C$\+comodule
morphism~\cite[Lemma~5.1]{Pqf}.
 Hence the induced map $\Hom_k(\C\ot_k U,\>\fL)\rarrow
\Hom_k(\C_+,\fL)$ is surjective, and it follows from
the contraassociativity axiom that $\fL^+\subset\fL$ is the image
of the composition of the contraaction maps
$\Hom_k(U,\Hom_k(\C,\fL))\rarrow\Hom_k(U,\fL)\rarrow\fL$.
 Thus $\fL^+$ is the image of the contraaction map
$\Hom_k(U,\fL)\rarrow\fL$.

 The vector space $U$ being finite-dimensional, we have
$\Hom_k(U,\fL)\simeq U\spcheck\ot_k\fL$.
 The map $U\spcheck\ot_k\fL\rarrow\fL$ can be obtained as
the composition $U\spcheck\ot_k\fL\rarrow\C\spcheck\ot_k\fL
\rarrow\fL$ of the map induced by the natural linear map
$U\spcheck\rarrow\C\spcheck$ and the $\C\spcheck$\+action map.
 It follows that any left $\C\spcheck$\+module morphism $\P\rarrow\Q$
between two left $\C$\+contramodules $\P$ and $\Q$ takes
$\P^+\subset\P$ into $\Q^+\subset\Q$.
 We have shown that $f(\P^+)\subset\Q^+$.

 Now we have $\P=\Hom_k(\C,V)$, hence $\P/\P^+=V$.
 By our assumption, the induced map $\P/\P^+\rarrow\Q/\Q^+$ vanishes.
 Therefore, $f(\P)\subset\Q^+$.
 By another assumption of ours, it follows that $\Q^+=\Q$.
 Applying the contramodule Nakayama lemma, we can conclude that $\Q=0$.

 More generally, let $R$ be a dense subring in the pseudo-compact
algebra~$\C\spcheck$.
 The discrete $k$\+vector space $H^1(\C)$ is dual to the pseudo-compact
quotient $k$\+vector space of the augmentation ideal
$\fI=\C_+\spcheck\subset\C\spcheck$ by the closure $\fI_2$ of the ideal
$\fI^2=(\C_+\spcheck)^2$ in $\C\spcheck$, that is
$H^1(\C)\spcheck\simeq\fI/\fI_2$.
 Since $H^1(\C)$ is finite-dimensional, it follows that $\fI/\fI_2$
is also finite-dimensional and discrete, and $\fI_2$ is open in~$\fI$.
 (We will see below that $\fI^2$ is in fact an open ideal in
$\C\spcheck$, so $\fI_2=\fI^2$.)
 Since $R$ is dense in $\C\spcheck$ (and consequently $R\cap\fI$ is
dense in~$\fI$, as the unit element of $\C\spcheck$ belongs to $R$ by
the definition of a subring), there exists a finite-dimensional
subspace $U\spcheck\subset R\cap\fI$ such that the projection map
$U\spcheck\rarrow\fI/\fI_2$ is surjective.

 Dualizing the composition of maps $U\spcheck\rarrow R\cap\fI\rarrow
\fI=\C_+\spcheck$, we obtain a $k$\+linear map $\C_+\rarrow U$ from
$\C_+$ to a finite-dimensional $k$\+vector space $U$ such that
the composition $H^1(\C)\rarrow\C_+\rarrow U$ is injective.
 Since $\C$ is conilpotent, the composition $\C_+\rarrow\C\ot_k\C_+
\rarrow\C\ot_k U$ is then also an injective left $\C$\+comodule
morphism~\cite[Lemma~5.1]{Pqf}.
 Arguing as above, we see that for any left $\C$\+contramodule $\fL$
the subcontramodule $\fL^+\subset\fL$ is the image of the composition
$U\spcheck\ot_k\fL\rarrow\C\spcheck\ot_k\fL\rarrow\fL$ of the map
induced by the embedding $U\spcheck\rarrow\C\spcheck$ and
the $\C\spcheck$\+action map.

 The image of the map $U\spcheck\rarrow\C\spcheck$ lies in
$R\subset\C\spcheck$.
 So the map $U\spcheck\ot_k\fL\rarrow\fL$ can be also obtained as
the composition $U\spcheck\ot_k\fL\rarrow R\ot_k\fL\rarrow\fL$ of
the map induced by the embedding $U\spcheck\rarrow R$ and
the left $R$\+action map.
 It follows that any left $R$\+module morphism $\P\rarrow\Q$
between two left $\C$\+contramodules $\P$ and $\Q$ takes
$\P^+\subset\P$ into $\Q^+\subset\Q$.
 Once again, we have shown that $f(\P^+)\subset\Q^+$, and the argument
finishes as above.

 Finally, let us explain why $\fI_2=\fI^2=\fI U\spcheck\subset
\C\spcheck$ (notice that, the vector space $U\spcheck$ being
finite-dimensional, the ideal $\fI U\spcheck\subset\C\spcheck$
is clearly closed as the image of a continious linear map between
pseudo-compact vector spaces).
 Equivalently, this means that the composition $\C_+\rarrow\C_+\ot_k
\C_+\rarrow\C_+\ot_k U$ has the same kernel as the comultiplication
map $\C_+\rarrow\C_+\ot_k\C_+$.
 The kernel of the composition $\C_+\rarrow\C_+\ot_k\C_+\rarrow
\C_+\ot_k\C\ot_kU$ is equal to the kernel of the map $\C_+\rarrow
\C_+\ot_k\C_+$, as the map $\C_+\rarrow\C\ot_k U$ is injective
as we have seen.
 The composition $\C_+\rarrow\C_+\ot_k U\rarrow\C_+\ot_k\C\ot_kU$
provides the same map, so the kernel of $\C_+\rarrow\C_+\ot_k U$
is contained in that of $\C_+\rarrow\C_+\ot_k\C_+$.
 It follows that $\fI^n=\fI(U\spcheck)^{n-1}$ is a closed ideal
in $\C\spcheck$ for every $n\ge2$.
\end{proof}

\subsection{Cotensor product and contratensor product}
\label{cotensor-contratensor-subsecn}
 The \emph{cotensor product} $\N\oc_\C\M$ of a right $\C$\+comodule $\N$
and a left $\C$\+comodule $\M$ is a $k$\+vector space constructed as
the kernel of the difference of the two maps
$$
 \N\ot_k\M\rightrightarrows \N\ot_k\C\ot_k\M,
$$
one of which is induced by the right coaction map $\N\rarrow\N\ot_k\C$
and the other one by the left coaction map $\M\rarrow\C\ot_k\M$.
 The equivalences of categories~\eqref{linear-duality}
and~\eqref{linear-duality-over-C} transform the functor of cotensor
product of $\C$\+comodules into the functor of completed tensor product
of pseudo-compact $\C\spcheck$\+modules.

 For any pseudo-compact right $\C\spcheck$\+module $\M\spcheck$ and
any pseudo-compact left $\C\spcheck$\+module $\N\spcheck$,
the \emph{completed tensor product}
$\M\spcheck\wot_{\C\spcheck}\N\spcheck$ is a pseudo-compact $k$\+vector
space constructed as the cokernel of the difference of two maps
$$
 \M\spcheck\wot_k\C\spcheck\wot_k\N\spcheck\rightrightarrows
 \M\spcheck\wot_k\N\spcheck,
$$
one of which is induced by the right action map
$\M\spcheck\wot_k\C\spcheck\rarrow\M\spcheck$ and the other one by
the left action map $\C\spcheck\wot_k\N\spcheck\rarrow\N\spcheck$.
 For any right $\C$\+comodule $\N$ and a left $\C$\+comodule $\M$,
there is a natural isomorphism of pseudo-compact vector spaces
$$
 (\N\oc_\C\M)\spcheck\simeq\M\spcheck\wot_{\C\spcheck}\N\spcheck.
$$

 The functor of cotensor product of $\C$\+comodules is left exact.
 The functor of completed tensor product of pseudo-compact
$\C\spcheck$\+modules is right exact.

 The \emph{contratensor product} $\N\ocn_\C\P$ of a right
$\C$\+comodule $\N$ and a left $\C$\+contramodule $\P$ is
a $k$\+vector space constructed as the cokernel of the difference
of two natural maps
\begin{equation} \label{C-contratensor-pair-of-maps}
 \N\ot_k\Hom_k(\C,\P)\.\rightrightarrows\N\ot_k\P.
\end{equation}
 Here the first map is simply induced by the left contraaction map
$\pi_\P\:\Hom_k(\C,\P)\allowbreak\rarrow\P$, while the second map is
the composition $\N\ot_k\Hom_k(\C,\P)\rarrow
\N\ot_k\C\ot_k\Hom_k(\C,\P)\rarrow\N\ot_k\P$ of the map induced by
the right coaction map $\nu_\N\:\N\rarrow\N\ot_k\C$ and the map
induced by the evaluation map $\C\ot_k\Hom_k(\C,\P)\rarrow\P$.

 For any right $\C$\+comodule $\N$, left $\C$\+contramodule $\P$,
and $k$\+vector space $V$ there is a natural isomorphism of
$k$\+vector spaces
$$
 \Hom_k(\N\ocn_\C\P,\>V)\simeq\Hom^\C(\P,\Hom_k(\N,V)),
$$
where we denote by $\Hom^\C$ the space of morphisms in the category
of left $\C$\+contra\-modules $\C\contra$ \cite[Section~3.1]{Prev}.

 The functor of contratensor product of $\C$\+comodules and
$\C$\+contramodules is right exact.
 Let us emphasize that the functor of completed tensor product
is dual to the functor of cotensor product, as mentioned above.
 There is \emph{no} simple connection between the completed tensor
product and the contratensor product.

 There is an important, natural pair of adjoint functors between
the categories of left $\C$\+comodules and left $\C$\+contramodules,
\begin{equation} \label{phi-psi-C-adjunction}
 \Psi_\C\:\C\comodl\,\rightleftarrows\,\C\contra:\!\Phi_\C,
\end{equation}
defined by the rules $\Phi_\C(\P)=\C\ocn_\C\P$ and $\Psi_\C(\M)=
\Hom_\C(\C,\M)$, where $\P$ is a left $\C$\+contramodule, $\M$ is
a left $\C$\+comodule, and $\Hom_\C$ denotes the space of morphisms
in the category of left $\C$\+comodules $\C\comodl$
\cite[Section~3.1]{Prev}.

 Here the construction of the vector space $\C\ocn_\C\P$ uses
the right $\C$\+comodule structure on the coalgebra $\C$, while
the left $\C$\+comodule structure on $\C$ induces a left $\C$\+comodule
structure on the contratensor product $\C\ocn_\C\P$.
 Similarly, the construction of the vector space $\Hom_\C(\C,\M)$ uses
the left $\C$\+comodule structure on $\C$, while the right
$\C$\+comodule structure on $\C$ induces a left $\C$\+contramodule
structure on $\Hom_\C(\C,\M)$
(see~\eqref{right-co-left-contra} for the construction).
 The pair of adjoint functors $\Psi_\C$ and $\Phi_\C$ restricts to
an equivalence between the full subcategories of injective left
$\C$\+comodules and projective left $\C$\+contramodules,
\begin{equation} \label{underived-co-contra}
 \C\comodl^\inj\simeq\C\contra^\proj.
\end{equation}
 The equivalence of categories~\eqref{underived-co-contra} takes
the cofree left $\C$\+comodule $\C\ot_kV$ to the free left
$\C$\+contramodule $\Hom_k(\C,V)$ for any $k$\+vector space~$V$
\cite[Section~1.2]{Prev}, \cite[Section~0.2.6]{Psemi}.

 Returning to the discussion of the comparison of contramodules with
pseudo-compact modules in Section~\ref{comodules-contramodules-subsecn},
one has to say that, generally speaking, of course,
the $\C$\+contramodule structure on $\Psi_\C(\M)$ does \emph{not}
underlie any pseudo-compact topology.
 In fact, for the trivial coalgebra $\C=k$, one has $\C\comodl
=k\vect=\C\contra$, and both $\Phi_\C$ and $\Psi_\C$ are simply
the identity functors $k\vect\rarrow k\vect$.

 However, there is one notable particular case when a natural
structure of a pseudo-compact left $\C\spcheck$\+module can be defined
on the vector space $\Psi_\C(\M)$, and the left $\C$\+contramodule
structure of $\Psi_\C(\M)$ comes from this pseudo-compact module
structure.
 This is the case for \emph{finitely cogenerated} injective left
$\C$\+comodules $\M$ (i.~e., the direct summands of cofree left
$\C$\+comodules $\C\ot_kV$ with finite-dimensional vector spaces of
cogenerators~$V$) and, more generally, for \emph{quasi-finitely
cogenerated} left $\C$\+comodules~$\M$.
 We refer to the discussion in Section~\ref{coalgebra-duality} of
the introduction, and to the details in
Sections~\ref{quasi-finite-subsecn}\+-\ref{qf-co-presented-subsecn}
below.

 For any coalgebra $\C$, a right $\C$\+comodule $\N$, and a left
$\C$\+contramodule $\P$ there is a natural surjective map from
the tensor product to the contratensor product
\begin{equation} \label{tensor-contratensor}
 \N\ot_{\C\spcheck}\P\lrarrow\N\ocn_\C\P.
\end{equation}
 Here the tensor product in the left-hand side is the abstract
(uncompleted) tensor product of abstract (nontopological)
$\C\spcheck$\+modules whose $\C\spcheck$\+module structures arise from
the comodule and contramodule structures via the forgetful
functors~\eqref{right-co-forgetful} and~\eqref{contra-forgetful}.
 In order to construct the surjective map~\eqref{tensor-contratensor},
it suffices to observe that there is a pair of commutative diagrams
$$
\begin{diagram}
\node{\N\ot_k\C\spcheck\ot_k\P}\arrow{e,t}{\rightrightarrows}
\arrow{s,V} \node{\N\ot_k\P} \arrow{s,=} \\
\node{\N\ot_k\Hom_k(\C,\P)} \arrow{e,t}{\rightrightarrows}
\node{\N\ot_k\P}
\end{diagram}
$$
where the tensor product $\N\ot_{\C\spcheck}\P$ is the cokernel of
the difference of the two maps in the upper horizontal line and
the contratensor product $\N\ocn_\C\P$ is the cokernel of
the difference of the two arrows in the lower horizontal line.

\begin{cor} \label{fg-coalgebra-contratensor-iso}
 Let\/ $\C$ be a conilpotent coalgebra such that the $k$\+vector space
$H^1(\C)$ is finite-dimensional.
 Then for any right\/ $\C$\+comodule\/ $\N$ and any left\/
$\C$\+contramodule\/ $\P$ the natural map\/ $\N\ot_{\C\spcheck}\P
\rarrow\N\ocn_\C\P$ is an isomorphism.

 Moreover, for any dense subring $R$ in the topological ring\/
$\C\spcheck$, the natural map $\N\ot_R\P\rarrow\N\ocn_\C\P$ is
an isomorphism.
\end{cor}

\begin{proof}
 Applying the functor $\Hom_k({-},V)$ to the map in question, we
obtain the map
\begin{multline*}
 \Hom^\C(\P,\.\Hom_k(\N,V))\.\simeq\.\Hom_k(\N\ocn_\C\P,\,V) \\
 \lrarrow \Hom_k(\N\ot_R\P,\,V) \.\simeq\.
 \Hom_R(\P,\.\Hom_k(\N,V)),
\end{multline*}
which is an isomorphism by
Theorem~\ref{fg-coalgebra-contra-fully-faithful}.
\end{proof}

\subsection{Cohomomorphisms} \label{cohom-subsecn}
 Yet another relevant tensor operation involving comodules and
contramodules is the vector space of \emph{cohomomorphisms}, or Cohom.
 It can be thought of as answering the following question.

 In the previous section, we have seen how to construct the dual vector
space to the cotensor product $(\N\oc_\C\M)\spcheck$ in terms of
the pseudo-compact $\C\spcheck$\+modules $\N\spcheck$ and $\M\spcheck$.
 In this section we will see how to construct the vector space
$(\N\oc_\C\M)\spcheck$ in terms of the left $\C$\+comodule $\M$ and
the left $\C$\+contramodule $\N\spcheck$.

 One difference is that we will only obtain $(\N\oc_\C\M)\spcheck$ as
an abstract vector space, without the pseudo-compact topology on it.
 On the other hand, the same construction will also produce the vector
space $\Hom_k(\N\oc_\C\M,\>V)$ from the left $\C$\+comodule $\M$ and
the left $\C$\+contramodule $\Hom_k(\N,V)$ for any vector space~$V$.

 Let $\M$ be a left $\C$\+comodule and $\P$ be a left $\C$\+contramodule.
 The $k$\+vector space of \emph{cohomomorphisms} $\Cohom_\C(\M,\P)$ is
constructed as the cokernel of the difference of the two maps
$$
 \Hom_k(\C\ot_k\M,\>\P)=\Hom_k(\M,\Hom_k(\C,\P))\.\rightrightarrows\.
 \Hom_k(\M,\P),
$$
one of which is induced by the left coaction map $\M\rarrow\C\ot_k\M$
and the other one by the left contraaction map $\Hom_k(\C,\P)\rarrow\P$.

 For any right $\C$\+comodule $\N$, left $\C$\+comodule $\M$, and
$k$\+vector space $V$ there is a natural isomorphism of $k$\+vector
spaces~\cite[Section~2.6]{Prev}
\begin{equation} \label{cotensor-cohom}
 \Hom_k(\N\oc_\C\M,\>V)\simeq\Cohom_k(\M,\Hom_k(\N,V)).
\end{equation}

 Let us emphasize that $\Cohom_\C$ is \emph{not} the functor Hom in any
category, if only because its two arguments are objects from two
different categories: $\M$ is a left $\C$\+comodule, and $\P$ is a left
$\C$\+contramodule.
 Furthermore, the functor $\Cohom_\C$ is right exact, while the functor
Hom in any abelian category is left exact.

\subsection{Injective and projective objects}
 Our next aim is to show that, under certain assumptions, all injective
$\C$\+comodules are injective $\C\spcheck$\+modules and all projective
$\C$\+contramodules are flat $\C\spcheck$\+modules.
 It will be convenient to increase the generality slightly and
use the language centered around the topological ring
$\fR=\C\spcheck$ rather than the coalgebra~$\C$.

 Let $\fR$ be an associative ring and $\fI\subset\fR$ be a two-sided
ideal such that $\fR$ is separated and complete in the $\fI$\+adic
topology, that is $\fR=\varprojlim_n\fR/\fI^n$.
 Consider the associated graded ring $\gr_\fI\fR =
\bigoplus_{n=0}^\infty\fI^n/\fI^{n+1}$ and the ideal
$\gr_\fI\fI = \bigoplus_{n=1}^\infty\fI^n/\fI^{n+1}\subset\gr_\fI\fR$.
 The following assertion is a version of the Artin--Rees lemma.

\begin{lem} \label{artin-rees}
 Assume that the ring\/ $\gr_\fI\fR$ is right Noetherian and
the ideal\/ $\gr_\fI\fI\subset\gr_\fI\fR$ is generated by
(a finite set of) central elements in\/ $\gr_\fI\fR$.
 Let $M$ be a finitely generated right\/ $\fR$\+module with
an\/ $\fR$\+submodule $N\subset M$.
 Then there exists an integer $m\ge0$ such that $N\cap M\fI^{n+m} =
(N\cap M\fI^m)\fI^n$ for all $n\ge0$.
\end{lem}

\begin{proof}
 As in~\cite[Lemma~13.2]{GW}, it suffices to show that the Rees
ring $\fR(\fI)=\bigoplus_{n=0}^\infty\fI^n$ is right Noetherian
(as a graded ring).
 Indeed, $M(\fI)=\bigoplus_{n=0}^\infty M\fI^n$ is a finitely generated
graded $\fR(\fI)$\+module, and $\bigoplus_{n=0}^\infty N\cap M\fI^n$ is
a graded $\fR(\fI)$\+submodule in $M(\fI)$.
 It remains to choose $m\ge0$ such that this submodule is generated
by some (finite) set of elements of the degree~$\le m$.

 Consider the ideal $\fJ=\fI\oplus\fI\oplus\fI^2\oplus\fI^3\subset
\fR(\fI)$, so that the quotient ring is $\fR(\fI)/\fJ=\fR/\fI$.
 The ring $\fR(\fI)$ is separated and complete (as a graded ring)
in the $\fJ$\+adic topology, that is $\fR(\fI)=\varprojlim_n
\fR(\fI)/\fJ^n$ in the category of graded abelian groups.
 The associated graded ring $\gr_\fJ\fR(\fI)=\bigoplus_{n=0}^\infty
\fJ^n/\fJ^{n+1}$ is isomorphic to the Rees ring of the graded ring
$\gr_\fI\fR$ (endowed with the decreasing filtration associated
with the grading), $\gr_\fJ\fR(\fI)\simeq\gr_\fI\fR(\gr_\fI\fI)$.

 Hence the (bi)graded ring $\gr_\fJ\fR(\fI)$ is a quotient ring of
the polynomial ring in a finite number of variables over
the right Noetherian ring $\gr_\fI\fR$, so it is right Noetherian,
as in~\cite[Theorems~1.9 and~13.3]{GW}.
 It remains to deduce the assertion that the graded ring $\fR(\fI)$
is right Noetherian.
 This is a standard argument: given a homogeneous right ideal
$\fK\subset\fR(\fI)$, one chooses a finite set of bihomogeneous
generators of the right ideal $\gr_\fJ\fK\subset\gr_\fJ\fR(\fI)$
and lifts them to homogeneous elements in $\fK$, obtaining
a finite set of generators of the right ideal $\fK\subset\fR(\fI)$.
\end{proof}

 Denote by $\discrR\fR\subset\modr\fR$ the full subcategory of
discrete right $\fR$\+modules (i.~e., right $\fR$\+modules $\N$ such
that the annihilator of every element $x\in\N$ is a open right ideal
in~$\fR$).
 The category $\discrR\fR$ is a Grothendieck abelian category, so
it has enough injective objects.
 
 The definition of a left $\fR$\+contramodule can be found
in~\cite[Section~1.2]{Pweak}, \cite[Section~2.1]{Prev}; we do not
repeat it here, as all we need is the construction of \emph{free}
left $\fR$\+contramodules.
 Given a set $X$ and a ring $R$, denote by $R[X]$ the free left
$R$\+module generated by the set~$X$.
 Then the free left $\fR$\+contramodule generated by $X$ is
$$
 \fR[[X]]=\varprojlim\nolimits_n(\fR/\fI^n)[X].
$$
 In other words, $\fR[[X]]$ is the set of all maps of sets
$f\:X\rarrow\fR$ converging to zero in the topology of $\fR$,
which means that for every $n\ge1$ one has $f(x)\in\fI^n$ for
all but a finite subset of elements $x\in X$.

 When $\C$ is a conilpotent coalgebra with finite-dimensional space
of cogenerators $H^1(\C)$, and $\fR=\C\spcheck$ is the dual algebra,
the pseudo-compact topology of $\C\spcheck$ coincides with the adic
topology for the augmentation ideal $\fI=\C_+\spcheck\subset\C\spcheck$.
 Hence the above definition of a free $\fR$\+contramodule agrees
with the definition of a free $\C$\+contramodule
(cf.~\cite[Section~1.10]{Pweak}, \cite[Section~2.3]{Prev}).

\begin{prop} \label{noetherian-injective-flat-prop}
 Assume that the ring\/ $\gr_\fI\fR$ is right Noetherian and
the ideal\/ $\gr_\fI\fI\subset\gr_\fI\fR$ is generated by
(a finite set of) central elements in\/ $\gr_\fI\fR$.
 Then \par
\textup{(a)} every injective object of\/ $\discrR\fR$ is also
an injective object of\/ $\modr\fR$; \par
\textup{(b)} for every set $X$, the left\/ $\fR$\+module\/
$\fR[[X]]$ is flat.
\end{prop}

\begin{proof}
 This is a version of~\cite[Proposition~B.9.1]{Pweak}
and~\cite[Proposition~2.2.2]{Prev}.
 Notice first of all that the ring $\fR$ is right Noetherian
by~\cite[Proposition~7.27]{DSMS} (cf.\ the related argument in
the proof of Lemma~\ref{artin-rees} above).
 Part~(a): let $\J$ be an injective object of $\discrR\fR$.
 In order to show that $\J$ is an injective right $\fR$\+module,
it suffices to check that, for any finitely generated right
$\fR$\+module $M$ with an $\fR$\+submodule $N\subset M$, any
$\fR$\+module morphism $f\:N\rarrow\J$ can be extended to
an $\fR$\+module morphism $M\rarrow\J$.
 Since the $\fR$\+module $N$ is finitely generated and
the $\fR$\+module $\J$ is discrete, there exists $n\ge0$ such that
$f(N\fI^n)=f(N)\fI^n=0$ in~$\J$.
 By Lemma~\ref{artin-rees}, there exists $m\ge0$ such that
$N\cap M\fI^m\subset N\fI^n$.
 Now $N/(N\cap M\fI^m)$ is a submodule in a discrete right
$\fR$\+module $M/M\fI^m$, so the $\fR$\+module morphism
$N/(N\cap M\fI^m)\rarrow N/N\fI^n\rarrow\J$ can be extended to
an $\fR$\+module morphism $M/M\fI^m\rarrow\J$.

 Part~(b): it suffices to show that the functor $M\longmapsto
M\ot_\fR\fR[[X]]$ is exact on the abelian category of finitely
generated right $\fR$\+modules.
 For any such $M$, consider the natural morphism of abelian groups
\begin{equation} \label{flatness-comparison-morphism}
 M\ot_\fR\fR[[X]]\lrarrow\varprojlim\nolimits_n M\ot_\fR\fR/\fI^n[X].
\end{equation}

 For any short exact sequence of finitely generated right
$\fR$\+modules $0\rarrow K\rarrow L\rarrow M\rarrow 0$, there is
a short exact sequence of $\fR/\fI^n$\+modules
$$
 0\lrarrow K/(K\cap L\fI^n)\lrarrow L/L\fI^n\lrarrow M/M\fI^n\lrarrow0,
$$
which remains exact after applying the functor
${-}\ot_{\fR/\fI^n}\fR/\fI^n[X]$.
 By Lemma~\ref{artin-rees}, the projective systems $K/(K\cap L\fI^n)$
and $K/K\fI^n$ are cofinal, so
$$
 \varprojlim\nolimits_n K/(K\cap L\fI^n)\ot_{\fR/\fI^n}\fR/\fI^n[X]
 = \varprojlim\nolimits_n K/K\fI^n\ot_{\fR/\fI^n}\fR/\fI^n[X].
$$
 Since $K/(K\cap L\fI^n)\ot_{\fR/\fI^n}\fR/\fI^n[X]$ is a projective
system of surjective morphisms, passing to the projective limit
produces a short exact sequence
$$
 0\rarrow\varprojlim\nolimits_n K\ot_\fR\fR/\fI^n[X]
 \rarrow\varprojlim\nolimits_n L\ot_\fR\fR/\fI^n[X]
 \rarrow\varprojlim\nolimits_n M\ot_\fR\fR/\fI^n[X]
 \rarrow0.
$$
 Hence the functor in the right-hand side of
the morphism~\eqref{flatness-comparison-morphism} is exact.

 The functor in the left-hand side is right exact, and the morphism
is an isomorphism when $M$ is a finitely generated \emph{free} right
$\fR$\+module.
 It follows that the morphism~\eqref{flatness-comparison-morphism}
is an isomorphism for every finitely generated right $\fR$\+module~$M$.
 Thus the left-hand side of~\eqref{flatness-comparison-morphism}
is also an exact functor of~$M$.
\end{proof}

\subsection{Profinite groups} \label{profinite-groups-subsecn}
 Let $H$ be a profinite group and $k$~be a field.
 An \emph{$H$\+module over~$k$} is a $k$\+vector space in which
$H$ acts by $k$\+linear automorphisms.
 An $H$\+module $\M$ over~$k$ is called \emph{discrete} if
the stabilizer of every element $x\in\M$ is an open subgroup in~$H$.
 The abelian category of $H$\+modules over~$k$ is denoted by $H\modl_k$
and the abelian category of discrete $H$\+modules over~$k$ is denoted
by $H\discr_k\subset H\modl_k$.

 Let $X$ denote a profinite set (or in other words, a compact
totally disconnected topological space).
 For any $k$\+vector space $V$, we denote by $V(X)$ the $k$\+vector
space of all locally constant functions $X\rarrow V$.
 We also denote by $V[[X]]$ the $k$\+vector space of all finitely
additive $V$\+valued measures defined on all the open-closed subsets
in~$X$.
 Then there are natural isomorphisms
$$
 V(X)\simeq V\ot_kk(X), \qquad V[[X]]\.\simeq\.\Hom_k(k(X),V),
$$
and
$$
 V(H)=\varinjlim_{U\subset H} V(H/U), \qquad
 V[[H]]=\varprojlim_{U\subset H} V[H/U],
$$
where the limits are taken over all the (if one wishes, normal) open
subgroups $U\subset H$, and $V[H/U]=V(H/U)$ denotes the vector space
of all functions $H/U\rarrow V$.
 For any two profinite sets $X$ and $Y$, there are natural
isomorphisms
$$
 k(X\times Y)\simeq k(X)\ot_k k(Y), \qquad
 V[[X\times Y]]\simeq V[[X]][[Y]].
$$

 An \emph{$H$\+contramodule over~$k$} is a $k$\+vector space $\P$
endowed with a $k$\+linear \emph{$H$\+contra\-action} map
$\pi_\P\:\P[[H]] \rarrow\P$ satisfying the following two axioms.
 Firstly, the two maps $\P[[H]][[H]]\simeq\P[[H\times H]]
\rightrightarrows\P[[H]]$, one of them provided by the pushforward
of measures with respect to the multiplication map $H\times H
\rarrow H$ and the other one induced by the contraaction map~$\pi_\P$,
should have equal compositions with the contraaction map~$\pi_\P$,
$$
 \P[[H]][[H]]\.\simeq\.\P[[H\times H]]\.\rightrightarrows\.\P[[H]]
 \rarrow\P.
$$
 Secondly, the composition of the map $\P\rarrow\P[[H]]$ assigning
to a vector $x\in\P$ the point measure concentrated at the unit
element $e\in H$ with the value~$x$ and the contraaction map~$\pi_\P$
should be equal to the identity map $\operatorname{id}_\P$,
$$
 \P\lrarrow\P[[H]]\lrarrow\P.
$$
 We refer to~\cite[Section~1.8]{Prev} for a discussion of
the intuition behind this concept.

 $H$\+contramodules over~$k$ form an abelian category $H\contra_k$.
 Given an $H$\+contra\-module $\P$ over~$k$, a vector $x\in\P$,
and an element $h\in H$, one can consider the point measure at
$h^{-1}\in H$ with the value $x\in\P$.
 Applying the contraaction map~$\pi_\P$ to this measure, one obtains
an element denoted by $h(x)\in\P$.
 This construction defines the underlying $H$\+module structure on
an $H$\+contramodule $\P$, providing an exact and faithful forgetful
functor $H\contra_k\rarrow H\modl_k$.

 The $k$\+vector space $\C=k(H)$ has a natural coalgebra structure
with the comultiplication map $\C\rarrow\C\ot_k\C$ provided by
the pullback of functions with respect to the multiplication map
$H\times H\rarrow H$ and the counit map $\C\rarrow k$ similarly
induced by the unit map $\{*\}\rarrow H$.
 The dual algebra $\C\spcheck=k[[H]]$ is the projective limit of
the group algebras $\varprojlim_{U\subset H} k[H/U]$, as above.

 The datum of a discrete action of $H$ on a $k$\+vector space $\M$
is equivalent to that of a (left or right) $\C$\+comodule structure
on~$\M$.
 Analogously, the datum of an $H$\+contramodule structure on
a $k$\+vector space $\P$ is equivalent to that of a (left or right)
$\C$\+contramodule structure.
 So there are natural isomorphisms of categories $H\discr_k=
k(H)\comodl=\comodr k(H)$ and $H\contra_k=k(H)\contra$ (where
the left and right $k(H)$\+comodules are identified by means of
the involutive anti-automorphism of the coalgebra $k(H)$ induced
by the inverse element map $H\rarrow H$).
 
 As usually, we denote by $k[H]$ the group algebra of the group $H$
(viewed as an abstract group with the topology forgotten).
 So we have natural isomorphisms of categories $H\modl_k\simeq k[H]
\modl\simeq\modr k[H]$.
 There is a natural injective homomorphism of $k$\+algebras
$k[H]\rarrow k[[H]]=\C\spcheck$ inducing the embedding functor
$H\discr_k\rarrow H\modl_k$ and the forgetful functor
$H\contra_k\rarrow H\modl_k$.

 The profinite group cohomology algebra $H^*(H,k)$ is naturally
isomorphic to the cohomology algebra $H^*(\C)$ of the coalgebra
$\C=k(H)$ \cite[Section~4.2]{Pbogom}.
 One can show this by interpreting both the cohomology algebras
in question as the Yoneda Ext-algebras in the related categories
(of $\C$\+comodules $=$ discrete $H$\+modules over~$k$), or simply
noticing that the complex of continuous (i.~e., locally constant)
cochains computing the profinite group cohomology coincides with
the cobar-complex~\cite[Section~1.1]{PV}, \cite[Section~2.1]{Pbogom}
computing the cohomology of~$\C$.

 Now let us assume that $H$ is a pro-$p$-group and $k$ is a field
of characteristic~$p$.
 Then the minimal number of generators of the profinite group $H$
can be computed as the dimension of the $k$\+vector space
$H^1(H,k)=H^1(\C)$.

\begin{cor} \label{fg-pro-p-contra-fully-faithful}
 Let $H$ be a finitely generated pro-$p$-group and $k$ be a field
of characteristic~$p$.
 Then the forgetful functor\/ $H\contra_k\rarrow H\modl_k$ is
fully faithful.

 Moreover, for any dense subgroup $H'\subset H$, the forgetful functor
$H\contra_k\rarrow H'\modl_k$ is fully faithful.
\end{cor}

\begin{proof}
 Follows from Theorem~\ref{fg-coalgebra-contra-fully-faithful}, as
$k[H']$ is a dense subring in $k[[H]]$.
\end{proof}

\begin{cor} \label{fg-pro-p-contratensor-iso}
 Let $H$ be a finitely generated pro-$p$-group and $k$ be a field
of characteristic~$p$.
 Then for any discrete $H$\+module\/ $\N$ and any $H$\+contramodule\/
$\P$ over~$k$ the natural map\/ $\N\ot_{k[H]}\P\rarrow\N\ocn_{k(H)}\P$
is an isomorphism.

 Moreover, for any dense subgroup $H'\subset H$, the natural map
$\N\ot_{k[H']}\P\rarrow\N\ocn_{k(H)}\nobreak\P$ is an isomorphism.
\end{cor}

\begin{proof}
 Follows from Corollary~\ref{fg-coalgebra-contratensor-iso}.
\end{proof}

 For the definition of a \emph{uniform pro-$p$-group}, we refer
to~\cite[Definition~4.1]{DSMS}.

\begin{cor} \label{inj-inj-proj-flat-cor}
 Let $H$ be a uniform pro-$p$-group and $k$ be a field of
characteristic~$p$.
 Then \par
\textup{(a)} every injective object of $H\discr_k$ is also
an injective object of $k[[H]]\modl$; \par
\textup{(b)} every projective object of $H\contra_k$ is
a flat $k[[H]]$\+module.
\end{cor}

\begin{proof}
 Set $\fR=k[[H]]$, and let $\fI\subset\fR$ be the augmentation ideal.
 According to~\cite[Theorem~7.24]{DSMS}, the graded ring
$\gr_\fI\fR$ is a commutative polynomial ring in a finite number of
variables over~$k$.
 So Proposition~\ref{noetherian-injective-flat-prop} applies.
\end{proof}

\begin{rem}
 For the benefit of a reader having no prior experience with
contramodules, let us revisit once again the concepts that we have
defined and discussed so far.
 Returning to the discussion in Sections~\ref{pseudo-and-contra-subsecn}
and~\ref{comodules-contramodules-subsecn},
for any coalgebra $\C$ over a field~$k$, the forgetful functor
$\C\spcheck\pscomp\rarrow\C\spcheck\modl$ from the category of
pseudo-compact left $\C\spcheck$\+modules to the category of
abstract left $\C\spcheck$\+modules factorizes naturally through
the category of left $\C$\+contramodules $\C\contra$.
 In particular, for any profinite group $H$ the forgetful functor
$H\pscomp_k\rarrow H\modl_k$ from the category of pseudo-compact
$H$\+modules to the category of abstract $H$\+modules over~$k$
factorizes naturally through the category of $H$\+contramodules
$H\contra_k$,
$$
 H\pscomp_k\lrarrow H\contra_k\lrarrow H\modl_k.
$$

 What is the difference between pseudo-compact modules and
contramodules?
 This question is best answered by specializing to the case of
the trivial group $H=\{e\}$ (corresponding to the trivial coalgebra
$\C=k$).
 In this case, the category of pseudo-compact $H$\+modules is just
the category of pseudo-compact $k$\+vector spaces, while
the category of $H$\+contramodules is the category of discrete
$k$\+vector spaces.
 So the functor $H\pscomp_k\rarrow H\modl_k$ forgets the topology
of a pseudo-compact vector space, leaving only the discrete
vector space structure; while the functor $H\contra_k\rarrow
H\modl_k$ is an isomorphism of categories in this case.

 What is the difference between contramodules and abstract modules?
 As it was mentioned above, for any coalgebra $\C$ over~$k$,
a $\C$\+contramodule structure on a $k$\+vector space $\P$ is defined
by a $k$\+linear map $\Hom_k(\C,\P)\rarrow\P$, while
a $\C\spcheck$\+module structure is given by a $k$\+linear map
$\C\spcheck\ot_k\P\rarrow\P$.
 The forgetful functor $\C\contra\rarrow\C\spcheck\modl$ takes
the restriction of the contraaction map $\Hom_k(\C,\P)\rarrow\P$ to
the subspace of all $k$\+linear maps of finite rank
$\C\spcheck\ot_k\P\subset\Hom_k(\C,\P)$.
 In particular, for any profinite group~$H$, an $H$\+contramodule
structure on a $k$\+vector space $\P$ is defined by a $k$\+linear map
$\P[[H]]\rarrow\P$ from the space $\P[[H]]$ of finitely additive
$\P$\+valued measures defined on open-closed subsets of~$H$.
 The forgetful functor $H\contra_k\rarrow H\modl_k$ takes
the restriction of the contraaction map $\P[[H]]\rarrow\P$ to
the subspace $\P[H]\subset\P[[H]]$ of all measures supported in
a finite set of points in~$H$.

 What is the contratensor product of comodules and contramodules?
 Once again, considering the case of $H=\{e\}$ and $\C=k$ is
instuctive.
 In this case, both the discrete $H$\+modules and
the $H$\+contramodules over~$k$ are just the discrete $k$\+vector
spaces, and their contratensor product~$\ocn_{k(H)}$ is just
the conventional tensor product~$\ot_k$ of discrete vector spaces
over~$k$.
 There are no pseudo-compact vector spaces here at all, and
no topological completion is involved.

 So, what is the difference between the tensor product $\N\ot_{k[H]}\P$
and the contratensor product $\N\ocn_{k(H)}\P$~?
 It is that the contratensor product is a \emph{smaller}
vector space than the tensor product; indeed, generally speaking,
$\N\ocn_{k(H)}\P$ is a (discrete) quotient vector space of
a (discrete) $k$\+vector space $\N\ot_{k[H]}\P$.

 Let us recall that the contratensor product $\N\ocn_\C\P$ of
a right comodule $\N$ and a left contramodule $\P$ over a coalgebra $\C$
is defined as the quotient space of $\N\ot_k\P$ by the image of
a natural $k$\+linear map coming from the vector space
$\N\ot_k\Hom_k(\C,\P)$, while the tensor product $\N\ot_{\C\spcheck}\P$
of the $\C\spcheck$\+modules $\N$ and $\P$ is defined as the quotient
space of $\N\ot_k\P$ by the image of a natural $k$\+linear map coming
from the vector space $\N\ot_k\C\spcheck\ot_k\P$.
 Now, of course, $\N\ot_k\C\spcheck\ot_k\P$ is a subspace in
$\N\ot_k\Hom_k(\C,\P)$, and the latter map is the restriction of
the former one onto this subspace.
 Informally speaking, as the structure of a $\C$\+contramodule on $\P$
is richer than that of a $\C\spcheck$\+module, one can use it to
construct a deeper (discrete) quotient of the $k$\+tensor product
space $\N\ot_k\P$ than the structure of $\C\spcheck$\+module
on $\P$ allows.

 Similarly, when $\C=k(H)$ for a profinite group $H$ (so
$\C\spcheck=k[[H]]$), the tensor product $\N\ot_{k[[H]]}\P$, generally
speaking, is a quotient space of the tensor product $\N\ot_{k[H]}\P$.
 Furthermore, when $H'\subset H$ is a dense subgroup, the tensor
product $\N\ot_{k[H]}\P$ is, generally speaking, a quotient space of
the tensor product $\N\ot_{k[H']}\P$.
 Thus, for a discrete $H$\+module $\N$ and an $H$\+contramodule $\P$
over~$k$, we have a sequence of surjective maps of discrete
$k$\+vector spaces
$$
 \N\ot_{k[H']}\P\lrarrow\N\ot_{k[H]}\P\lrarrow\N\ot_{k[[H]]}\P
 \lrarrow\N\ocn_{k(H)}\P.
$$
 Corollary~\ref{fg-pro-p-contratensor-iso} claims that, when $H$ is
a finitely generated pro-$p$-group and $k$ is a field of
characteristic~$p$, all these maps are isomorphisms.
\end{rem}

\Section{Smooth $G$-Modules and $G$-Contramodules} \label{G-section}

\subsection{Contramodules over a locally profinite group}
\label{G-contramodules-subsecn}
 Let $G$ be a locally profinite group and $k$~be a field.
 A \emph{$G$\+module} over~$k$ is a $k$\+vector space endowed with
an action of $G$ (viewed as an abstract group).
 A $G$\+module $\bM$ over~$k$ is called \emph{smooth} if
the stabilizer of every element $x\in\bM$ is an open subgroup in~$G$.
 The abelian category of $G$\+modules over~$k$ is denoted by $G\modl_k$
and the abelian category of smooth $G$\+modules is denoted by
$G\smooth_k\subset G\modl_k$.

 Both the infinite direct sums and infinite products exist in
the abelian category $G\smooth_k$.
 The embedding functor $G\smooth_k\rarrow G\modl_k$ preserves
the infinite direct sums (but not the infinite products).
 In fact, the infinite product $\bM$ of a family of objects
$\bM_\alpha\in G\smooth_k$ is the submodule of all $G$\+smooth vectors
$\bM\subset M$ in their product $M=\prod_\alpha\bM_\alpha$ as objects
in $G\modl_k$.

 In order to define contramodules over locally profinite groups, we
will need the following extension of the discussion of finitely
additive measures on profinite sets in
Section~\ref{profinite-groups-subsecn} to the case of
locally profinite sets.

 Let $X$ be a locally profinite set (that is a locally compact totally
disconnected topological space).
 For any $k$\+vector space $V$, we denote by $V(X)$ the $k$\+vector
space of all compactly supported locally constant functions
$X\rarrow V$.
 We also denote by $V[[X]]$ the $k$\+vector space of all compactly
supported finitely additive $V$\+valued measures defined on all
the open-closed subsets in~$X$.
 Then for any two locally profinite sets $X$ and $Y$ there are
natural isomorphisms
$$
 V(X)\simeq V\ot_kk(X), \qquad k(X\times Y)\simeq k(X)\ot_kk(Y)
$$
and natural injective maps
$$
 V[[X]]\ot_k k[[Y]]\lrarrow V[[X\times Y]]\lrarrow V[[X]][[Y]].
$$
 In particular, there is a natural injective map $V\ot_kk[[X]]
\rarrow V[[X]]$.
 For any continuous map of locally profinite sets $X\rarrow Y$,
there is a natural pushforward map $V[[X]]\rarrow V[[Y]]$
(cf.~\cite[Section~E.1.1]{Psemi} and~\cite[Section~1.8]{Prev}).

 A \emph{$G$\+contramodule over~$k$} is a $k$\+vector space $\bP$
endowed with a $k$\+linear \emph{$G$\+contraaction} map
$\pi_\bP\:\bP[[G]]\rarrow\bP$ satisfying
the \emph{contraassociativity} and \emph{contraunitality} axioms.
 Specifically, the two maps $\bP[[G\times G]]\rightrightarrows
\bP[[G]]$, one of them provided by the pushforward of measures
with respect to the multiplication map $G\times G\rarrow G$, and
the other one constructed as the composition of the natural
injection $\bP[[G\times G]]\rarrow\bP[[G]][[G]]$ with the map
$\bP[[G]][[G]]\rarrow\bP[[G]]$ induced by the contraaction
map~$\pi_\bP$, should have equal compositions with
the contraaction map~$\pi_\bP$,
$$
 \bP[[G\times G]]\.\rightrightarrows\.\bP[[G]]\rarrow\bP.
$$
 Besides, the composition of the map $\bP\rarrow\bP[[G]]$ assigning
to a vector $x\in\bP$ the point measure on $G$ concentrated at
the unit element $e\in G$ and taking the value~$x$ on
the neighborhoods of~$e$ with the contraaction map~$\pi_\bP$ should be
equal to the identity map~$\operatorname{id}_\bP$,
$$
 \bP\lrarrow\bP[[G]]\lrarrow\bP.
$$

 $G$\+contramodules over~$k$ form an abelian category $G\contra_k$.
 For every element $g\in G$, one can consider the point measure at~$g$
with the value $1\in k$.
 This defines an injective map $k[G]\rarrow k[[G]]$.
 The composition of $k$\+linear maps $\bP\ot_kk[G]\rarrow
\bP\ot_kk[[G]]\rarrow\bP[[G]]\rarrow\bP$ provides a map $\bP\ot_kk[G]
\rarrow\bP$ defining the underlying $G$\+module structure on
a $G$\+contramodule~$\bP$.
 Hence we obtain an exact and faithful forgetful functor
$G\contra_k\rarrow G\modl_k$.

 Both the infinite direct sums and infinite products exist in
the abelian category $G\contra_k$.
 The forgetful functor $G\contra_k\rarrow G\modl_k$ preserves
the infinite products (but not the infinite direct sums).
 We refer to~\cite[Section~3.1.2]{Psemi} and~\cite[Section~2.6]{Prev}
for the construction of the infinite direct sums in the categories
of semicontramodules over semialgebras, which includes
the category of $G$\+contramodules as a particular
case~\cite[Example~2.6]{Prev}.

 For any smooth $G$\+module $\bN$ over $k$ and any $k$\+vector space
$V$, there is a natural structure of $G$\+contramodule on
the $k$\+vector space $\Hom_k(\bN,V)$.
 We refer to~\cite[Section~E.1.4]{Psemi} for the construction of
the contraaction map $\pi\:\Hom_k(\bN,V)[[G]]\rarrow\Hom_k(\bN,V)$
and to~\cite[Section~1.8]{Prev} for the discussion of its intuitive
meaning as a certain integration operation.

\begin{cor} \label{locally-fg-pro-p-fully-faithful}
 Let $G$ be a locally profinite group with a compact open subgroup
$H\subset G$.
 Assume that $H$ is a finitely generated pro-$p$-group and $k$~is
a field of characteristic~$p$.
 Then the forgetful functor $G\contra_k\rarrow G\modl_k$ is fully
faithful.

 Moreover, for any dense subgroup $G'\subset G$, the forgetful functor
$G\contra_k\rarrow G'\modl_k$ is fully faithful. 
\end{cor}

\begin{proof}
 Let $\bP$ and $\bQ$ be two $G$\+contramodules over~$k$, and let
$f\:\bP\rarrow\bQ$ be a $k[G']$\+module morphism.
 Set $H'=H\cap G'$.
 Then $f$~is, in particular, a $k[H']$\+module morphism.
 By Corollary~\ref{fg-pro-p-contra-fully-faithful}, it follows that
$f$~is a morphism of $H$\+contramodules over~$k$.
 Finally, we notice that the composition
$$
 \bP[[H]]\ot_kk[G']\lrarrow\bP[[H]]\ot_kk[[G]]\lrarrow
 \bP[[H\times G]]\lrarrow\bP[[G]]
$$
of the natural injective maps and the pushforward map is surjective.
 As $f$~is a morphism of $H$\+contramodules and a morphism of
$G'$\+modules, we can conclude that it is a morphism of
$G$\+contramodules over~$k$.
\end{proof}

\subsection{Contratensor product over a locally profinite group}
\label{contratensor-loc-profinite-subsecn}
 Let $G$ be locally profinite group and $k$~be a field.
 The \emph{contratensor product} $\bN\Ocn_{k,G}\nobreak\bP$ of
a smooth $G$\+module $\bN$ and a $G$\+contramodule $\bP$ over~$k$ is
a $k$\+vector space constructed as the cokernel of the difference of
two natural maps
\begin{equation} \label{G-contratensor-pair-of-maps}
 \bN\ot_k\bP[[G]]\.\rightrightarrows\.\bN\ot_k\bP.
\end{equation}
 The first map is simply induced by the contraaction map $\pi_\bP\:
\bP[[G]]\rarrow\bP$.
 The second map is defined by the formula
$$
 x\ot\mu\longmapsto\int_G xg^{-1}\ot d\mu_g
$$
for all $x\in\bN$ and $\mu\in\bP[[G]]$.
 Here $g\longmapsto xg^{-1}=gx$ is a smooth $\bN$\+valued function
on $G$, which can be integrated with the compactly supported
$\bP$\+valued measure~$\mu$ on open-closed subsets of $G$,
resulting in an element of $\bN\ot_k\bP$.

 For any smooth $G$\+module $\bN$, \ $G$\+contramodule $\bP$, and
a vector space $V$ over~$k$ there is a natural isomorphism of
$k$\+vector spaces
\begin{equation} \label{G-contratensor-hom-V}
 \Hom_k(\bN\Ocn_{k,G}\bP,\>V)\simeq\Hom_k^G(\bP,\Hom_k(\bN,V)),
\end{equation}
where $\Hom_k^G(\bP,\bQ)$ denotes the space of all morphisms
$\bP\rarrow\bQ$ is the category $G\contra_k$.
 Clearly, there is also a natural surjective morphism
$$
 \bN\ot_{k[G]}\bP\lrarrow\bN\Ocn_{k,G}\bP.
$$

 When $H$ is a profinite group, the contratensor product
$\N\Ocn_{k,H}\P$ of a discrete $H$\+module $\N$ and an $H$\+contramodule
$\P$ over~$k$ is nothing but their contratensor product
$\N\ocn_{k(H)}\P$ as a right comodule and a left contramodule over
the coalgebra $\C=k(H)$ (as defined in
Sections~\ref{cotensor-contratensor-subsecn}
and~\ref{profinite-groups-subsecn}).
 Indeed, one has $\P[[H]]\simeq\Hom_k(\C,\P)$, and this isomorphism
identifies the pair of maps $\N\ot_k\P[[H]]\rightrightarrows\N\ot_k\P$\,
\eqref{G-contratensor-pair-of-maps} with the pair of maps
$\N\ot_k\Hom_k(\C,\P)\rightrightarrows\N\ot_k\P$\,
\eqref{C-contratensor-pair-of-maps}.
 The following result is a generalization of
Corollary~\ref{fg-pro-p-contratensor-iso} to locally profinite groups.

\begin{cor} \label{locally-fg-pro-p-contratensor-iso}
 Let $G$ be a locally profinite group with a compact open subgroup
$H\subset G$.
 Assume that $H$ is a finitely generated pro-$p$-group and $k$~is
a field of characteristic~$p$.
 Then for any smooth $G$\+module\/ $\bN$ and any $G$\+contramodule\/
$\bP$ over~$k$ the natural map\/ $\bN\ot_{k[G]}\bP\rarrow
\bN\Ocn_{k,G}\bP$ is an isomorphism.

 Moreover, for any dense subgroup $G'\subset G$, the natural map\/
$\bN\ot_{k[G']}\bP\rarrow\bN\Ocn_{k,G}\nobreak\bP$ is an isomorphism.
\end{cor}

\begin{proof}
 This is deduced from Corollary~\ref{locally-fg-pro-p-fully-faithful}
in the same way as Corollary~\ref{fg-coalgebra-contratensor-iso}
is deduced from Theorem~\ref{fg-coalgebra-contra-fully-faithful}.
\end{proof}

\subsection{Smooth module-contramodule correspondence}
\label{smooth-contra-subsecn}
 Let $G'$ and $G''$ be two locally profinite groups, and
let $\bK$ be a smooth $(G'\times\nobreak G'')$\+module over~$k$.
 Then the functor
$$
 \bK\Ocn_{k,G''}{-}\:G''\contra_k\lrarrow G'\smooth_k
$$
taking a $G''$\+contramodule $\bP$ over~$k$ to the smooth $G'$\+module
$\bK\Ocn_{k,G''}\bP$ is left adjoint to the functor
$$
 \Hom_{k[G']}(\bK,{-})\:G'\smooth_k\lrarrow G''\contra_k
$$
taking a smooth $G'$\+module $\bM$ over~$k$ to the $G''$\+contramodule
$\Hom_{k[G']}(\bK,\bM)$.

 The smooth $(G\times G)$\+module $\bS=k(G)$ of compactly supported
locally constant $k$\+valued functions on a locally profinite group $G$
plays a central role.
 We denote the related pair of adjoint functors by
\begin{equation} \label{Phi-G}
 \Phi_G=\bS\Ocn_{k,G}{-}\:G\contra_k\lrarrow G\smooth_k
\end{equation}
and
\begin{equation} \label{Psi-G}
 \Psi_G=\Hom_{k[G]}(\bS,{-})\:G\smooth_k\lrarrow G\contra_k.
\end{equation}

 As it was mentioned and briefly discussed in
Sections~\ref{loc-profinite-group-semialgebra-subsecn}
and~\ref{hom-V-subsecn} of the introduction,
choosing a compact open subgroup $H\subset G$ endows
the smooth $(G\times G)$\+module $\bS$ viewed as a discrete
$(H\times H)$\+module with the structure of a semiassociative,
semiunital \emph{semialgebra} over the coalgebra $\C=k(H)$.
 The category of smooth $G$\+modules over~$k$ is then identified
with the category of (left or right) \emph{$\bS$\+semimodules},
$G\smooth_k=\bS\simodl=\simodr\bS$, and the category of
$G$\+contramodules over~$k$ is isomorphic to the category of
(left or right) \emph{$\bS$\+semicontramodules},
$G\contra_k=\bS\sicntr$ \cite[Sections~E.1.2\+-E.1.3]{Psemi},
\cite[Example~2.6]{Prev}.

 Given a smooth $G$\+module $\bN$ and a $G$\+contramodule $\bP$
over~$k$, their contratensor product $\bN\Ocn_{k,G}\bP$ is nothing
but their contratensor product $\bN\Ocn_\bS\bP$ as a right
$\bS$\+semimodule and a left $\bS$\+semicontramodule, in the sense
of~\cite[Sections~0.3.7 and~6.1.1]{Psemi}.
 In particular, the pair of adjoint functors $\Phi_G$ and $\Psi_G$
is isomorphic to the pair of adjoint functors $\Phi_\bS$ and
$\Psi_\bS$ of~\cite[Sections~0.3.7 and~6.1.4\+-6.2]{Psemi}.

 For a profinite group $H$ and the related coalgebra $\C=k(H)$,
the functors $\Phi_H$ and $\Psi_H$ agree with the functors $\Phi_\C$
and $\Psi_\C$ which were defined and discussed in
Section~\ref{cotensor-contratensor-subsecn}.
 This was essentially explained in
Section~\ref{contratensor-loc-profinite-subsecn}.
 Let us emphasize that the the functors $\Phi$ and $\Psi$ are
\emph{covariant}, and no passage to the dual vector space is involved
in their construction.

 In particular, according to the discussion in
Section~\ref{nonnatural-subsecn} of the introduction, when
the proorder of the profinite group $H$ is not divisible by
the characteristic of~$k$ (e.~g., $\operatorname{char} k=0$),
the functor $\Phi_H$ assigns to an $H$\+contramodule $\P$ its
smooth $H$\+submodule of all $H$\+smooth vectors in~$\P$.
 Furthermore, following the brief discussions in
Sections~\ref{coalgebra-duality}
and~\ref{cotensor-contratensor-subsecn}, when a $k(H)$\+comodule
$\M$ is \emph{quasi-finitely cogenerated} (or equivalently, in the more
traditional terminology, a smooth $H$\+module $\M$ is
\emph{admissible}), the $H$\+contramodule $\Psi_H(\M)$ is associated
with a certain pseudo-compact $k[[H]]$\+module (see
Lemmas~\ref{psi-pseudocompact} and~\ref{admissible-quasi-finite}(a)
for the details; cf.\ Proposition~\ref{contra-admissible-duality}).

\begin{ex}
 Generally speaking, when the characteristic of~$k$ divides
the proorder of the group $H$, the functor $\Phi_H$ is quite different
from the functor of maximal smooth $H$\+submodule.
 In particular, the functor $\Phi_H$ is right exact, while the functor
of maximal smooth $H$\+submodule is left exact.

 The following example is instructive.
 Let $H=\Z_p$ be the additive group of $p$\+adic integers and $k$~be
a field of characteristic~$p$.
 Then the ring $k[[H]]=\varprojlim_{n\ge1} k[z]/(z^{p^n}-\nobreak1)$ is
isomorphic to the ring of formal power series $k[[x]]$ in the variable
$x=z-1$.
 The category of smooth $H$\+modules over~$k$ is equivalent to
the category of $k[[x]]$\+modules with a locally nilpotent action
of~$x$; while the category of $H$\+contramodules over~$k$ is
equivalent to the category of $k[[x]]$\+contramodules, which means
$k[[x]]$\+modules with an $x$\+power infinite summation operation,
as described in~\cite[Section~1.3]{Prev}
(see also~\cite[Sections~1.5\+-1.6]{Prev} for a discussion of this
category's properties).

 The smooth $H$\+module $k(H)$ corresponds to the Pr\"ufer
$k[[x]]$\+module $k((x))/xk[[x]]$, and the functor $\Phi_H$ assigns to
an $H$\+contramodule $\P$ over~$k$ the smooth $H$\+module
$k(H)\ot_{k[[H]]}\P=k((x))/xk[[x]]\ot_{k[[x]]}\P$.
 This is quite different from the functor of the maximal smooth
$H$\+submodule (\,$=$~maximal $k[[x]]$\+submodule with a locally
nilpotent action of~$x$).
 For example, the projective $H$\+contramodules $\P$ over~$k$ correspond
to the projective (\,$=$~free) $k[[x]]$\+contramodules $V[[x]]$, where
$V$ ranges over the $k$\+vector spaces.
 There are no vectors with a locally nilpotent action of~$x$ in
$\P=V[[x]]$, so the maximal smooth $H$\+submodule in $\P$ vanishes;
while $\Phi(\P)=k(H)\ot_k V=k((x))/xk[[x]]\ot_k V$ is a quite nonzero
smooth $H$\+module.
\end{ex}

 Let $G'\subset G$ be an open subgroup.
 Then the functors $\Phi$ and $\Psi$ related to the groups $G$ and
$G'$ form commutative diagrams with the forgetful functors
$G\smooth_k\rarrow G'\smooth_k$ and $G\contra_k\rarrow G'\contra_k$,
\begin{equation} \label{G-G-prime-Phi-Psi}
\begin{diagram}
\node{\Psi_G\:G\smooth_k}\arrow{s}\arrow{e}
\node{G\contra_k}\arrow{s} \\
\node{\Psi_{G'}\:G'\smooth_k}\arrow{e}
\node{G'\contra_k}
\end{diagram}
\ \ 
\begin{diagram}
\node{\Phi_G\:G\contra_k}\arrow{s}\arrow{e}
\node{G\smooth_k}\arrow{s} \\
\node{\Phi_{G'}\:G'\contra_k}\arrow{e}
\node{G'\smooth_k}
\end{diagram}
\end{equation}
 Commutativity of the leftmost diagram follows immediately from
the fact that the $G$\+module $\bS_G=k(G)$ can be obtained by applying
the compactly supported smooth induction functor to
the $G'$\+module $\bS_{G'}=k(G')$, that is
$k(G)=\operatorname{ind}_{G'}^G k(G')$.
 Commutativity of the rightmost diagram can be obtained, e.~g.,
as a particular case of the results of~\cite[Section~8.1.2]{Psemi}.

\subsection{Weakly compactly injective and weakly compactly
projective objects}
 Let $H_1\subset H_2$ be an open subgroup in a profinite group.
 Then one can easily see that the forgetful functor
$H_2\discr_k\rarrow H_1\discr_k$ preserves injectives (and in fact
takes cofree $k(H_2)$\+comodules to cofree $k(H_1)$\+comodules).
 Similarly, the forgetful functor $H_2\contra_k\rarrow H_1\contra_k$
preserves projectives (and in fact takes free
$k(H_2)$\+contramodules to free $k(H_1)$\+contramodules).

 Given a locally profinite group $G$, let us call a smooth
$G$\+module $\bM$ over~$k$ \emph{weakly compactly injective} if there
exists a compact open subgroup $H\subset G$ such that $\bM$ is
injective as an object of $H\discr_k$.
 Similarly, let us call a $G$\+contramodule $\bP$ over~$k$
\emph{weakly compactly projective} if there exists a compact open
subgroup $H\subset G$ such that $\bP$ is projective as an object of
$H\contra_k$.

\begin{prop} \label{underived-semico-semicontra-prop}
 For any locally profinite group $G$ and a field~$k$, the functors\/
$\Psi_G$ and\/ $\Phi_G$ restrict to mutually inverse equivalences
between the full subcategories of weakly compactly injective objects in
$G\smooth_k$ and weakly compactly projective objects in $G\contra_k$.
\end{prop}

\begin{proof}
 It suffices to show that for every compact open subgroup $H\subset G$
the functors $\Psi_G$ and $\Phi_G$ restrict to mutually inverse
equivalences between the full subcategory of smooth $G$\+modules
over~$k$ that are injective as discrete $H$\+modules and the full
subcategory of $G$\+contramodules over~$k$ that are projective as
$H$\+contramodules.

 In view of commutativity of the diagrams~\eqref{G-G-prime-Phi-Psi}
for $G'=H$, the question reduces to showing that the functors
$\Psi_H=\Psi_\C$ and $\Phi_H=\Phi_\C$ (where $\C=k(H)$) restrict to
mutually inverse equivalences between the full subcategories of
injective objects in $H\discr_k=\C\comodl$ and projective objects
in $H\contra_k=\C\contra$.
 The latter is a standard result about comodules and
contramodules over a coalgebra over a field~\cite[Sections~0.2.6
and~5.1.3]{Psemi}, \cite[Sections~1.2 and~3.4]{Prev},
\cite[Sections~5.1\+-5.2]{Pkoszul}.
 Alternatively, one can apply directly the results
of~\cite[Sections~0.3.7 and~6.2]{Psemi}
(see also~\cite[Section~3.5]{Prev}).
\end{proof}

 A smooth $G$\+module over~$k$ is called \emph{semiprojective} if it
is a direct summand of an (infinite) direct sum of copies of
the smooth $G$\+module $\bS=k(G)$, or in other words, if it is a direct
summand of the smooth $G$\+module $\bS\ot_kV$ for some $k$\+vector
space~$V$.
 A $G$\+contramodule over~$k$ is called \emph{semiinjective} if it is
a direct summand of an (infinite) product of copies of
the $G$\+contramodule $\bS\spcheck$ Pontryagin dual to the smooth
$G$\+module $\bS=k(G)$, or in other words, if it is a direct summand
of the $G$\+contramodule $\Hom_k(\bS,V)$ for some $k$\+vector space~$V$.

 Semiprojective smooth $G$\+modules over~$k$ are weakly compactly
injective; in fact, they are injective objects of $H\discr_k$ for
every compact open subgroup $H\subset G$.
 Semiinjective $G$\+contramodules over~$k$ are weakly compactly
projective; in fact, they are projective objects of $H\contra_k$
for every compact open subgroup $H\subset G$.

\begin{prop} \label{semi-injectives-projectives-correspondence}
\textup{(a)} There are enough injective objects in the abelian
category $G\smooth_k$, and the forgetful functors $G\smooth_k\rarrow
H\discr_k$ preserve injectives, so injectives in $G\smooth_k$ are
weakly compactly injective.
 The functors\/ $\Psi_G$ and\/ $\Phi_G$ identify the full subcategory
of injective objects in $G\smooth_k$ with the full subcategory of
semiinjective objects in $G\contra_k$. \par
\textup{(b)} There are enough projective objects in the abelian
category $G\contra_k$, and the forgetful functors $G\contra_k
\rarrow H\contra_k$ preserve projectives, so projectives in
$G\contra_k$ are weakly compactly projective.
 The functors\/ $\Phi_G$ and\/ $\Psi_G$ identify the full subcategory
of projective objects in $G\contra_k$ with the full subcategory of
semiprojective objects in $G\smooth_k$.
\end{prop}

\begin{proof}
 Follows from Proposition~\ref{underived-semico-semicontra-prop}
and~\cite[Proposition~3.5]{Prev} (see~\cite[Proposition~4.1]{Pmc}
for a slightly more general result).
\end{proof}

\subsection{The universal acting algebra}
 Introducing a $k$\+algebra containing the group algebra $k[G]$ and
acting in all the smooth $G$\+modules over~$k$ is a natural thing
to do (see, e.~g, \cite[Section~1]{Koh}).
 In the context of the present paper, one may also wonder about
a $k$\+algebra containing $k[G]$ and acting in all
the $G$\+contramodules over~$k$.
 We choose the approach of computing and working with
the \emph{universal} such $k$\+algebra in both cases.
 It turns out that the two answers only differ by the passage to
the opposite algebra.

 Let us denote by $\fT=\Hom_{k[G]}(\bS,\bS)^\rop$ the opposite algebra to
the $k$\+algebra of endomorphisms of the smooth $G$\+module $\bS$
over~$k$ (with its left $G$\+module structure).
 This means that there is a right action of $\fT$ in $\bS$ commuting
with the left action of~$G$.
 The right action of $G$ in $\bS$ provides an injective homomorphism
$k[G]\rarrow\fT$.

 In particular, when $G=H$ is a profinite group, $\fT=\fR=\C\spcheck
=k[[H]]$ is the Pontryagin dual algebra to the coalgebra $\C=k(H)$.

\begin{prop}
\textup{(a)} The $k$\+algebra of endomorphisms of the forgetful
functor $G\smooth_k\rarrow k\modl$ is naturally isomorphic
to\/~$\fT^\rop$. \par
\textup{(b)} The $k$\+algebra of endomorphisms of the forgetful
functor $G\contra_k\rarrow k\modl$ is naturally isomorphic to\/~$\fT$.
\end{prop}

\begin{proof}
 First of all, we have to construct a natural right action of $\fT$
in smooth $G$\+modules and a natural left action of $\fT$ in
$G$\+contramodules over~$k$.
 The most straightforward approach would be to compute the forgetful
functor $\simodr\bS=G\smooth_k\rarrow k\modl$ as the functor of
semitensor product ${-}\mathbin{\lozenge}_\bS\bS$ with the left
$\bS$\+semimodule $\bS$ \cite[Sections~0.3.2 and~1.4.1]{Psemi} and
the forgetful functor $\bS\sicntr=G\contra_k\rarrow k\modl$ as
the functor of semihomomorphisms $\operatorname{SemiHom}_\bS(\bS,{-})$
from the left $\bS$\+semimodule $\bS$
\cite[Sections~0.3.5 and~3.4.1\+-3.4.2]{Psemi}.
 A more roundabout argument below is based on the constructions
discussed above in this paper. 

 Part~(a): let $G\smooth_k^\inj\subset G\smooth_k$ denote the full
subcategory of injective smooth $G$\+modules and
$G\contra_k^\siin\subset G\contra_k$ denote the full subcategory
of semiinjective $G$\+contramodules over~$k$.
 Since there are enough injectives in $G\smooth_k$, the algebra
of endomorphisms of the forgetful functor $G\smooth_k\rarrow
k\modl$ is isomorphic to the algebra of endomorphisms of the
restriction of this functor to the full subcategory of injective
objects $G\smooth_k^\inj\subset G\smooth_k$.
 In view of the equivalence of categories $G\smooth_k^\inj\simeq
G\contra_k^\siin$ from
Proposition~\ref{semi-injectives-projectives-correspondence}(a),
the latter algebra is isomorphic to the algebra of endomorphisms
of the functor of contratensor product $\bS\Ocn_{k,G}{-}\:
G\contra_k^\siin\rarrow k\modl$.
 The ring $\fT^\rop$ of endomorphisms of the smooth $G$\+module
$\bS$ (with its right $G$\+module structure) acts naturally by
endomorphisms of this functor on the left.

 Now consider the smooth $G$\+module $\bS$ with its left $G$\+module
structure.
 The group $G$ acts on the right by automorphisms of the object
$\bS\in G\smooth_k$.
 Hence any endomorphism of the forgetful functor $G\smooth_k\rarrow
k\modl$ must act in the $k$\+vector space $\bS$ by an operator
commuting with the right action of $G$, i.~e., by an operator
coming from an element of $\fT^\rop$.
 It remains to show that every nonzero endomorphism of the forgetful
functor acts in $\bS$ by a nonzero operator.
 Indeed, in view of 
Proposition~\ref{semi-injectives-projectives-correspondence}(b),
$\bS$ is a projective generator of the exact category of weakly
compactly injective smooth $G$\+modules over~$k$, which by
Proposition~\ref{semi-injectives-projectives-correspondence}(a)
contains an injective cogenerator of the abelian category $G\smooth_k$.

 Part~(b): let $G\contra_k^\proj\subset G\contra_k$ denote the full
subcategory of projective $G$\+contramodules and $G\smooth_k^\sipr
\subset G\smooth_k$ denote the full subcategory of semiprojective
smooth $G$\+modules over~$k$.
 Since there are enough projectives in $G\contra_k$, the algebra of
endomorphisms of the forgetful functor $G\contra_k\rarrow k\modl$
is isomorphic to the algebra of endomorphisms of the restriction of
this functor to the full subcategory of projective objects
$G\contra_k^\proj\subset G\contra_k$.
 In view of the equivalence of categories $G\contra_k^\proj\simeq
G\smooth_k^\sipr$ from
Proposition~\ref{semi-injectives-projectives-correspondence}(b),
the latter algebra is isomorphic to the algebra of endomorphisms
of the corepresentable functor $\Hom_{k[G]}(\bS,{-})\:
G\smooth_k^\sipr\rarrow k\modl$.
 The latter algebra is isomorphic to the opposite algebra to
the algebra of endomophisms of the corepresenting object $\bS$
(with its left $G$\+module structure), that is the algebra~$\fT$.
\end{proof}

 Let $\bN$ be a smooth $G$\+module and $\bP$ be a left
$G$\+contramodule over~$k$.
 Then for every $k$\+vector space $V$ there is a natural morphism
of $k$\+vector spaces
\begin{multline*}
 \Hom_k(\bN\Ocn_{k,G}\bP,\,V)\.\simeq\.
 \Hom_k^G(\bP,\.\Hom_k(\bN,V))\\ \lrarrow
 \Hom_\fT(\bP,\.\Hom_k(\bN,V))\.\simeq\.
 \Hom_k(\bN\ot_\fT\bP,\,V).
\end{multline*}
 This $k$\+linear map being functorial in~$V$, it follows that there
is a natural morphism of $k$\+vector spaces
$$
 \bN\ot_\fT\bP\lrarrow\bN\Ocn_{k,G}\bP,
$$
which is surjective, because the map
$$
 \Hom_k^G(\bP,\.\Hom_k(\bN,V))\lrarrow\Hom_\fT(\bP,\.\Hom_k(\bN,V))
$$
is injective (as the identity embedding of two subspaces in
$\Hom_k(\bP,\Hom_k(\bN,V))$).

\begin{cor} \label{locally-fg-pro-p-contra-ring-T-fully-faithful}
 Let $G$ be a locally profinite group with a compact open subgroup
$H\subset G$.
 Assume that $H$ is a finitely generated pro-$p$-group and $k$~is
a field of characteristic~$p$.
 Then the forgetful functor $G\contra_k\rarrow\fT\modl$ is fully
faithful.
\end{cor}

\begin{proof}
 This is a weaker version of
Corollary~\ref{locally-fg-pro-p-fully-faithful}, as $k[G]$ is
a subalgebra in~$\fT$.
\end{proof}

\begin{cor} \label{locally-fg-pro-p-ring-T-contratensor-iso}
  Let $G$ be a locally profinite group with a compact open subgroup
$H\subset G$.
 Assume that $H$ is a finitely generated pro-$p$-group and $k$~is
a field of characteristic~$p$.
 Then for any smooth $G$\+module\/ $\bN$ and any $G$\+contramodule\/
$\bP$ over~$k$ the natural map\/ $\bN\ot_\fT\bP\rarrow\bN\Ocn_{k,G}\bP$
is an isomorphism.
\end{cor}

\begin{proof}
 Follows from
Corollary~\ref{locally-fg-pro-p-contra-ring-T-fully-faithful} or
Corollary~\ref{locally-fg-pro-p-contratensor-iso}.
\end{proof}

 We denote by $\Ext_{\fT^\rop}$ and $\Tor^\fT$ the conventional Ext and
Tor functors over the ring $\fT^\rop$ or $\fT$ (i.~e., the derived
functors of Hom and tensor product computed in the abelian categories
of $\fT$\+modules).

\begin{cor} \label{adjusted-to-ext-tor-over-T}
 Let $G$ be a $p$\+adic Lie group and $k$ be a field of
characteristic~$p$.
 Then \par
\textup{(a)} for any weakly compactly injective smooth $G$\+module\/
$\bM$ over~$k$ one has\/ $\Ext_{\fT^\rop}^i(\bS,\bM)=0$ for all\/
$i>0$; \par
\textup{(b)} for any weakly compactly projective $G$\+contramodule\/
$\bP$ over~$k$ one has\/ $\Tor^\fT_i(\bS,\bP)=0$ for all\/ $i>0$.
\end{cor}

\begin{proof}
 The group $G$ has a base of neighborhoods of zero formed by open
subgroups that are uniform pro-$p$-groups \cite[Proposition~1.16,
Theorem~4.2, and Theorem~8.32]{DSMS}.
 So we can fix such a subgroup $H\subset G$ for which $\bM$ is
an injective object in $H\discr_k$ and $\bP$ is a projective object
in $H\contra_k$.
 Set $\fR=k[[H]]=\C\spcheck\simeq\fR^\rop$;
then $\fR=\Hom_{k[H]}(\C,\C)$ is a subring in
$\fT=\Hom_{k[G]}(\bS,\bS)^\rop=\Hom_{k[H]}(\C,\bS)$.

 Furthermore, $\fT=\Psi_G(\bS)$ is naturally a $G$\+contramodule
over~$k$ with its left $\fT$\+module structure induced by its
$G$\+contramodule structure.
 The forgetful functor $G\contra_k\rarrow H\contra_k$ takes $\fT$
to a projective $H$\+contramodule over~$k$, since a semiprojective
smooth $G$\+module $\bS$ over $k$ is injective as a discrete
$H$\+module (cf.\ the proof of
Proposition~\ref{underived-semico-semicontra-prop} and
the subsequent discussion).
 By Corollary~\ref{inj-inj-proj-flat-cor}(b), it follows that $\fT$
is a flat left $\fR$\+module.
 In view of Corollary~\ref{fg-coalgebra-contratensor-iso},
the commutative diagram~\eqref{G-G-prime-Phi-Psi}, and
Proposition~\ref{underived-semico-semicontra-prop} or
Corollary~\ref{locally-fg-pro-p-ring-T-contratensor-iso}, we have
$$
 \C\ot_\fR\fT\.\overset{\ref{fg-coalgebra-contratensor-iso}}=\.
 \C\ocn_\C\fT\.=\.\Phi_H(\fT)
 \.\overset{\eqref{G-G-prime-Phi-Psi}}=\.\Phi_G(\fT)
 \.\overset{\ref{underived-semico-semicontra-prop}
 \ \text{or}\ \ref{locally-fg-pro-p-ring-T-contratensor-iso}}
 \simeq\.\bS,
$$
where the isomorphism $\Phi_\C(\fT)=\C\ocn_\C\fT\simeq
\C\Ocn_{k,H}\fT=\Phi_H(\fT)$ holds according to the discussion in
Sections~\ref{contratensor-loc-profinite-subsecn}\+-%
\ref{smooth-contra-subsecn}.

 Now we can compute
$$
 \boR\Hom^i_{\fT^\rop}(\bS,\bM)\.\simeq\.
 \boR\Hom^i_{\fT^\rop}(\fT^\rop\ot_\fR^\boL\C,\,\bM)
 \.\simeq\.\boR\Hom^i_\fR(\C,\bM)\.=\.0, \quad i>0
$$
by Corollary~\ref{inj-inj-proj-flat-cor}(a), and
$$
 H_i(\bS\ot_\fT^\boL\bP)\.=\.H_i((\C\ot_\fR^\boL\fT)\ot_\fT^\boL\bP)
 \.=\.H_i(\C\ot_\fR^\boL\bP)\.=\.0, \quad i>0
$$
by Corollary~\ref{inj-inj-proj-flat-cor}(b).
\end{proof}

\Section{Derived Equivalence and Duality Adjunction}

\subsection{Derived equivalence} \label{derived-equivalence-subsecn}
 Let $H$ be a profinite group and $k$ be a field.
 The \emph{$k$\+cohomological dimension of $H$} is conventionally
defined as the supremum of the set of all integers~$i$ for
which there exists a discrete $H$\+module $\M$ over~$k$ such that
$H^i(H,\M)\ne0$.
 Alternatively, the $k$\+cohomological dimension of $H$ can be
defined as the homological dimension of the abelian category
$H\discr_k$.
 The $k$\+cohomological dimension of any open subgroup $H'\subset H$
does not exceed that of~$H$; moreover, the $k$\+cohomological
dimensions of $H$ and $H'$ coincide if $H$ contains no elements
of finite order equal to $\operatorname{char}k$ \cite{Ser}.

 For any coalgebra $\C$ over~$k$, the homological dimensions of
the abelian categories of left $\C$\+comodules, right $\C$\+comodules,
and left $\C$\+contramodules are equal to each other, as all of
them are equal to the homological dimensions of the derived functors
$\operatorname{Cotor}^\C_*$ and $\operatorname{Coext}_\C^*$
\cite[Sections~0.2.2, 0.2.5, and~0.2.9]{Psemi},
\cite[Section~4.5]{Pkoszul}, \cite[Corollary~1.9.4]{Pweak}.
 In particular, the homological dimensions of the abelian categories
$H\discr_k$ and $H\contra_k$ are equal to each other (and to
the $k$\+cohomological dimension of~$H$).

 Let us say that a locally profinite group $G$ is \emph{locally of
finite $k$\+cohomological dimension~$n$} if it has a compact open
subgroup $H\subset G$ of finite $k$\+cohomological dimension~$n$.
 In this case, compact open subgroups of $k$\+cohomological
dimension~$n$ form a base of neighborhoods of zero in~$G$.

 In particular, a $p$\+adic Lie group $G$ is locally of finite
$k$\+cohomological dimension for any field~$k$.
 Specifically, in the unnatural characteristic $\operatorname{char}k
\ne p$ one has $n=0$, and in the natural characteristic
$\operatorname{char}k=p$ the local $k$\+cohomological dimension~$n$
is equal to the dimension of the group.

\begin{thm} \label{derived-equivalence}
 Let $k$ be a field and $G$ be a locally profinite group locally of
finite $k$\+cohomological dimension.
 Then for any derived category symbol\/ $\star=\b$, $+$, $-$,
or\/~$\varnothing$ there is a natural triangulated equivalence
between the derived categories of smooth $G$\+modules and
$G$\+contramodules over~$k$,
\begin{equation} \label{loc-finite-cohomol-dim-derived-equivalence}
 \boR\Psi_G\:\sD^\star(G\smooth_k)\.\simeq\.\sD^\star(G\contra_k):
 \!\boL\Phi_G
\end{equation}
provided by the derived functors of the adjoint functors
\textup{\eqref{Phi-G}} and~\textup{\eqref{Psi-G}}.
\end{thm}

\begin{proof}
 Denote by $G\smooth_k^\wcin\subset G\smooth_k$ the full subcategory
of weakly compactly injective smooth $G$\+modules and by
$G\contra_k^\wcpr\subset G\contra_k$ the full subcategory of
weakly compactly projective $G$\+contramodules over~$k$.

 The full subcategory $G\smooth_k^\wcin$ is closed under the cokernels
of monomorphisms, extensions, and direct summands in the abelian
category $G\smooth_k$; and every object of $G\smooth_k$ is a subobject
of an object from $G\smooth_k^\wcin$.
 In other words, $G\smooth_k^\wcin$ is a \emph{coresolving subcategory}
in $G\smooth_k$.
 In particular, as a full subcategory closed under extensions,
the category $G\smooth_k^\wcin$ inherits a Quillen exact category
structure from the abelian category $G\smooth_k$.

 Similarly, the full subcategory $G\contra_k^\wcpr$ is closed under
the kernels of epimorphisms, extensions, and direct summands in
the abelian category $G\contra_k$; and every object of $G\contra_k$
is a quotient object of an object from $G\contra_k^\wcpr$.
 In other words, $G\contra_k^\wcpr$ is a \emph{resolving subcategory}
in $G\contra_k$.
 In particular, the full subcategory $G\contra_k^\wcpr$ inherits
an exact category structure from the ambient abelian category
$G\contra_k$.

 The result of Proposition~\ref{underived-semico-semicontra-prop}
provides an equivalence between the two exact categories
$G\smooth_k^\wcin$ and $G\contra_k^\wcpr$.
 All of these assertions do not yet depend on the locally finite
$k$\+cohomological dimension assumption.

 When the group $G$ is locally of finite $k$\+cohomological
dimension~$n$, every smooth $G$\+module over~$k$ has a finite right
resolution of length~$\le n$ by modules from $G\smooth_k^\wcin$, and
every $G$\+contramodule over~$k$ has a finite left resolution of
length~$\le n$ by contramodules from $G\contra_k^\wcpr$.
 In other words, one can say that the weakly compactly injective
dimension of any smooth $G$\+module over~$k$ does not exceed~$n$,
and the weakly compactly projective dimension of any $G$\+contramodule
over~$k$ does not exceed~$n$.

 Applying the result of~\cite[Proposition~13.2.6]{KS}
or~\cite[Proposition~A.5.6]{Pcosh} and the assertion dual to it
(see also~\cite[\S\S\,I.5 and~I.7]{Har}),
we obtain triangulated equivalences
$$
 \sD^\star(G\smooth_k)\.\simeq\.\sD^\star(G\smooth_k^\wcin)\.\simeq\.
 \sD^\star(G\contra_k^\wcpr)\.\simeq\.\sD^\star(G\contra_k).
$$
 In other words, one can say that the (mutually inverse) derived
functors $\boR\Psi$ and $\boL\Phi$ providing the desired triangulated
equivalence are constructed using resolutions of compexes of smooth
$G$\+modules and complexes of $G$\+contramodules by arbitrary
complexes of weakly compactly injective smooth $G$\+modules and
weakly compactly projective $G$\+contramodules over~$k$.
\end{proof}

 In conclusion, let us restate the simplified descriptions of
some elements of the picture of Theorem~\ref{derived-equivalence}
provided by some results of Section~\ref{G-section}.

 The functor $\Psi_G\:G\smooth_k\rarrow G\contra_k$ is always simply
$\Psi_G\:\bM\longmapsto \Hom_{k[G]}(\bS,\bM)$, where $\bS=k(G)$ is
the smooth $G\times G$\+module of compactly supported locally
constant $k$\+valued functions on~$G$.

 The functor $\Phi_G\:G\contra_k\rarrow G\smooth_k$ is, generally
speaking, defined as the contratensor product
$\Phi_G\:\bP\longmapsto\bS\Ocn_{k,G}\bP$.
 However, when the group $G$ has a compact open subgroup $H$ which
is a finitely generated pro-$p$-group, and $k$ is a field of
characteristic~$p$, the contratensor product does not differ from
the conventional tensor product, $\Phi_G\:\bP\longmapsto\bS\ot_{k[G]}
\bP$ (see Corollary~\ref{locally-fg-pro-p-contratensor-iso}).
 In particular, this is applicable to any $p$\+adic Lie group~$G$.

 When $G$ is a $p$\+adic Lie group, the derived functor $\boR\Psi_G$
can be computed as the conventional $\boR\Hom$ over
the ring~$\fT^\rop$,
$$
 \boR\Psi_G\:\bM^\bu\.\longmapsto\.\boR\Hom_{\fT^\rop}(\bS,\bM^\bu),
$$
and the derived functor $\boL\Phi_G$ can be computed as
the conventional derived tensor product over the ring~$\fT$,
$$
 \boL\Phi_G\:\bP^\bu\.\longmapsto\.\bS\ot_\fT^\boL\bP^\bu.
$$
 This follows essentially from
Corollary~\ref{adjusted-to-ext-tor-over-T}.

\subsection{Duality adjunction} \label{duality-adjunction-subsecn}
 Let $k$~be a field and $G$ be a locally profinite group locally of
finite $k$\+cohomological dimension.
 Let $V$ be a fixed $k$\+vector space.
 Then the contravariant exact functor $\Hom_k({-},V)\:(G\smooth_k)^\sop
\rarrow G\contra_k$ induces a contravariant triangulated functor
between the derived categories
\begin{equation} \label{G-hom-V}
 \sD^\star(G\smooth_k)^\sop\lrarrow\sD^\star(G\contra).
\end{equation}
 Composing the contravariant triangulated functor~\eqref{G-hom-V}
with the triangulated equivalence $\boL\Phi_G\:\sD^\star(G\contra_k)
\rarrow\sD^\star(G\smooth_k)$\,
\eqref{loc-finite-cohomol-dim-derived-equivalence}, we obtain
a contravariant triangulated functor
\begin{equation} \label{Delta-G-V}
 \Delta_G^V\:\sD^\star(G\smooth_k)^\sop\lrarrow
 \sD^\star(G\smooth_k).
\end{equation}

 The following result is a particular case of the discussion in
Section~\ref{duality-adjunction-introd} of the introduction.

\begin{prop} \label{derived-self-adjunction-prop}
 The duality functor~\eqref{Delta-G-V} is \emph{right self-adjoint},
that is, for any two complexes of smooth $G$\+modules\/
$\bM^\bu$ and\/ $\bN^\bu\in\sD^\star(G\smooth_k)$, there is a natural
isomorphism of Hom spaces
\begin{equation} \label{G-self-adjunction-isomorphism}
 \Hom_{\sD^\star(G\smooth_k)}(\bM^\bu,\Delta_G^V(\bN^\bu))\.\simeq\.
 \Hom_{\sD^\star(G\smooth_k)}(\bN^\bu,\Delta_G^V(\bM^\bu)).
\end{equation}
\end{prop}

\begin{proof}[Sketch of proof]
 The argument is based on the results of the book~\cite{Psemi}
and uses the formalism of semi-infinite homology and cohomology
of locally profinite groups (which is a particular case of
the semi-infinite homology and cohomology of semialgebras over
coalgebras).
 Both side of the desired
isomorphism~\eqref{G-self-adjunction-isomorphism}
are computed as a certain semi-infinite (co)homology space.

 We choose a compact open subgroup $H\subset G$ of finite
$k$\+cohomological dimension and consider
the semialgebra $\bS=k(G)$ over the coalgebra $\C=k(H)$
\cite[Section~E.1]{Psemi}, \cite[Example~2.6]{Prev}.
 Then smooth $G$\+modules over~$k$ are the same thing as
$\bS$\+semimodules, and $G$\+contramodules over~$k$ are the same
thing as $\bS$\+semicontramodules.
 The involutive anti-authomorphisms of $\C$ and $\bS$ induced by
the inverse element maps $H\rarrow H$ and $G\rarrow G$ identify
the left and right $\bS$\+semimodules.

 For any complex of left semimodules $\bM^\bu$ and any complex of
right semimodules $\bN^\bu$ over a semialgebra $\bS$, the construction
of~\cite[Section~2.7]{Psemi} produces a cohomologically graded
$k$\+vector space of semi-infinite homology
$$
 \SemiTor^\bS(\bN^\bu,\bM^\bu)\in\sD(k\vect).
$$
 In the situation at hand, when $\bS=k(G)$ and both $\bM^\bu$ and
$\bN^\bu$ are complexes of smooth $G$\+modules over~$k$,
the functor $\SemiTor^\bS$ can be thought of as a double-sided
derived functor of the functor of \emph{$(G,H)$\+semiinvariants}
$(\bN^\bu\ot_k\bM^\bu)_{G,H}$ in the tensor product of complexes
of smooth $G$\+modules $\bN^\bu\ot_k\bM^\bu$ \cite[Section~E.2]{Psemi}.
 In particular, one has
$$
 \SemiTor^\bS(\bN^\bu,\bM^\bu)\simeq\SemiTor^\bS(\bM^\bu,\bN^\bu).
$$

 Furthermore, for any complex of left semimodules $\bM^\bu$ and any
complex of left semicontramodules $\bP^\bu$ over a semialgebra $\bS$,
the construction of~\cite[Section~4.7]{Psemi} produces a graded
$k$\+vector space of semi-infinite cohomology
$$
 \SemiExt_\bS(\bM^\bu,\bP^\bu)\in\sD(k\vect).
$$
 Generally speaking, for a locally profinite group $G$ and the related
semialgebra $\bS=k(G)$ over the coalgebra $\C=k(H)$, the functors
SemiTor and SemiExt depend on the choice of a compact open subgroup
$H\subset G$; but when $G$ is locally of finite $k$\+cohomological
dimension and $H$ is chosen to be of finite $k$\+cohomological dimension,
they don't.
 (We do not need to use this fact.)

 According to~\cite[last formula of Section~4]{Psemi}, for any complexes
of smooth $G$\+modules $\bN^\bu$ and $\bM^\bu$ we have
$$
 \Hom_k(\SemiTor^\bS(\bN^\bu,\bM^\bu),\.V)\simeq
 \SemiExt_\bS(\bM^\bu,\.\Hom_k(\bN^\bu,V)).
$$
 Furthermore, by~\cite[Corollary~6.6(a)]{Psemi}, for any complex
of smooth $G$\+modules $\bM^\bu$ and any complex of $G$\+contramodules
$\bP^\bu$ over~$k$ we have
$$
 \SemiExt_\bS(\bM^\bu,\bP^\bu)\simeq
 \Ext_\bS(\bM^\bu,\.\boL\Phi_\bS(\bP^\bu)).
$$
 Combining the natural isomorphisms above, we get
\begin{multline} \label{adjunction-computed}
\Ext_\bS(\bM^\bu,\>\boL\Phi_\bS\Hom_k(\bN^\bu,V))\.\simeq\.
\SemiExt_\bS(\bM^\bu,\.\Hom_k(\bN^\bu,V)) \\ \.\simeq\.
\Hom_k(\SemiTor^\bS(\bN^\bu,\bM^\bu),\.V) \.\simeq\.
\Hom_k(\SemiTor^\bS(\bM^\bu,\bN^\bu),\.V) \\ \.\simeq\.
\SemiExt_\bS(\bN^\bu,\.\Hom_k(\bM^\bu,V)) \.\simeq\.
\Ext_\bS(\bN^\bu,\>\boL\Phi_\bS\Hom_k(\bM^\bu,V)).
\end{multline}
 Passing to the degree-zero components in the isomorphism of
graded vector spaces~\eqref{adjunction-computed}, we obtain
the desired adjunction
isomorphism~\eqref{G-self-adjunction-isomorphism}.
\hbadness=1100
\end{proof}

 The self-adjunction of Proposition~\ref{derived-self-adjunction-prop}
can be also described in terms of the adjunction morphisms rather than
adjunction isomorphisms.
 For any complex of smooth $G$\+modules $\bM^\bu\in
\sD^\star(G\smooth_k)$ there is a natural adjunction morphism
\begin{equation} \label{derived-adjunction-morphism}
 \bM^\bu\lrarrow\Delta^V_G\Delta^V_G(\bM^\bu)
\end{equation}
in the derived category $\sD^\star(G\smooth_k)$ which corresponds to
the identity morphism $\Delta^V_G(\bM^\bu)\rarrow\Delta^V_G(\bM^\bu)$
under the adjunction isomorphism~\eqref{G-self-adjunction-isomorphism}.
 The natural morphism~\eqref{derived-adjunction-morphism} can be
constructed as follows.

 Let $H\subset G$ be a compact open subgroup of finite
$k$\+cohomological dimension.
 Up to an isomorphism in $\sD^\star(G\smooth_k)$, we can assume
$\bM^\bu$ to be a complex of smooth $G$\+modules over~$k$ whose terms
are injective as discrete $H$\+modules.
 Then $\bP^\bu=\Psi_G(\bM^\bu)$ and $\bQ^\bu=\Hom_k(\bM^\bu,V)$
are complexes of $G$\+contramodules over~$k$ whose terms are
projective as $H$\+contramodules over~$k$.
 By Proposition~\ref{underived-semico-semicontra-prop}, we have
a natural isomorphism of complexes of smooth $G$\+modules
$\Phi_G(\bP^\bu)\simeq\bM^\bu$.

 According to the formula~\eqref{G-contratensor-hom-V}, we have
a natural isomorphism of complexes of $k$\+vector spaces
$\Hom_k(\bM^\bu\Ocn_{k,G}\bQ^\bu,\>V)\simeq
\Hom^G_k(\bQ^\bu,\.\Hom_k(\bM^\bu,V))=\Hom^G_k(\bQ^\bu,\bQ^\bu)$.
 Hence the identity morphism of complexes of $G$\+contramodules
$\bQ^\bu\rarrow\bQ^\bu$ corresponds to a morphism of complexes of
$k$\+vector spaces {\hbadness=2100
\begin{equation}
 \bM^\bu\Ocn_{k,G}\bQ^\bu\lrarrow V.
\end{equation}

}\begin{lem}
 For any locally profinite group $G$, any field~$k$, and any two
$G$\+contra\-modules\/ $\bP$ and\/ $\bQ$ over~$k$, there is a natural
isomorphism of $k$\+vector spaces
$$
 \Phi_G(\bP)\Ocn_{k,G}\bQ\.\simeq\.\Phi_G(\bQ)\Ocn_{k,G}\bP.
$$
\end{lem}

\begin{proof}
 We have to construct an isomorphism $(\bS\Ocn_{k,G}\bP)\Ocn_{k,G}\bQ
\simeq(\bS\Ocn_{k,G}\bQ)\Ocn_{k,G}\bP$.
 Following the definition of the contratensor product in
Section~\ref{contratensor-loc-profinite-subsecn}, both the vector
spaces in question can be identified with the cokernel of a certain
natural map
$$
 t\:\bS\ot_k\bP[[G]]\ot_k\bQ\.\oplus\.\bS\ot_k\bP\ot_k\bQ[[G]]
 \lrarrow\bS\ot_k\bP\ot_k\bQ.
$$
 Let us explain what the map~$t$ is.
 It has two components $t_1\:\bS\ot_k\bP[[G]]\ot_k\bQ\rarrow
\bS\ot_k\bP\ot_k\bQ$ and $t_2\:\bS\ot_k\bP\ot_k\bQ[[G]]\rarrow
\bS\ot_k\bP \ot_k\bQ$, each of which is the difference of two natural
maps, $t_i=r_i-s_i$.
 The maps~$r_1$ and~$r_2$ are induced by the $G$\+contraaction
maps~$\pi_\bP$ and~$\pi_\bQ$, respectively.
 The maps~$s_1$ and~$s_2$ are constructed using the smooth $G$\+action
in $\bS$, as explained in
Section~\ref{contratensor-loc-profinite-subsecn}.
 There are two such smooth actions, the left and the right one; and
one uses, say, the right action to construct the map
$\bS\ot_k\bP[[G]]\rarrow\bS\ot_k\bP$ inducing the map~$s_1$,
and the left action to construct the map
$\bS\ot_k\bQ[[G]]\rarrow\bS\ot_k\bQ$ inducing the map~$s_2$.
\end{proof}

 Returning to the situation at hand, we now have
$$
 \bM^\bu\Ocn_{k,G}\bQ^\bu\simeq\Phi_G(\bP^\bu)\Ocn_{k,G}\bQ^\bu
 \simeq\Phi_G(\bQ^\bu)\Ocn_{k,G}\bP^\bu.
$$
 Set $\bN^\bu=\Phi_G(\bQ^\bu)$.
 We have constructed a morphism of complexes of $k$\+vector spaces
\begin{equation} \label{N-P-V}
 \bN^\bu\Ocn_{k,G}\bP^\bu\lrarrow V.
\end{equation}

 Once again, by~\eqref{G-contratensor-hom-V}, we have a natural
isomorphism of complexes of $k$\+vector spaces
$\Hom_k(\bN^\bu\Ocn_{k,G}\bP^\bu,\>V)\simeq\Hom_k^G(\bP^\bu,\.
\Hom_k(\bN^\bu,V))$.
 In view of this isomorphism, the map~\eqref{N-P-V} corresponds
to a morphism of complexes of $G$\+contramodules over~$k$
$$
 \bP^\bu\lrarrow\Hom_k(\bN^\bu,V).
$$
 Applying the functor $\Phi_G$, we obtain
\begin{equation} \label{adjunction-morphism-computed}
 \bM^\bu\simeq\Phi_G(\bP^\bu)\lrarrow\Phi_G\Hom_k(\bN^\bu,V)
 =\Phi_G\Hom_k(\Phi_G\Hom_k(\bM^\bu,V),V).
\end{equation}

 It remains to recall that $\bQ^\bu$ is a complex of
$G$\+contramodules whose terms are projective as $H$\+contramodules
over~$k$, hence $\bN^\bu=\Phi_G(\bQ^\bu)$ is a complex of
smooth $G$\+modules whose terms are injective as discrete $H$\+modules
over~$k$, hence $\Hom_k(\bN^\bu,V)$ is also a complex of
$G$\+contramodules whose terms are projective as $H$\+contramodules
over~$k$.
 Thus the complexes $\Hom_k(\bM^\bu,V)$ and
$\Hom_k(\Phi_G\Hom_k(\bM^\bu,V),V)$ are adjusted to the derived functor
$\boL\Phi_G$, that is, one can compute $\boL\Phi_G$ for these
complexes simply by applying the functor~$\Phi_G$.
{\hbadness=1500\par}

 The map~\eqref{adjunction-morphism-computed} provides the desired
self-adjunction morphism~\eqref{derived-adjunction-morphism}, and
our construction is finished.

\Section{Admissibility Conditions}
\label{admissibility-secn}

\subsection{Quasi-finite comodules and contramodules}
\label{quasi-finite-subsecn}
 Let $\C$ be a (coassociative, counital) coalgebra over a field~$k$.
 The definitions of finitely cogenerated $\C$\+comodules and finitely
generated $\C$\+contramodules are straightforward: a left $\C$\+comodule
$\M$ is said to be \emph{finitely cogenerated} if it is a subcomodule of
a cofree left $\C$\+comodule $\C\ot_k V$ with a finite-dimensional
vector space of cogenerators~$V$.
 Similarly, a left $\C$\+contramodule $\P$ is said to be \emph{finitely
generated} if it is a quotient contramodule of a free left
$\C$\+contramodule $\Hom_k(\C,V)$ with a finite-dimensional vector
space of generators~$V$.
 However, these finiteness conditions (discussed in detail
in~\cite[Section~2]{Pmc}) are sometimes too restrictive, and
\emph{quasi-finiteness} conditions are preferable.

 Let $\C$ and $\E$ be (coassociative, counital) coalgebras over
a field~$k$, and let $\E\rarrow\C$ be a coalgebra morphism.
 Then any left $\E$\+comodule $\N$ can be also considered as a left
$\C$\+comodule, with the $\C$\+coaction map defined as the composition
$\N\rarrow\E\ot_k\N\rarrow\C\ot_k\N$.
 So there is an exact, faithful functor of ``corestriction of scalars''
$\E\comodl\rarrow\C\comodl$.
 Similarly, any left $\E$\+contramodule $\Q$ can be considered as
a left $\C$\+contramodule, with the $\C$\+contraaaction map provided
by the composition $\Hom_k(\C,\Q)\rarrow\Hom_k(\E,\Q)\rarrow\Q$.
 So there is an exact, faithful functor of ``contrarestriction of
scalars'' $\E\contra\rarrow\C\contra$.

 We are interested in the particular case when $\E$ is a subcoalgebra
in~$\C$.
 In this case, the functors $\E\comodl\rarrow\C\comodl$ and
$\E\contra\rarrow\C\contra$ are fully faithful.
 We will say that a left $\C$\+comodule $\M$ \emph{is a left\/
$\E$\+comodule} if it belongs to the essential image of the functor
$\E\comodl\rarrow\C\comodl$.
 This means that the image of the coaction map $\M\rarrow\C\ot_k\M$
is contained in the subspace $\E\ot_k\M\subset\C\ot_k\M$.
 Similarly, a left $\C$\+contramodule $\P$ is said to \emph{be a left\/
$\E$\+contramodule} if it belongs to the essential image of
the functor $\E\contra\rarrow\C\contra$.
 This means that the contraaction map $\Hom_k(\C,\P)\rarrow\P$
factorizes through the surjection $\Hom_k(\C,\P)\rarrow\Hom_k(\E,\P)$
induced by the inclusion $\E\rarrow\C$.

 Any left $\C$\+comodule $\M$ has a unique maximal subcomodule which
is an $\E$\+comod\-ule.
 This subcomodule, denoted by ${}_\E\M\subset\M$, can be constructed
as the full preimage of the subspace $\E\ot_k\M\subset\C\ot_k\M$
under the coaction map $\M\rarrow\C\ot_k\M$.
 The functor $\M\longmapsto{}_\E\M$ is right adjoint to the functor
of corestriction of scalars $\E\comodl\rarrow\C\comodl$ discussed above.
 Similarly, any left $\C$\+contramodule $\P$ has a unique maximal
quotient contramodule which is an $\E$\+contramodule.
 This quotient contramodule, denoted by ${}^\E\P$, can be constructed
as the cokernel of the composition $\Hom_k(\C/\E,\P)\rarrow\Hom_k(\C,\P)
\rarrow\P$ of the injective map $\Hom_k(\C/\E,\P)\rarrow\Hom_k(\C,\P)$
induced by the natural surjection $\C\rarrow\C/\E$ with the contraaction
map $\Hom_k(\C,\P)\rarrow\P$.
 The functor $\P\longmapsto{}^\E\P$ is left adjoint to the functor of
contrarestriction of scalars $\E\contra\rarrow\C\contra$ discussed
above.
 The maximal $\E$\+subcomodule of a right $\C$\+comodule $\M$ is denoted
by $\M_\E\subset\M$.

 A left $\C$\+comodule $\M$ is said to be \emph{quasi-finitely
cogenerated} (or simply \emph{quasi-finite}) if the vector space
of $\C$\+comodule homomorphisms $\Hom_\C(\L,\M)$ is finite-dimensional
for any finite-dimensional left $\C$\+comodule~$\L$.
 This definition goes back to the classical paper of
Takeuchi~\cite{Tak}; later expositions can be found, e.~g.,
in~\cite{GTNT} and~\cite[Section~2]{Ppc}.
 A left $\C$\+comodule $\M$ is quasi-finitely cogenerated if and only
if the left $\E$\+comodule ${}_\E\M$ is finite-dimensional for every
finite-dimensional subcoalgebra $\E\subset\C$ \cite[Lemma~2.1]{Ppc}.
 Similarly, a left $\C$\+contramodule $\P$ is said to be
\emph{quasi-finitely generated} (or just \emph{quasi-finite}) if
the vector space of $\C$\+contramodule homomorphisms $\Hom^\C(\P,\fL)$
is finite-dimensional for any finite-dimensional left
$\C$\+contramodule $\fL$, or equivalently if the left $\E$\+contramodule
${}^\E\P$ is finite-dimensional for every finite-dimensional
subcoalgebra $\E\subset\C$ \cite[Lemma~2.5]{Ppc}.

\begin{lem} \label{psi-pseudocompact}
 For any quasi-finitely cogenerated left comodule\/ $\M$ over
a coalgebra\/ $\C$, there is a natural, functorially defined
pseudo-compact left\/ $\C\spcheck$\+module structure on
the $k$\+vector space $\Hom_\C(\C,\M)$ such that the left\/
$\C$\+contramodule structure of\/ $\Psi_\C(\M)=\Hom_\C(\C,\M)$
discussed in Section~\ref{cotensor-contratensor-subsecn} underlies
this pseudo-compact left\/ $\C\spcheck$\+module structure.
\end{lem}

\begin{proof}
 In fact, for any coalgebra $\D$ and any $\C$\+$\D$\+bicomodule $\K$
the left $\D$\+contra\-module structure of the Hom space
$\Hom_\C(\K,\M)$ for a quasi-finitely cogenerated left $\C$\+comodule
$\M$ is associated with a certain functorially defined
pseudo-compact left $\D\spcheck$\+module structure.
 This is essentially the result of~\cite[Section~1.7]{Tak}, where
a right $\D$\+comodule denoted by $h_{\C\textrm-}(\M,\K)$ is constructed
so that $\Hom_\C(\K,\M)=h_{\C\textrm-}(\M,\K)\spcheck$.
 As no contramodules are mentioned in~\cite{Tak}, we give a sketch of
the argument in the situation at hand (with $\K=\C=\D$).

 The coalgebra $\C$ is the union of its finite-dimensional subcoalgebras
$\E$, that is $\C=\varinjlim_\E\E$.
 Hence $\Hom_\C(\C,\M)=\varprojlim_\E\Hom_\C(\E,\M)$.
 The left $\C$\+contramodule structure on $\Hom_\C(\C,\M)$ (induced
by the right $\C$\+comodule structure on~$\C$) is
the projective limit of the left $\C$\+contramodule structures
on $\Hom_\C(\E,\M)$, which in turn arise from left $\E$\+contramodule
structures (induced by the right $\E$\+comodule structure on~$\E$).
 By assumption, the $k$\+vector space $\Q=\Hom_\C(\E,\M)$ is
finite-dimensional.
 Any finite-dimensional left contramodule $\Q$ over a finite-dimensional
coalgebra $\E$ is dual to a finite-dimensional right $\E$\+comodule
$\Q\spcheck$; so the formula
$$
 \Hom_\C(\C,\M)=\varprojlim\nolimits_\E\Hom_\C(\E,\M)=
 (\varinjlim\nolimits_\E \Hom_\C(\E,\M)\spcheck)\spcheck
$$
represents $\Hom_\C(\C,\M)$ as the dual vector space to a certain
right $\C$\+comodule $\N=\varinjlim_\E\Hom_\C(\E,\M)\spcheck$.
 For any right $\C$\+comodule $\N$, the dual vector space $\N\spcheck$
is a pseudo-compact left $\C\spcheck$\+module.
 In other words, any finite-dimensional contramodule $\Q$ over
a finite-dimensional coalgebra $\E$ is naturally a pseudo-compact
$\E\spcheck$\+module, and the projective limit of pseudo-compact
modules is a pseudo-compact module.
\end{proof}

\begin{lem} \label{qf-co-generated-dualization}
 A right comodule\/ $\N$ over a coalgebra $\C$ is quasi-finitely
cogenerated if and only if the left contramodule\/
$\N\spcheck=\Hom_k(\N,k)$ over\/ $\C$ is quasi-finitely generated.
\end{lem}

\begin{proof}
 For any subcoalgebra $\E\subset\C$ and any $k$\+vector space $V$, there
is a natural isomorphism of left $\E$\+contramodules
$$
 {}^\E\Hom_k(\N,V)\simeq\Hom_k(\N_\E,V),
$$
as one can see, e.~g., from the constructions of the $\E$\+comodule
${}_\E\M$ and the $\E$\+contra\-module ${}^\E\P$ above.
 In particular, the left $\E$\+contramodule ${}^\E(\N\spcheck)\simeq
(\N_\E)\spcheck$ is finite-dimensional if and only if the right
$\E$\+comodule $\N_\E$ is.
\end{proof}

\subsection{Quasi-finitely copresented comodules and
quasi-finitely presented contramodules}
\label{qf-co-presented-subsecn}
 Let $\C$ be a coalgebra over a field~$k$.
 It is well-known that any quasi-finitely cogenerated $\C$\+comodule is
a subcomodule of a quasi-finitely cogenerated injective $\C$\+comodule
(see, e.~g., \cite[Lemma~2.2(b)]{Ppc}).
 Similarly, any quasi-finitely generated $\C$\+contramodule is
a quotient contramodule of a quasi-finitely generated projective
$\C$\+contramodule~\cite[Lemma~2.6(b)]{Ppc}.

 A $\C$\+comodule is said to be \emph{quasi-finitely copresented} if it
is the kernel of a morphism between two quasi-finitely cogenerated
injective $\C$\+comodules.
 A $\C$\+contramodule is said to be \emph{quasi-finitely presented} if
it is the cokernel of a morphism between two quasi-finitely generated
projective $\C$\+contramodules.
 We denote the full subcategory of quasi-finitely copresented left
$\C$\+comodules by $\C\comodl_\qfc\subset\C\comodl$, the full
subcategory of quasi-finitely copresented right $\C$\+comodules by
$\comodrqfc\C\subset\comodr\C$, and the full subcategory of
quasi-finitely presented left $\C$\+contramodules by
$\C\contra_\qfp\subset\C\contra$.

 Obviously, an injective $\C$\+comodule is quasi-finitely cogenerated
if and only if it is quasi-finitely copresented; and a projective
$\C$\+contramodule is quasi-finitely generated if and only if it is
quasi-finitely presented.

\begin{prop} \label{qf-co-contra}
 The equivalence between the additive categories of injective left\/
$\C$\+comodules and projective left\/ $\C$\+contramodules%
~(\ref{phi-psi-C-adjunction}\+-\ref{underived-co-contra})
$$
 \Psi_\C\:\C\comodl^\inj\.\simeq\.\C\contra^\proj:\!\Phi_\C
$$
restricts to an equivalence between the full subcategories of
quasi-finitely cogenerated injective left\/ $\C$\+comodules and
quasi-finitely generated projective left\/ $\C$\+contramodules,
$$
 \C\comodl_\qfc^\inj\.\simeq\.\C\contra_\qfp^\proj.
$$
\end{prop}

\begin{proof}
 For any subcoalgebra $\E\subset\C$, there is a commutative diagram of
additive functors and equivalences~\cite[Section~7.1.4]{Psemi},
\cite[Section~5.4]{Pkoszul}
$$
\begin{diagram}
\node{\Psi_\C\:\C\comodl}\arrow{e,=}\arrow{s}
\node{\C\contra:\!\Phi_\C}\arrow{s} \\
\node{\Psi_\E\:\E\comodl}\arrow{e,=}
\node{\E\contra:\!\Phi_\E}
\end{diagram}
$$
where the leftmost vertical functor takes a left $\C$\+comodule $\M$
to the left $\E$\+comodule ${}_\E\M$, while the rightmost vertical
functor takes a left $\C$\+contramodule $\P$ to the left
$\E$\+contramodule ${}^\E\P$.
 It remains to observe that over a finite-dimensional coalgebra $\E$
the functors $\Psi_\E=\Hom_\E(\E,{-})$ and $\Phi_\E=\E\ocn_\E{-}$
take finite-dimensional $\E$\+comodules to finite-dimensional
$\E$\+contramodules and vice versa.
\end{proof}

\begin{prop} \label{qf-co-pres-duality}
 The dualization functor
$$
 \N\longmapsto\N\spcheck=\Hom_k(\N,k)\:\comodr\C\lrarrow\C\contra
$$
restricts to an anti-equivalence between the full subcategories of
quasi-finitely cogenerated injective right\/ $\C$\+comodules and
quasi-finitely generated projective left\/ $\C$\+contramodules,
$$
 (\comodrqfcinj\C)^\sop\.\simeq\.\C\contra_\qfp^\proj.
$$
 Moreover, the same functor restricts to an anti-equivalence between
the full subcategories of quasi-finitely copresented right\/
$\C$\+comodules and quasi-finitely presented left\/ $\C$\+contramodules,
$$
 (\comodrqfc\C)^\sop\.\simeq\.\C\contra_\qfp.
$$
 Thus any quasi-finitely presented left\/ $\C$\+contramodule carries
a natural, functorially defined structure of a pseudo-compact left\/
$\C\spcheck$\+module.
\end{prop}

\begin{proof}
 This is~\cite[Proposition~2.8(a)]{Ppc}.
 Essentially, one first proves the first assertion and then deduces
the second one by passing to the (co)kernels of morphisms.
 The particular case for finitely (co)presented co/contramodules in
place of the quasi-finitely (co)presented ones can be found
in~\cite[Proposition~2.9(a)]{Pmc}.
\end{proof}

\subsection{Coextension and contraextension of scalars}
\label{extension-of-scalars-subsecn}
 Let $\C\rarrow\D$ be a morphism of coalgebras.
 Then, according to the discussion in
Section~\ref{quasi-finite-subsecn}, any left $\C$\+comodule can be
considered as a left $\D$\+comodule and any left $\C$\+contramodule
can be considered as a left $\D$\+contramodule.
 So there are exact, faithful functors of corestriction of scalars
$\C\comodl\rarrow\D\comodl$ and contrarestriction of scalars
$\C\contra\rarrow\D\contra$.

 The functor of corestriction of scalars $\C\comodl\rarrow\D\comodl$
has a right adjiont functor of ``coextension of scalars''
\cite[Section~4.8]{Pkoszul}, \cite[Section~7.1.2]{Psemi} assigning
to a left $\D$\+comodule $\N$ the left $\C$\+comodule ${}_\C\N$,
which can be constructed as the cotensor product ${}_\C\N=\C\oc_\D\N$
(see the discussion of cotensor products in
Section~\ref{cotensor-contratensor-subsecn}).
 The coextension of scalars can be also constructed as the unique
left exact functor $\D\comodl\rarrow\C\comodl$ taking the cofree left
$\D$\+comodule $\D\ot_kV$ to the cofree left $\C$\+comodule $\C\ot_kV$
for every $k$\+vector space~$V$.

 Similarly, the functor of contrarestriction of scalars $\C\contra
\rarrow\D\contra$ has a left adjoint functor of ``contraextension of
scalars'' \cite[Section~4.8]{Pkoszul}, \cite[Section~7.1.2]{Psemi}
assigning to a left $\D$\+contramodule $\Q$ the left
$\C$\+contramodule ${}^\C\Q$, which can be constructed as the vector
space of cohomomorphisms ${}^\C\Q=\Cohom_\D(\C,\Q)$ (see
Section~\ref{cohom-subsecn}) with the left $\C$\+contramodule structure
induced by the right $\C$\+comodule structure on~$\C$.
 The contraextension of scalars can be also constructed as the unique
right exact functor $\D\contra\rarrow\C\contra$ taking the free left
$\D$\+contramodule $\Hom_k(\D,V)$ to the free left $\C$\+contramodule
$\Hom_k(\C,V)$ for every $k$\+vector space~$V$.

 The following two lemmas will be useful in
Section~\ref{contra-admissibility-subsecn}.

\begin{lem} \label{extension-scalars-presentability}
 Let $\C\rarrow\D$ be a morphism of coalgebras over a field~$k$.
 Then \par
\textup{(a)} the functor of coextension of scalars\/ $\D\comodl
\rarrow\C\comodl$ takes quasi-finitely cogenerated\/ $\D$\+comodules to quasi-finitely cogenerated\/ $\C$\+comodules; \par
\textup{(b)} the functor of coextension of scalars\/ $\D\comodl
\rarrow\C\comodl$ takes quasi-finitely copresented\/ $\D$\+comodules to
quasi-finitely copresented\/ $\C$\+comodules; \par
\textup{(c)} the functor of contraextension of scalars\/
$\D\contra\rarrow\C\contra$ takes quasi-finitely generated\/
$\D$\+contramodules to quasi-finitely generated\/ $\C$\+contramodules;
\par
\textup{(d)} the functor of contraextension of scalars\/
$\D\contra\rarrow\C\contra$ takes quasi-finitely presented\/
$\D$\+contramodules to quasi-finitely presented\/
$\C$\+contramodules.
\end{lem}

\begin{proof}
 Part~(a): let $\N$ be a quasi-finitely cogenerated left $\D$\+comodule
and $\L$ be a finite-dimensional left $\C$\+comodule.
 Then $\L$ is also a finite-dimensional left $\D$\+comodule, so
the $k$\+vector space $\Hom_\D(\L,\N)$ is finite-dimensional.
 Now the adjunction isomorphism $\Hom_\C(\L,{}_\C\N)\simeq
\Hom_\D(\L,\N)$ shows that the vector space $\Hom_\C(\L,{}_\C\N)$ is
finite-dimensional.

 Part~(b): we observe that the functor of coextension of scalars
$\N\longmapsto{}_\C\N$ is right adjoint to an exact functor, hence
it is left exact and takes injective objects to injective objects. 
 Now let $\N$ be a quasi-finitely copresented left $\D$\+comodule;
so $\N$ is the kernel of a morphism of quasi-finitely cogenerated
injective left $\D$\+comodules $\I\rarrow\J$.
 Then ${}_\C\N$ is the kernel of the induced morphism of injective
left $\C$\+comodules ${}_\C\I\rarrow{}_\C\J$.
 By part~(a), the $\C$\+comodules ${}_\C\I$ and ${}_\C\J$ are also
quasi-finitely cogenerated.

 The proofs of parts~(c) and~(d) are dual-analogous to (a) and~(b),
respectively.
\end{proof}

\begin{lem} \label{cobar-fragment-lemma}
\textup{(a)} For any coalgebra morphism\/ $\C\rarrow\D$ and any left\/
$\C$\+comodule\/ $\M$, there is a natural exact sequence of left\/
$\C$\+comodules
\begin{equation} \label{relative-cobar-fragment}
 0\lrarrow\M\lrarrow\C\oc_\D\M\lrarrow\C\oc_\D\C\oc_\D\M.
\end{equation}
 Here the left\/ $\C$\+comodule\/ $\C\oc_\D\M$ can be obtained by
applying to\/ $\M$ the composition of the functors of corestriction and
coextension of scalars\/ $\C\comodl\rarrow\D\comodl\rarrow\C\comodl$.
 The left\/ $\C$\+comodule\/ $\C\oc_\D\C\oc_\D\M$ is obtained by applying
the same composition of functors twice to the left\/
$\C$\+comodule\/~$\M$. \par
\textup{(b)} For any coalgebra morphism\/ $\C\rarrow\D$ and any left\/
$\C$\+contramodule\/ $\P$, there is a natural exact sequence of left\/
$\C$\+contramodules
\begin{multline*} %\label{relative-contrabar-fragment}
 \Cohom_\D(\C\oc_\D\C,\>\P)\simeq\Cohom_\D(\C,\.\Cohom_\D(\C,\P)) \\
 \lrarrow\Cohom_\D(\C,\P)\lrarrow\P\lrarrow0.
\end{multline*}
 Here the left\/ $\C$\+contramodules\/ $\Cohom_\D(\C,\P)$ and\/
$\Cohom_\D(\C,\Cohom_\D(\C,\P))$ can be obtained by applying to\/ $\P$
the composition of the functors of contrarestriction and contraextension
of scalars\/ $\C\contra\rarrow\D\contra\rarrow\C\contra$
(one or two times).
\end{lem}

\begin{proof}
 Part~(a): in the sequence~\eqref{relative-cobar-fragment}, the map
$\M\rarrow\C\oc_\D\M$ comes from the left $\C$\+coaction map
$\M\rarrow\C\ot_k\M$, whose image is contained in the subspace
$\C\oc_\D\M\subset\C\ot_k\M$.
 In fact, the left coaction $\M\rarrow\C\ot_k\M$ is an injective map
whose image is equal to $\C\oc_\C\M\subset\C\oc_\D\M\subset\C\ot_k\M$.

 In particular, the comultiplication $\C\rarrow\C\ot_k\C$ is
an injective map whose image is equal to $\C\oc_\C\C\subset\C\oc_\D\C
\subset\C\ot_k\C$; so there is also a map $\C\rarrow\C\oc_\D\C$.
 The map $\C\oc_\D\M\rarrow\C\oc_\D\C\oc_\D\M$ in
the sequence~\eqref{relative-cobar-fragment} is the difference of two
maps $\C\oc_\D\M\rightrightarrows\C\oc_\D\C\oc_\D\M$, one of which is
induced by the map $\C\rarrow\C\oc_\D\C$ and the other one by
the map $\M\rarrow\C\oc_\D\M$.

 The assertion that the composition of the two maps
in~\eqref{relative-cobar-fragment} vanishes is a restatement of
the coassociativity equation for the left $\C$\+coaction in~$\M$.
 The assertion that the sequence~\eqref{relative-cobar-fragment}
is exact simply means that $\M$ is the kernel of the map
$\C\oc_\D\M\rarrow\C\oc_\D\C\oc_\D\M$.
 Now we have $\C\oc_\D\M\subset\C\ot_k\M$ and $\C\oc_\D\C\oc_\D\M
\subset\C\ot_k\C\ot_k\M$, so it suffices to check that $\M$ is
the kernel of the difference of two maps $\C\ot_k\M
\rightrightarrows\C\ot_k\C\ot_k\M$ induced by the coaction and
comultiplication maps.
 The latter assertion is a restatement of the natural isomorphism
$\M\simeq\C\oc_\C\M$ \cite[Corollary~2.2]{EM},
\cite[Section~1.2.1]{Psemi}.

 Part~(b) is dual-analogous and uses the natural isomorphism
$\P\simeq\Cohom_\C(\C,\P)$ \cite[Section~3.2.1]{Psemi}.
 We omit the details.
\end{proof}

\subsection{Quasi-contra-Noetherian coalgebras}
\label{quasi-contra-Noetherian-subsecn}
 For any coalgebra $\C$, the class of all quasi-finitely copresented
right $\C$\+comodules is closed under the kernels and extensions in
$\comodr\C$ \cite[Lemma~2.4(a\+b)]{Ppc}.
 Similarly, the class of all quasi-finitely presented left
$\C$\+contramodules is closed under the cokernels and extensions in
$\C\contra$ \cite[Lemma~2.7(a\+b)]{Ppc}.
 A coalgebra $\C$ is called \emph{right quasi-cocoherent} if the class of
all quasi-finitely copresented right $\C$\+comodules is closed under
cokernels in $\comodr\C$, or equivalently (in view of
Proposition~\ref{qf-co-pres-duality}), the class of quasi-finitely
presented left $\C$\+contramodules is closed under kernels in
$\C\contra$.
 Over a right quasi-cocoherent coalgebra $\C$, the categories of
quasi-finitely copresented right comodules and quasi-finitely
presented left contramodules are abelian~\cite[Section~2]{Ppc}.

 Due to Propositions~\ref{qf-co-contra} and~\ref{qf-co-pres-duality},
left and right quasi-cocoherent coalgebras and semialgebras over them
are quite suitable for developing an involutive triangulated duality
theory (see the discussion in Sections~\ref{coalgebra-duality}\+-%
\ref{semialgebra-duality} of the introduction).
 However, the notions of quasi-finitely copresented comodules and
quasi-finitely presented contramodules are somewhat complicated.
 To simplify matters a bit, we discuss certain more restrictive
assumptions in this section.

 A coalgebra $\C$ is said to be \emph{right co-Noetherian} if every
quotient comodule of a finitely cogenerated right $\C$\+comodule is
finitely cogenerated, or equivalently, every quotient comodule of
the right $\C$\+comodule $\C$ is finitely
cogenerated~\cite[Theorem~3]{WW}.
 Finitely cogenerated right comodules over a right co-Noetherian
coalgebra form an abelian category.
 We will say that a coalgebra $\C$ is \emph{left contra-Noetherian} if
every subcontramodule of a finitely generated left $\C$\+contramodule is
finitely generated, or equivalently, every subcontramodule of the left
$\C$\+contramodule $\C\spcheck$ is finitely cogenerated.
 Finitely generated left contramodules over a left contra-Noetherian
coalgebra form an abelian category.
 Any left contra-Noetherian coalgebra $\C$ is right
co-Noetherian~\cite[Lemma~2.10(c)]{Pmc}.

 A coalgebra $\C$ is said to be \emph{right quasi-co-Noetherian} if
every quotient comodule of a quasi-finitely cogenerated right
$\C$\+comodule is quasi-finitely cogenerated, or equivalently, every
quotient comodule of the right $\C$\+comodule $\C$ is quasi-finitely
cogenerated~\cite{GTNT}, \cite[Proposition~2.3]{Ppc}.
 All quasi-finitely cogenerated right comodules over a right
quasi-co-Noetherian coalgebra $\C$ are quasi-finitely copresented, and
any such a coalgebra $\C$ is right quasi-cocoherent; so quasi-finitely
cogenerated right comodules over it form
an abelian category~\cite{GTNT}, \cite[Section~2]{Ppc}.

 We will say that a coalgebra $\C$ is \emph{left
quasi-contra-Noetherian} if every subcontramodule of a quasi-finitely
generated left $\C$\+contramodule is quasi-finitely generated, or
equivalently (in view of~\cite[Lemma~2.6(b)]{Ppc}), every
subcontramodule of a quasi-finitely generated projective left
$\C$\+contramodule is quasi-finitely generated.
 All quasi-finitely generated left contramodules over a left
quasi-contra-Noetherian coalgebra $\C$ are quasi-finitely presented, and
any such a coalgebra $\C$ is right quasi-cocoherent; so quasi-finitely
generated left contramodules over it form an abelian category.
 In fact, the following lemma holds.

\begin{lem}
 Any left quasi-contra-Noetherian coalgebra is right quasi-co-Noetherian.
\end{lem}

\begin{proof}
 Let $\C$ be a left quasi-contra-Noetherian coalgebra over~$k$.
 Let $\N$ be a quasi-finitely cogenerated right $\C$\+comodule and $\L$
be a quotient comodule of~$\N$.
 Then $\N\spcheck$ is a left $\C$\+contramodule and $\L\spcheck$ is
a subcontramodule of $\N\spcheck$.
 By Lemma~\ref{qf-co-generated-dualization}, the left $\C$\+contramodule
$\N\spcheck$ is quasi-finitely generated.
 Since $\C$ is left quasi-contra-Noetherian, it follows that the left
$\C$\+contramodule $\L\spcheck$ is quasi-finitely generated.
 Applying Lemma~\ref{qf-co-generated-dualization} again, we can conclude
that the right $\C$\+comodule $\L$ is quasi-finitely cogenerated.
 Thus the coalgebra $\C$ is right quasi-co-Noetherian.
\end{proof}

 A (coassicoative, counital) coalgebra $\C$ over a field~$k$ is said to
be \emph{cosemisimple} if the abelian category of left $\C$\+comodules
is semisimple, or equivalently, the abelian category of right
$\C$\+comodules is semisimple, or equivalently, the abelian category of
left $\C$\+contramodules is semisimple~\cite[Section~4.5]{Pkoszul}.
 A coalgebra is cosemisimple if and only if it is the sum of its
cosimple subcoalgebras (where a coalgebra is called \emph{cosimple} if
it has no nonzero proper subcoalgebras)~\cite[Section~4.3]{Pkoszul}.
 Any coalgebra $\C$ has a unique maximal cosemisimple subcoalgebra
$\C^\ss$, which can be constructed as the (direct) sum of all cosimple
subcoalgebras of~$\C$.
 We refer to~\cite[Chapters~1\+-2 and~9]{Swe}
and~\cite[Section~A.2]{Psemi} for the background material.

\begin{exs} \label{coalgebras-quasi-finiteness-examples}
 (1)~Any cosemisimple coalgebra $\C$ is left and right
quasi-contra-Noetherian.
 Indeed, any subcontramodule of a quasi-finitely generated
$\C$\+contramodule $\P$ is at the same time a quotient
contramodule of $\P$, hence also quasi-finitely generated
(as the class of quasi-finitely generated contramodules over any
coalgebra is closed under quotients). {\hbadness=1700\par}

\smallskip
 (2)~Let $\C$ be a coalgebra whose maximal cosemisimple subcoalgebra
$\C^\ss$ is finite-dimensional (e.~g., $\C$ is conilpotent, which
means that $\C^\ss=k$; see the discussion in
Section~\ref{full-and-faithful-subsecn}).
 Then any quasi-finitely cogenerated $\C$\+comodule is finitely
cogenerated~\cite[Lemma~2.2(e)]{Pmc}, and similarly, any quasi-finitely
generated $\C$\+contramodule is finitely
generated~\cite[Lemma~2.5(e)]{Pmc}.
 So the quasi-finiteness conditions on $\C$\+comod\-ules and
$\C$\+contramodules are equivalent to the similarly named
finiteness conditions.

\smallskip
 (3)~Let $\C$ be a coalgebra whose dual algebra $\C\spcheck$ is left
Noetherian (as an abstract associative ring).
 Then the coalgebra $\C^\ss$ is
finite-dimensional~\cite[Section~2]{Pmc}, so (2)~applies.
 Furthermore, by~\cite[Lemma~2.10(b)]{Pmc}, the coalgebra $\C$ is
\emph{right Artinian} in the sense of~\cite[Section~2]{Pmc}, and
it follows that $\C$ is left contra-Noetherian~\cite[Lemmas~2.6(a)
and~2.10(a)]{Pmc}.
 In view of the discussion in~(2), we can conclude that the coalgebra
$\C$ is also left quasi-contra-Noetherian.
\end{exs}

\subsection{Admissible smooth modules and contraadmissible
contramodules} \label{contra-admissibility-subsecn}
 Let $G$ be locally profinite group, $H\subset G$ be a compact open
subgroup, and $k$~be a field.
 A smooth $G$\+module $\bM$ over~$k$ is said to be \emph{admissible}
if, for any (compact) open subgroup $F\subset G$, the $k$\+vector subspace
of $F$\+invariant elements $\bM^F\subset\bM$ is finite-dimensional.
 Clearly, a smooth $G$\+module is admissible if and only if it is
admissible as a smooth $H$\+module.

 Here is the dual-analogous definition for $G$\+contramodules over~$k$.
 Given a contramodule $\bP$ over a (locally) profinite group $F$, its
vector space of \emph{contrainvariants} $\bP_F$ is defined as
the (unique) maximal quotient $F$\+contramodule of $\bP$ with
the trivial contraaction of~$F$.
 Here the trivial contraaction of $F$ in a vector space $U$ is
the map $U[[F]]\rarrow U$ taking a $U$\+valued measure~$\mu$ on $F$ to
the vector $\mu(F)\in U$ which the measure~$\mu$ assigns to the whole
topological space~$F$.
 This rule defines a $F$\+contramodule structure on an arbitrary
$k$\+vector space~$U$ (cf.\ the definition of a $G$\+contramodule in
Section~\ref{G-contramodules-subsecn}), called the \emph{trivial}
$F$\+contramodule structure.

 The functor of contrainvariants $\bP\longmapsto\bP_F\:F\contra_k
\rarrow k\vect$ is left adjoint to the functor assigning to any
$k$\+vector space $U$ the vector space $U$ endowed with the trivial
$F$\+contramodule structure.
 Explicitly, the vector space $\bP_F$ is the cokernel of the restriction
$\pi_\bP|_{\bP[[F]]_0}\:\bP[[F]]_0\rarrow\bP$, where $\bP[[F]]_0\subset
\bP[[F]]$ is the subspace consisting of all the measures~$\mu$ for which
the measure of the whole space $F$ vanishes, $\mu(F)=0$ (and
$\pi_\bP\:\bP[[F]]\rarrow\bP$ is the contraaction map).
 
 We will say that a $G$\+contramodule $\bP$ over~$k$ is
\emph{contraadmissible} if, for any (compact) open subgroup $F\subset G$,
the $k$\+vector space of $F$\+contrainvariants $\bP_F$ is
finite-dimensional.
 Clearly, it suffices to check this condition for ``small enough''
open subgroups $F\subset G$, as for any open subgroups $F'\subset F''
\subset G$ the vector space $\bP_{F''}$ is a quotient space
of~$\bP_{F'}$.
 So a $G$\+contramodule $\bP$ is contraadmissible if and only if it is
contraadmissible as an $H$\+contramodule.

 We recall that for any profinite group $H$ and any field~$k$,
the inverse element map $H\rarrow H$ induces an isomorphism
$\C\simeq\C^\rop$ between the coalgebra $\C=k(H)$ and its opposite
coalgebra.
 So there is no difference between left and right $\C$\+comodules.
 
\begin{lem} \label{admissible-quasi-finite}
 Let $H$ be a profinite group and $k$~be a field.  Then \par
\textup{(a)} a discrete $H$\+module over~$k$ is admissible if and only
if it is a quasi-finitely cogenerated comodule over the coalgebra
$k(H)$; \par
\textup{(b)} an $H$\+contramodule over~$k$ is contraadmissible if and
only if it is a quasi-finitely generated contramodule over the coalgebra
$k(H)$.
\end{lem}

\begin{proof}
 For any open normal subgroup $F\subset H$, the coalgebra $\E=k(H/F)$ is
a finite-dimensional subcoalgebra in $\C=k(H)$.
 Conversely, for any finite-dimensional subcoalgebra $\E\subset k(H)$
there exists an open normal subgroup $F\subset H$ such that
$\E\subset k(H/F)\subset k(H)$.
 It remains to observe that for the subcoalgebra $\E=k(H/F)$, for any
discrete $H$\+module $\M$ over~$k$, and for any $H$\+contramodule $\P$
over~$k$, one has ${}_\E\M=\M^F$ and ${}^\E\P=\P_F$.
 In other words, the maximal $k(H/F)$\+subcomodule of
the $k(H)$\+comodule $\M$ is the subspace of $F$\+invariants
$\M^F\subset\M$, and the maximal quotient $k(H/F)$\+contramodule
of the $k(H)$\+contramodule $\P$ is the quotient space of
$F$\+contrainvariants $\P\twoheadrightarrow\P_F$.
\end{proof}

 Let $H$ be a profinite group and $k$~be a field.
 We will say that a discrete $H$\+module $\M$ over~$k$ is
\emph{admissibly copresented} if it is quasi-finitely copresented as
a $k(H)$\+comodule.
 In view of Lemma~\ref{admissible-quasi-finite}(a), this means that
$\M$ is the kernel of a morphism between two admissible injective
discrete $H$\+modules over~$k$.
 Similarly, we will say that an $H$\+contramodule $\P$ over~$k$ is
\emph{contraadmissibly presented} if it is quasi-finitely presented as
a $k(H)$\+contramodule.
 In view of Lemma~\ref{admissible-quasi-finite}(b), this means that
$\P$ is the cokernel of a morphism between two contraadmissible
projective $H$\+contramodules over~$k$.

\begin{prop} \label{adm-pres-restriction-to-subgroup}
 Let $H$ be a profinite group, $F\subset H$ be an open subgroup, and
$k$~be a field.  Then\par
\textup{(a)} a discrete $H$\+module over~$k$ is admissibly copresented
if and only if it is admissibly copresented as a discrete $F$\+module;
\par
\textup{(b)} an $H$\+contramodule over~$k$ is contraadmissibly presented
if and only if it is contraadmissibly presented as a $F$\+contramodule.
\end{prop}

\begin{proof}
 Consider two coalgebras $\C=k(H)$ and $\D=k(F)$.
 Then there is a natural surjective coalgebra morphism $\C\rarrow\D$
assigning to a locally constant function $H\rarrow k$ its restriction
to~$F$.
 A key observation is that $\C$ is an injective (in fact, cofree)
left $\D$\+comodule (as well as right $\D$\+comodule).
 Here, in the language of group representations, considering
a $\C$\+comodule as a $\D$\+comodule means restricting the group
action from $H$ to $F$; and similarly with the contramodules.

 Since $\C$ is an injective $\D$\+comodule, any injective $\C$\+comodule
is also injective as a $\D$\+comodule, and any projective
$\C$\+contramodule is also projective as a $\D$\+contramodule.
 Furthermore, in view of Lemma~\ref{admissible-quasi-finite} and
the above discusssion of (contra)admissibility, any quasi-finitely
cogenerated $\C$\+comodule is also a quasi-finitely cogenerated
$\D$\+comodule, and any quasi-finitely generated $\C$\+contramodule is
also a quasi-finitely generated $\D$\+contramodule.

 Now let $\M$ be a quasi-finitely copresented left $\C$\+comodule; so
$\M$ is the kernel of a morphism of quasi-finitely cogenerated injective
$\C$\+comodules $\I\rarrow\J$.
 Then the left $\D$\+comodule $\M$ is the kernel of the morphism of
quasi-finitely cogenerated injective $\D$\+comodules $\I\rarrow\J$.
 Hence $\M$ is a quasi-finitely copresented left $\D$\+comodule.
 This proves the ``only if'' implication in part~(a); and the proof of
the ``only if'' implication in part~(b) is similar.

 To prove the ``if'' implication in part~(a), consider the functor
of coextension of scalars assigning to a left $\D$\+comodule $\N$
the left $\C$\+comodule ${}_\C\N=\C\oc_\D\N$, as discussed in
Section~\ref{extension-of-scalars-subsecn}.
 In the language of smooth representations (in the situation at hand
with $\C=k(H)$ and $\D=k(F)$), the coextension of scalars is
the functor of induced representation from $F$ to $H$, that is
${}_\C\N=\operatorname{ind}_F^H\N$.

 Let $\M$ be a left $\C$\+comodule that is quasi-finitely copresented
as a left $\D$\+comodule.
 Then, by Lemma~\ref{extension-scalars-presentability}(b), the left
$\C$\+comodule ${}_\C\M=\C\oc_\D\M$ is quasi-finitely copresented.
 According to the ``only if'' implication, which we have already proved,
it follows that $\C\oc_\D\M$ is also quasi-finitely copresented as
a $\D$\+comodule.
 Applying Lemma~\ref{extension-scalars-presentability}(b) again, we
see that the left $\C$\+comodule ${}_\C({}_\C\M)=\C\oc_\D\C\oc_\D\M$ is
quasi-finitely copresented, too.
 By Lemma~\ref{cobar-fragment-lemma}(a), we can conclude that $\M$ is
the kernel of a morphism between two quasi-finitely copresented left
$\C$\+comodules.
 It remains to recall that the class of quasi-finitely copresented
comodules over any coalgebra $\C$ is closed under
kernels~\cite[Lemma~2.4(a)]{Ppc}.

 The proof of the ``if'' implication in part~(b) is dual-analogous
and based on Lemmas~\ref{extension-scalars-presentability}(d)
and~\ref{cobar-fragment-lemma}(b).
\end{proof}

 Let $G$ be a locally profinite group, $H\subset G$ be a compact
open subgroup, and~$k$ be a field.
 We will say that a smooth $G$\+module $\bM$ over~$k$ is \emph{admissibly
copresented} if it is admissibly copresented as a discrete $H$\+module
over~$k$.
 According to Proposition~\ref{adm-pres-restriction-to-subgroup}(a), this
property of a smooth $G$\+module $\bM$ does not depend on the choice of
a compact open subgroup $H$ in the locally profinite group~$G$.
 The full subcategory of admissibly copresented smooth $G$\+modules will
be denoted by $G\smooth_{k,\acp}\subset G\smooth_k$.

 Similarly, we will say that a $G$\+contramodule $\bP$ over~$k$ is
\emph{contraadmissibly presented} if it is contraadmissibly presented as
an $H$\+contramodule over~$k$.
 According to Proposition~\ref{adm-pres-restriction-to-subgroup}(b), this
property of a $G$\+contramodule $\bP$ does not depend on the choice of
a compact open subgroup $H\subset G$.
 The full subcategory of contraadmissibly presented $G$\+contramodules
will be denoted by $G\contra_{k,\caap}\subset G\smooth_k$.

\begin{prop} \label{contra-admissible-duality}
 The dualization functor
$$
 \bM\longmapsto\bM\spcheck=\Hom_k(\bM,k)\:
 G\smooth_k\lrarrow G\contra_k
$$
restricts to an anti-equivalence between the full subcategories of
admissibly copresented smooth $G$\+modules and contraadmissibly presented
$G$\+contramodules over~$k$,
$$
 (G\smooth_{k,\acp})^\sop\.\simeq\.G\contra_{k,\caap}.
$$
\end{prop}

\begin{proof}[Sketch of proof]
 Proposition~\ref{qf-co-pres-duality}, applied to the case of
the coalgebra $\C=k(H)$, tells that the dualization functor
$$
 \M\longmapsto\M\spcheck=\Hom_k(\M,k)\:H\discr_k\rarrow H\contra_k
$$
restricts to an anti-equivalence between the full subcategories of
admissibly copresented discrete $H$\+modules and contraadmissibly
presented $H$\+contramodules over~$k$,
$$
 (H\discr_{k,\acp})^\sop\.\simeq\.H\contra_{k,\caap}.
$$
 The idea is to show that, given an admissibly copresented discrete
$H$\+module $\M$ and the dual contraadmissibly presented
$H$\+contramodule $\P=\M\spcheck$ over~$k$, the dualization functor
$\bM\longmapsto\bP=\bM\spcheck$ establishes a bijective correspondence
between the ways to extend the discrete $H$\+action in $\M$ to
a smooth $G$\+action and the ways to extend the $H$\+contraaction
in $\P$ to a $G$\+contraaction.

 Following~\cite[Section~E.1]{Psemi}, given a discrete $H$\+module $\M$
over~$k$, extending the action of $H$ in $\M$ to an action of
the group~$G$ is equivalent to defining a structure of
\emph{right semimodule} over the $\C$\+semialgebra $\bS=k(G)$ on
the right $\C$\+comodule~$\M$.
 This means specifying a \emph{right semiaction} map
$\M\oc_\C\bS\rarrow\M$, which should be a morphism of left
$\C$\+comodules satisfying the semiassociativity and semiunitality
equations.
 Here $\M\oc_\C\bS=\operatorname{ind}_H^G\M$ is the smooth $G$\+module
induced from the smooth $H$\+module~$\M$.
 Similarly, given an $H$\+contramodule $\P$ over~$k$, extending
the contraaction of $H$ in $\P$ to a contraaction of $G$ is equivalent
to defining a structure of \emph{left semicontramodule} over $\bS$ on
the left $\C$\+contramodule~$\P$.
 This means specifying a \emph{left semicontraaction} map
$\P\rarrow\Cohom_\C(\bS,\P)$, which should be a morphism of left
$\C$\+contramodules satisfying the semicontraassociativity and
semicontraunitality equations~\cite[Example~2.6]{Prev}.
 See also the discussions in
Sections~\ref{loc-profinite-group-semialgebra-subsecn}
and~\ref{smooth-contra-subsecn} above.

 Now in the situation at hand, we have $\P=\M\spcheck$ and
$\Cohom_\C(\bS,\P)=(\M\oc_\C\nobreak\bS)\spcheck$
(see~\eqref{cotensor-cohom}).
 In order to prove the desired bijection between the $\bS$\+semimodule
structures on $\M$ and the $\bS$\+semicontramodule structures on $\P$,
we use the result of~\cite[Proposition~2.8(b)]{Ppc}, claiming, in
particular, that for any coalgebra $\C$, and right $\C$\+comodule $\N$,
and any quasi-finitely cogenerated right $\C$\+comodule $\M$,
the dualization functor $\N\longmapsto\N\spcheck$ induces an isomorphism
between the Hom spaces in the categories of right $\C$\+comodules and
left $\C$\+contramodules
$$
 \Hom_{\C^\rop}(\N,\M)\simeq\Hom^\C(\M\spcheck,\N\spcheck).
$$
 Notice that the right $\C$\+comodule $\N=\M\oc_\C\bS$ does not need to
be quasi-finitely cogenerated, of course; but the right $\C$\+comodule
$\M$ is (quasi-finitely copresented, hence) quasi-finitely cogenerated
by assumption, which is sufficient
to make~\cite[Proposition~2.8(b)]{Ppc} applicable.

 It is straightforward to check that a right $\C$\+comodule map
$\M\oc_\C\bS\rarrow\M$ satisfies the above-mentioned equations from
the definition of a semimodule and only if the corresponding left
$\C$\+contramodule map $\P\rarrow\Cohom_\C(\bS,\P)$ satisfies
the above-mentioned (dual-analogous) equations from the definition
of a semicontramodule.
 This proves that our functor $(G\smooth_{k,\acp})^\sop\rarrow
G\contra_{k,\caap}$ is surjective on the isomorphism classes of objects.
 Checking that it is bijective on morphisms is another
application of~\cite[Proposition~2.8(b)]{Ppc}.
\end{proof}

 It follows from Proposition~\ref{contra-admissible-duality} that
every contraadmissibly presented $G$\+contra\-module over~$k$ carries
a natural, functorially defined pseudo-compact $G$\+module structure
(i.~e., a pseudo-compact topology making it the pseudo-compact dual
vector space to a smooth $G$\+module).

\subsection{Locally quasi-contra-Noetherian and locally
quasi-cocoherent groups} \label{loc-qf-groups-subsecn}
 In the previous section we have shown that various
quasi-finiteness/admissibility conditions on smooth modules and
contramodules do not depend on the choice of a compact open subgroup
$H$ in a given locally profinite group~$G$.
 In this section we prove similar results for local quasi-finiteness
conditions on the group $G$ itself (cf.\ the discussion in
Section~\ref{admissibly-co-presented-subsecn} of the introduction).

\begin{lem} \label{extension-of-scalars-reflects}
 Let $H$ be a profinite group, $F\subset H$ be an open subgroup, and
$k$~be a field.
 Consider the related surjective morphism\/ $\C=k(H)\rarrow k(F)=\D$ of
coalgebras over~$k$.  Then \par
\textup{(a)} a\/ $\D$\+comodule\/ $\N$ is quasi-finitely
cogenerated if and only if the\/ $\C$\+comodule ${}_\C\N$ is
quasi-finitely cogenerated; \par
\textup{(b)} a\/ $\D$\+comodule\/ $\N$ is quasi-finitely
copresented if and only if the\/ $\D$\+comodule ${}_\C\N$ is
quasi-finitely copresented; \par
\textup{(c)} a\/ $\D$\+contramodule\/ $\Q$ is quasi-finitely
generated if and only if the\/ $\C$\+contramodule\/ ${}^\C\Q$ is
quasi-finitely generated; \par
\textup{(d)} a\/ $\D$\+contramodule\/ $\Q$ is quasi-finitely
presented if and only if the\/ $\C$\+contramodule\/ ${}^\C\Q$ is
quasi-finitely presented.
\end{lem}

\begin{proof}
 In all the four parts~(a\+d), the ``only if'' implication holds for any
morphism of coalgebras $\C\rarrow\D$ by
Lemma~\ref{extension-scalars-presentability}.
 Let us prove the ``if''.

 The key observation is that the $\D$\+$\D$\+bicomodule $\C$ contains
the $\D$\+$\D$\+bicomodule $\D$ as a direct summand (or in other
words, the smooth $F\times F$\+module $k(H)$ contains the smooth
$F\times F$\+module $k(F)$ as a direct summand).
 Consequently, for any left $\D$\+comodule $\N$, applying the functor
of coextension of scalars $\N\longmapsto{}_\C\N=\C\oc_\D\nobreak\N=
\operatorname{ind}_H^G\N$ produces a $\C$\+comodule whose underlying
$\D$\+comodule contains $\N$ as a direct summand.
 Similarly, for any left $\D$\+contramodule $\Q$, applying the functor
of contraextension of scalars $\Q\longmapsto{}^\C\Q=\Cohom_\D(\C,\Q)$
produces a $\C$\+contramodule whose underlying $\D$\+contramodule
contains $\Q$ as a direct summand.

 Part~(a): assume that the $\C$\+comodule ${}_\C\N$ is quasi-finitely
cogenerated.
 As explained in the proof of
Proposition~\ref{adm-pres-restriction-to-subgroup}, it follows that
the $\D$\+comodule ${}_\C\N$ is quasi-finitely cogenerated.
 Since the class of quasi-finitely cogenerated comodules over
any coalgebra is closed under direct summands, the assertion follows.
 The argument for part~(b) is similar, using
Proposition~\ref{adm-pres-restriction-to-subgroup}(a);
and parts~(c\+d) are also similar.
\end{proof}

 Due to the isomorphism of coalgebras $\C=k(H)\simeq\C^\rop$, there is
no difference between the ``left'' and ``right'' versions of
the properties appearing in the next corollary.

\begin{cor} \label{quasi-finiteness-group-locality}
 Let $H$ be a profinite group, $F\subset H$ be an open subgroup, and
$k$~be a field.  Then \par
\textup{(a)} the ring $k[[H]]=k(H)\spcheck$ is Noetherian if and only if
the ring $k[[F]]=k(F)\spcheck$ is Noetherian; \par
\textup{(b)} the coalgebra $k(H)$ is quasi-contra-Noetherian if and only
if the coalgebra $k(F)$ is quasi-contra-Noetherian; \par
\textup{(c)} the coalgebra $k(H)$ is quasi-cocoherent if and only if
the coalgebra $k(F)$ is quasi-cocoherent.
\end{cor}

\begin{proof}
 Part~(a) is essentially well-known; cf.~\cite[Exercise~7.6]{DSMS}.
 We will prove part~(b), parts~(a) and~(c) being similar.
 The argument is a straightforward consequence of the results of
Section~\ref{contra-admissibility-subsecn} and this section.
 As above, we set $\C=k(H)$ and $\D=k(F)$.

 ``If'': let $\P$ be a quasi-finitely generated $\C$\+contramodule
and $\Q\subset\P$ be a subcontramodule.
 In other words, this means that $\P$ is a contraadmissible
$H$\+contramodule over~$k$ (see Lemma~\ref{admissible-quasi-finite}(b)).
 Hence $\P$ is also a contraadmissible $F$\+contramodule over~$k$,
i.~e., a quasi-finitely generated $\D$\+contramodule.
 By assumption, $\D$ is quasi-contra-Noetherian, so $\Q$ is
a quasi-finitely generated $\D$\+contramodule, too.
 In other words, $\Q$ is a contraadmissible $F$\+contramodule over~$k$,
hence a contraadmissible $H$\+contramodule.
 Thus $\Q$ is a quasi-finitely generated $\C$\+contramodule.

 ``Only if'': let $\P$ be a quasi-finitely generated
$\D$\+contramodule and $\Q\subset\P$ be a subcontramodule.
 Since $\C$ is an injective left $\D$\+comodule, the functor of
contraextension of scalars $\P\longmapsto{}^\C\P=\Cohom_\D(\C,\P)$
is exact.
 So ${}^\C\P$ is a $\C$\+contramodule and ${}^\C\Q\subset{}^\C\P$
is a $\C$\+subcontramodule.
 By Lemma~\ref{extension-of-scalars-reflects}(c),
the $\C$\+contramodule ${}^\C\P$ is quasi-finitely generated.
 By assumption, $\C$ is quasi-contra-Noetherian, so
the $\C$\+contramodule ${}^\C\Q$ is quasi-finitely generated, too.
 Again by Lemma~\ref{extension-of-scalars-reflects}(c),
the $\D$\+contramodule $\Q$ is quasi-finitely generated.

 In part~(c), one has to use
Proposition~\ref{adm-pres-restriction-to-subgroup}(a) and
Lemma~\ref{extension-of-scalars-reflects}(b).

 In part~(a), the implication ``if'' is a particular case
of~\cite[Corollary~1.5]{GW}.
 To prove the ``only if'', use the fact that any left
$k[[F]]$\+module $M$ is a direct summand of the left $k[[F]]$\+module
$k[[H]]\ot_{k[[F]]}M$, since the $k[[F]]$-$k[[F]]$-bimodule $k[[F]]$
is a direct summand of the $k[[F]]$-$k[[F]]$-bimodule~$k[[H]]$.
\end{proof}

 Let $G$ be a locally profinite group, $H\subset G$ be a compact open
subgroup, and $k$~be a field.
 We will say that $G$ is \emph{locally Noetherian over~$k$} if
the ring $k[[H]]=k(H)\spcheck$ is Noetherian.
 Similarly, we will say that $G$ is \emph{locally
quasi-contra-Noetherian over~$k$} if the coalgebra $k(H)$ is
quasi-contra-Noetherian, and we will say that $G$ is \emph{locally
quasi-cocoherent over~$k$} if the coalgebra $k(H)$ is
quasi-cocoherent.
 In view of Corollary~\ref{quasi-finiteness-group-locality}, these
properties of a locally profinite group $G$ do not depend on
the choice of a compact open subgroup $H\subset G$.
 When speaking of such properties in application to profinite groups
$H$, we will drop the adverb ``locally''.

 According to Example~\ref{coalgebras-quasi-finiteness-examples}\,(3),
any locally Noetherian (locally profinite) group is locally
quasi-contra-Noetherian.
 Following the discussion in
Section~\ref{quasi-contra-Noetherian-subsecn}, any locally
quasi-contra-Noetherian group is locally quasi-cocoherent
(over the same field~$k$).
 Moreover, for a locally profinite group $G$ that is
locally quasi-contra-Noetherian over a field~$k$, any admissible
smooth $G$\+module over~$k$ is admissibly copresented and
any contraadmissible $G$\+contramodule over~$k$ is contraadmissibly
presented.

\begin{exs} \label{loc-quasi-contra-Noetherian-groups}
 (1)~Following Example~\ref{coalgebras-quasi-finiteness-examples}\,(1),
any locally profinite group $G$ which has a compact open subgroup $H$
of the proorder not divisible by the characteristic of~$k$ is locally
quasi-contra-Noetherian over~$k$.
 In particular, over a field~$k$ of characteristic~$0$ any locally
profinite group is locally quasi-contra-Noetherian.

\smallskip
 (2)~According to~\cite[Corollary~7.25 and Theorem~8.32]{DSMS}, over
a field $k$~of characteristic~$p$, any $p$\+adic Lie group $G$ is
locally Noetherian.
\end{exs}

\Section{Involutive Triangulated Duality} \label{involutive-secn}

\subsection{Quasi-cocoherent coalgebras} \label{q-cocoh-coalgebras}
 Let $\C$ be a coalgebra over a field~$k$.
 Assume that the coalgebra $\C$ is left quasi-cocoherent and right
quasi-cocoherent (see Sections~\ref{quasi-finite-subsecn}\+-
\ref{qf-co-presented-subsecn} and~\ref{quasi-contra-Noetherian-subsecn}
for the relevant definitions).

 Then the category of quasi-finitely copresented left $\C$\+comodules
$\C\comodl_\qfc$ is an abelian subcategory of the abelian category
$\C\comodl$ with an exact identity embedding functor $\C\comodl_\qfc
\rarrow\C\comodl$.
 The abelian category $\C\comodl_\qfc$ has enough injective objects,
which form the full subcategory
$$
\C\comodl_\qfc\cap\C\comodl^\inj=\C\comodl_\qfc^\inj\subset\C\comodl;
$$
so the injective objects of $\C\comodl_\qfc$ are also injective
in $\C\comodl$.
 The category of quasi-finitely copresented right $\C$\+comodules
$\comodrqfc\C\subset\comodr\C$ has similar properties.

 Furthermore, the category of quasi-finitely presented left
$\C$\+contramodules $\C\contra_\qfp$ is an abelian subcategory of
the abelian category $\C\contra$ with an exact identity embedding
functor $\C\contra_\qfp\rarrow\C\contra$.
 The abelian category $\C\contra_\qfp$ has enough projective objects,
which form the full subcategory {\hbadness=2100
$$
 \C\contra_\qfp\cap\C\contra^\proj=\C\contra_\qfp^\proj\subset\C\contra;
$$
so} the projective objects of $\C\contra_\qfp$ are also projective
in $\C\contra$.

 Recall from the discussion in Section~\ref{derived-equivalence-subsecn}
that, for any coalgebra $\C$, the three abelian categories $\C\comodl$,
\,$\comodr\C$, and $\C\contra$ have the same homological dimension
(called the homological dimension of~$\C$).
 For any coalgebra~$\C$, the equivalence of additive categories
$\C\comodl^\inj\simeq\C\contra^\proj$
\,(\ref{phi-psi-C-adjunction}\+-\ref{underived-co-contra}) induces
an equivalence of the homotopy categories $\Hot(\C\comodl^\inj)
\simeq\Hot(\C\contra^\proj)$ of the additive categories of injective
left $\C$\+comodules and projective left $\C$\+contramodules.
 For a coalgebra $\C$ of finite homological dimension, this leads to
an equivalence of the derived categories
\begin{equation} \label{finite-homol-dim-co-contra}
 \boR\Psi_\C\:\sD(\C\comodl)=\Hot(\C\comodl^\inj)\simeq
 \Hot(\C\contra^\proj)=\sD(\C\contra):\!\boL\Phi_\C
\end{equation}
(cf.\ the discussion in Section~\ref{co-contra-subsecn}).

 For any left and right quasi-cocoherent coalgebra $\C$, the bounded
derived category of quasi-finitely copresented left $\C$\+comodules
is a full subcategory of the derived category of left $\C$\+comodules,
and the bounded derived category of quasi-finitely presented left
$\C$\+contramodules is a full subcategory of the derived category
of left $\C$\+contramodules,
$$
 \sD^\b(\C\comodl_\qfc)\.\subset\.\sD(\C\comodl)
 \quad\text{and}\quad
 \sD^\b(\C\contra_\qfp)\.\subset\.\sD(\C\contra).
$$
 Now let $\C$ be a left and right quasi-cocoherent coalgebra of finite
homological dimension.
 Then the equivalence of additive categories $\C\comodl^\inj_\qfc
\simeq\C\contra^\proj_\qfp$ from Proposition~\ref{qf-co-contra} shows
that the triangulated equivalence~\eqref{finite-homol-dim-co-contra}
restricts to an equivalence between the bounded derived categories of
quasi-finitely (co)presented comodules and contramodules,
\begin{multline} \label{qf-bounded-derived-co-contra}
 \boR\Psi_\C\:\sD^\b(\C\comodl_\qfc)=\Hot^\b(\C\comodl^\inj_\qfc)
 \\ \simeq
 \Hot^\b(\C\contra^\proj_\qfp)=\sD^\b(\C\contra_\qfp):\!\boL\Phi_\C.
\end{multline}

 Furthermore, following Proposition~\ref{qf-co-pres-duality},
the dualization functor $\N\longmapsto\N\spcheck$ provides
an anti-equivalence of abelian categories $(\comodrqfc\C)^\sop\simeq
\C\contra_\qfp$, inducing an anti-equivalence of the bounded derived
categories
\begin{equation} \label{C-derived-dualization-equivalence}
 \sD^\b(\comodrqfc\C)^\sop\.\simeq\.\sD^\b(\C\contra_\qfp).
\end{equation}
 Composing the triangulated anti-equivalence $\N\longmapsto\N\spcheck$
\eqref{C-derived-dualization-equivalence} with the triangulated
equivalence $\boL\Phi_\C$ \eqref{qf-bounded-derived-co-contra}, we
obtain an anti-equivalence between the bounded derived categories of
quasi-finitely copresented left and right $\C$\+comodules,
\begin{equation} \label{Delta-C-anti-equivalence}
 \Delta_\C\:\sD^\b(\comodrqfc\C)^\sop\.\simeq\.\sD^\b(\C\comodl_\qfc).
\end{equation}

\begin{prop} \label{coalgebra-two-Deltas-inverse}
 For any left and right quasi-cocoherent coalgebra\/ $\C$ of finite
homological dimension, the triangulated anti-equivalences\/
$\Delta_\C\:\sD^\b(\comodrqfc\C)^\sop\rarrow\sD^\b(\C\comodl_\qfc)$
and\/ $\Delta_{\C^\rop}\:\sD^\b(\C\comodl_\qfc)^\sop\rarrow
\sD^\b(\comodrqfc\C)$ are inverse to each other.
\end{prop}

\begin{proof}[Sketch of proof]
 More generally, for any coalgebra $\C$ of finite homological dimension
and any $k$\+vector space one can compose the contravariant
triangulated functor
$$
 \Hom_k({-},V)\:\sD(\comodr\C)^\sop\lrarrow\sD(\C\contra)
$$
with the triangulated equivalence $\boL\Phi_\C$
\eqref{finite-homol-dim-co-contra}, obtaining a contravariant
triangulated functor
$$
 \Delta_\C^V\:\sD(\comodr\C)^\sop\lrarrow\sD(\C\comodl).
$$

 Arguing similarly to Section~\ref{duality-adjunction-subsecn}
(cf.\ Section~\ref{duality-adjunction-introd}), one shows that
the contravariant functors $\Delta_\C^V$ and
$$
 \Delta_{\C^\rop}^V\:\sD(\C\comodl)^\sop\lrarrow\sD(\comodr\C)
$$
are \emph{right adjoint to each other}.
 This means that for every complex of left $\C$\+comodules $\M^\bu$
and every complex of right $\C$\+comodules $\N^\bu$ there is a natural
isomorphism of Hom spaces in the derived categories
\begin{equation} \label{coalgebra-duality-adjunction-iso}
 \Hom_{\sD(\C\comodl)}(\M^\bu,\Delta_\C^V\N^\bu)\simeq
 \Hom_{\sD(\comodr\C)}(\N^\bu,\Delta_{\C^\rop}^V\M^\bu).
\end{equation}
 Similarly to the proof of
Proposition~\ref{derived-self-adjunction-prop}, one can compute
both the left- and the right-hand side
of~\eqref{coalgebra-duality-adjunction-iso} as the degree-zero
cohomology vector space of what is denoted in~\cite{Psemi} by
$\Hom_k(\operatorname{Cotor}^\C(\M^\bu,\N^\bu),V)$
(see~\cite[Section~0.2.2]{Psemi}).
 One can also construct the adjunction morphisms $\M^\bu\rarrow
\Delta_\C^V\Delta_{\C^\rop}^V(\M^\bu)$ in $\sD(\C\comodl)$ and
$\N^\bu\rarrow\Delta_{\C^\rop}^V\Delta_\C^V(\N^\bu)$ in
$\sD(\comodr\C)$ in a way similar to the construction
in Section~\ref{duality-adjunction-subsecn}.

 Now specializing to the vector space $V=k$ and restricting to the full
subcategories $\sD^\b(\C\comodl_\qfc)\subset\sD(\C\comodl)$ and
$\sD^\b(\comodrqfc\C)\subset\sD(\comodr\C)$, we have a pair of
right adjoint contravariant functors
$\Delta_\C\:\sD^\b(\comodrqfc\C)^\sop\rarrow\sD^\b(\C\comodl_\qfc)$
and $\Delta_{\C^\rop}\:\sD^\b(\C\comodl_\qfc)^\sop\rarrow
\sD^\b(\comodrqfc\C)$.
 We already know that the latter two functors are triangulated
anti-equivalences~\eqref{Delta-C-anti-equivalence}.
 Any two adjoint equivalences are mutually inverse.
\end{proof}

\subsection{Profinite groups}
 Let $k$~be a field and $H$~be a profinite group of finite
$k$\+cohomological dimension.
 Assume further that $H$ is quasi-cocoherent over~$k$, i.~e.,
the coalgebra $\C=k(H)$ is quasi-cocoherent (see
Section~\ref{loc-qf-groups-subsecn}).

 Recall that the inverse element map $H\rarrow H$ induces an isomorphism
$\C\simeq\C^\rop$, so the categories of left and right $\C$\+comodules
are equivalent to each other.
 They are also equivalent to the category $H\discr_k$ of discrete
$H$\+modules over~$k$, while the category of (left or right)
$\C$\+contramodules is equivalent to the category $H\contra_k$ of
$H$\+contramodules over~$k$.

 We also recall the notation $H\discr_{k,\acp}\subset H\discr_k$ and
$H\contra_{k,\caap}\subset H\contra_k$ for the full subcategories of
admissibly copresented discrete $H$\+modules and contraadmissibly
presented $H$\+contramodules over~$k$ (see 
Section~\ref{contra-admissibility-subsecn} for a more general
discussion for a locally profinite group~$G$).
 Under our quasi-cocoherence assumption on the group $H$, these two
categories are abelian.
 When the profinite group $H$ is quasi-contra-Noetherian over~$k$, these
full subcategories coincide with the full subcategories of admissible
discrete $H$\+modules and contraadmissible $H$\+contramodules over~$k$,
respectively (see Section~\ref{quasi-contra-Noetherian-subsecn}).

 Specializing the discussion in Section~\ref{q-cocoh-coalgebras} to
the case of the coalgebra $\C=k(H)$, we see that the derived
equivalence
\begin{equation} \label{H-finite-co-homol-dim-co-contra}
 \boR\Psi_H\:\sD(H\discr_k)\.\simeq\.\sD(H\contra_k):\!\boL\Phi_H
\end{equation}
restricts to an equivalence between the bounded derived categories
of admissibly copresented discrete modules and contraadmissibly
presented contramodules
$\sD^\b(H\discr_{k,\acp})\subset\sD(H\discr_k)$ and
$\sD^\b(H\contra_{k,\caap})\subset\sD^\b(H\contra_k)$,
\begin{equation} \label{H-admissibly-bounded-derived-co-contra}
 \boR\Psi_H\:\sD^\b(H\discr_{k,\acp})\.\simeq\.
 \sD^\b(H\contra_{k,\caap}):\!\boL\Phi_H.
\end{equation}
 Furthermore, the dualization functor $\M\longmapsto\M\spcheck$ provides
an anti-equivalence of abelian categories $(H\discr_{k,\acp})^\sop\simeq
H\contra_{k,\caap}$ (cf.\ Proposition~\ref{contra-admissible-duality}),
inducing an anti-equivalence of the bounded derived categories
\begin{equation} \label{H-derived-dualization-equivalence}
 \sD^\b(H\discr_{k,\acp})^\sop\.\simeq\.
 \sD^\b(H\contra_{k,\caap}).
\end{equation}

 As in Section~\ref{q-cocoh-coalgebras}, we compose the triangulated
anti-equivalence $\M\longmapsto\M\spcheck$
\eqref{H-derived-dualization-equivalence} with the triangulated
equivalence $\boL\Phi_H$ \eqref{H-admissibly-bounded-derived-co-contra},
producing an auto-anti-equivalence of the bounded derived category of
admissibly copresented discrete $H$\+modules
\begin{equation} \label{Delta-H}
 \Delta_H\:\sD^\b(H\discr_{k,\acp})^\sop\simeq
 \sD^\b(H\discr_{k,\acp}).
\end{equation}

\begin{thm} \label{H-involutive-duality-thm}
 For any field~$k$ and a profinite group $H$ of finite
$k$\+cohomological dimension that is quasi-cocoherent over~$k$,
the functor\/ $\Delta_H$ is an involutive triangulated
auto-anti-equivalence of the bounded derived category of admissibly
copresented discrete $H$\+modules over~$k$.
\end{thm}

\begin{proof}
 The equivalences of categories $k(H)\comodl_\qfc\simeq H\discr_{k,\acp}
\simeq\comodrqfc k(H)$ identify the functor $\Delta_H$ with both
the functors $\Delta_\C$ and $\Delta_{\C^\rop}$ for $\C=k(H)$, so
the assertion follows from
Proposition~\ref{coalgebra-two-Deltas-inverse}.
 Alternatively, one can put $V=k$ and observe that the functor
$\Delta_H$ is a restriction of the functor $\Delta_H^k\:
\sD(H\discr_k)^\sop\rarrow\sD(H\discr_k)$ \,\eqref{Delta-G-V},
which is right self-adjoint by
Proposition~\ref{derived-self-adjunction-prop}.
 So the functor $\Delta_H$ is a right self-adjoint
auto-anti-equivalence.
 Hence the adjunction morphism $\M^\bu\rarrow\Delta_H\Delta_H(\M^\bu)$
is an isomorphism in the derived category for any complex
$\M^\bu\in\sD^\b(H\discr_{k,\acp})$, and the functor $\Delta_H$ is
involutive. \hfuzz=1.3pt
\end{proof}

\subsection{Locally profinite groups}
 Let $k$~be a field and $G$ be a locally profinite group locally of
finite $k$\+cohomological dimension (see
Section~\ref{derived-equivalence-subsecn}).
 Assume further that $G$ is locally quasi-cocoherent over~$k$
(see Section~\ref{loc-qf-groups-subsecn}).

 We consider the full subcategory $\sD^\b_\acp(G\smooth_k)\subset
\sD(G\smooth_k)$ in the derived category of smooth $G$\+modules over~$k$
consisting of all the bounded complexes of smooth $G$\+modules
\emph{with admissibly copresented cohomology modules}.
 Similarly, we also consider the full subcategory
$\sD^\b_\caap(G\contra_k)\subset\sD(G\contra_k)$ in the derived category
of $G$\+contramodules over~$k$ consisting of all the bounded complexes
of $G$\+contramodules \emph{with contraadmissibly presented
cohomology contramodules}.

 Let $H\subset G$ be a compact open subgroup of finite
$k$\+cohomological dimension.
 It follows from the construction of the triangulated equivalence
in Theorem~\ref{derived-equivalence} and the commutative 
diagrams~\eqref{G-G-prime-Phi-Psi} that the derived
equivalences~\eqref{loc-finite-cohomol-dim-derived-equivalence} for
the groups $G$ and $H$ form a commutative diagram with the forgetful
functors; see diagram~\eqref{fhd-G-H-semico-semicontra-diagram}
in Section~\ref{fhd-subsecn}.
 According to~(\ref{H-finite-co-homol-dim-co-contra}\+-%
\ref{H-admissibly-bounded-derived-co-contra}), the triangulated
equivalence in the lower line
of~\eqref{fhd-G-H-semico-semicontra-diagram} restricts to
a triangulated
equivalence~\eqref{H-admissibly-bounded-derived-co-contra}.
 Passing to the full preimages
of~\eqref{H-admissibly-bounded-derived-co-contra} with respect to
the vertical forgetful functors
in~\eqref{fhd-G-H-semico-semicontra-diagram}, we obtain a commutative
diagram of triangulated equivalences and forgetful functors
\begin{equation} \label{G-H-admissible-derived-co-contra}
\begin{diagram}
\node{\boR\Psi_G\:\sD^\b_\acp(G\smooth_k)}\arrow{s}\arrow[2]{e,=}
\node[2]{\sD^\b_\caap(G\contra_k):\!\boL\Phi_G}\arrow{s} \\
\node{\boR\Psi_H\:\sD(H\discr_{k,\acp})}\arrow[2]{e,=}
\node[2]{\sD(H\contra_{k,\caap}):\!\boL\Phi_H}
\end{diagram}
\end{equation}

 Furthermore, the dualization functor $\bM\longmapsto\bM\spcheck\:
(G\smooth_k)^\sop\rarrow G\contra_k$ takes $G\smooth_{k,\acp}$ into
$G\contra_{k,\caap}$, and therefore, induces a contravariant
triangulated functor
\begin{equation} \label{G-derived-dualization-functor}
 \sD^\b_\acp(G\smooth_k)^\sop\lrarrow\sD^\b_\caap(G\contra_k).
\end{equation}
 The functor~\eqref{G-derived-dualization-functor} forms
a commutative diagram with the triangulated
anti-equiva\-lence~\eqref{H-derived-dualization-equivalence}
and the forgetful functors,
\begin{equation} \label{G-H-derived-dualization-diagram}
\begin{diagram}
\node{\sD^\b_\acp(G\smooth_k)^\sop} \arrow{e} \arrow{s}
\node{\sD^\b_\caap(G\contra_k)} \arrow{s} \\
\node{\sD^\b(H\discr_{k,\acp})^\sop} \arrow{e,=}
\node{\sD^\b(H\contra_{k,\caap})}
\end{diagram}
\end{equation}

 Composing the horizontal functors
in~\eqref{G-H-derived-dualization-diagram} with the triangulated
equivalences $\boL\Phi_G$ and $\boL\Phi_H$
\eqref{G-H-admissible-derived-co-contra}, we obtain a commutative
diagram of triangulated functors
\begin{equation} \label{Delta-G-H}
\begin{diagram}
\node{\Delta_G\:\sD^\b_\acp(G\smooth_k)^\sop} \arrow{e} \arrow{s}
\node{\sD^\b_\acp(G\smooth_k)} \arrow{s} \\
\node{\Delta_H\:\sD^\b(H\discr_{k,\acp})^\sop} \arrow{e,=}
\node{\sD^\b(H\discr_{k,\acp})}
\end{diagram}
\end{equation}
with an equivalence in the lower horizontal line
(cf.\ diagram~\eqref{loc-profinite-duality-diagram} in
Section~\ref{loc-profinite-duality-subsecn} of the introduction).

\begin{thm} \label{G-involutive-duality-theorem}
 For any field~$k$ and a locally profinite group $G$ locally of finite
$k$\+cohomological dimension that is locally quasi-cocoherent over~$k$,
the functor~\eqref{G-derived-dualization-functor} is a triangulated
anti-equivalence.
 The functor\/ $\Delta_G$ is an involutive triangulated
auto-anti-equivalence of the bounded derived category of smooth
$G$\+modules over~$k$ with admissibly copresented cohomology modules\/
$\sD^\b_\acp(G\smooth_k)$.
\end{thm}

\begin{proof}
 As in the proof of Theorem~\ref{H-involutive-duality-thm}, we put
$V=k$ and observe that the functor $\Delta_G$ is a restriction of
the functor $\Delta_G^k\:\sD(G\smooth_k)^\sop\rarrow
\sD(G\smooth_k)$ \,\eqref{Delta-G-V}, which is right self-adjoint by
Proposition~\ref{derived-self-adjunction-prop}.
 Hence the functor $\Delta_G$ is right self-adjoint, too.
 In order to show that $\Delta_G$ is an involutive
auto-anti-equivalence, it remains to check that the adjunction morphism
$\bM^\bu\rarrow\Delta_G\Delta_G(\bM^\bu)$ is an isomorphism in
the category $\sD^\b_\acp(G\smooth_k)$ for any complex
$\bM^\bu\in\sD^\b_\acp(G\smooth_k)$.

 Indeed, one can see from the construction of the adjunction morphism
in Section~\ref{duality-adjunction-subsecn} that for any complex of
smooth $G$\+modules $\bN^\bu$ over~$k$ the forgetful functor
$\sD(G\smooth_k)\rarrow\sD(H\discr_k)$ takes the adjunction morphism
$\bN^\bu\rarrow\Delta_G^k\Delta_G^k(\bN^\bu)$ to the adjunction
morphism $\bN^\bu\rarrow\Delta_H^k\Delta_H^k(\bN^\bu)$.
 Now for a complex $\bM^\bu\in\sD^\b_\acp(G\smooth_k)$, the image of
$\bM^\bu$ under the forgetful functor is (isomorphic to) a complex from
$\sD^\b(H\discr_{k,\acp})$.
 By Theorem~\ref{H-involutive-duality-thm}, the adjunction morphism
$\bM^\bu\rarrow\Delta_H\Delta_H(\bM^\bu)$ is an isomorphism in
$\sD^\b(H\discr_{k,\acp})$.
 Since the forgetful functor $\sD(G\smooth_k)\rarrow\sD(H\discr_k)$ is
conservative, it follows that the adjunction morphism
$\bM^\bu\rarrow\Delta_G\Delta_G(\bM^\bu)$ is an isomorphism in
$\sD^\b_\acp(G\smooth_k)$.

 We have shown that the functor $\Delta_G$ is an (involutive)
auto-anti-equivalence.
 Since the functor $\boL\Phi_G$
\eqref{G-H-admissible-derived-co-contra} is an equivalence, it follows
that the functor~\eqref{G-derived-dualization-functor} is
an anti-equivalence.
\end{proof}

 In conclusion, we recall that when the group $G$ is locally
quasi-contra-Noetherian over the field~$k$, all admissible smooth
$G$\+modules over~$k$ are admissibly copresented, so
$\sD^\b_\acp(G\smooth_k)$ is the bounded derived category of smooth
$G$\+modules over~$k$ with admissible cohomology modules.
 In particular, following
Examples~\ref{loc-quasi-contra-Noetherian-groups}, a p-adic Lie
group $G$ is locally quasi-contra-Noetherian over any field~$k$.
 A $p$\+adic Lie group is also of locally of finite $k$\+cohomological
dimension for any field~$k$, so
Thereom~\ref{G-involutive-duality-theorem} applies.

 Over a field~$k$ of characteristic~$0$, any locally profinite group
is locally of cohomological dimension~$0$ and locally
quasi-contra-Noetherian, so Theorem~\ref{G-involutive-duality-theorem}
is applicable in this case, too.

\begin{rem}
 Notice the difference between Theorems~\ref{H-involutive-duality-thm}
and~\ref{G-involutive-duality-theorem}: while the former provides
an involutive duality on the bounded derived category of admissibly
copresented discrete $H$\+modules, the latter establishes
an involutive duality on the bounded derived category of complexes of
\emph{arbitrary} smooth $G$\+modules \emph{with admissibly
copresented cohomology modules}.

 One can also consider the bounded derived category of the abelian
category of admissibly copresented smooth $G$\+modules,
$\sD^\b(G\smooth_{k,\acp})$.
 There is a natural triangulated functor $\sD^\b(G\smooth_{k,\acp})
\rarrow\sD^\b_\acp(G\smooth_k)$, but generally speaking this functor
is far from being an equivalence.
 It is \emph{neither} fully faithful \emph{nor} surjective on objects.
 (It suffices to consider the case of a \emph{discrete} group $G$,
which is, of course, locally Noetherian and locally of cohomological
dimension~$0$ over any field~$k$.
 Now the admissible $G$\+modules are simply
the \emph{finite-dimensional} representations of the group $G$ over~$k$; 
and, generally speaking, the bounded derived category of
finite-dimensional representations of an infinite discrete group is very
different from the bounded derived category of complexes of
infinite-dimensional representations with finite-dimensional cohomology
modules.)

 Our constructions do \emph{not} seem to allow to obtain an involutive
duality on the derived category of admissible or admissibly copresented
$G$\+modules $\sD^\b(G\smooth_{k,\acp})$.
 The problem is that the construction of the derived functor $\boL\Phi_G$
requres replacing a given complex of $G$\+contramodules over~$k$ with
a complex of $G$\+contramodules that are projective as contramodules
over a certain compact open subgroup $H\subset G$ (see
the proof of Theorem~\ref{derived-equivalence}).
 The known construction of such resolutions~\cite[Section~3.3.3]{Psemi}
leads outside of the class of contraadmissible $G$\+contramodules.
\end{rem}

\begin{rem}
 This paper is inspired by Kohlhaase's paper~\cite{Koh}, and in
the case of a $p$\+adic Lie group $G$ in the natural characteristic
$\operatorname{char}k=p$ the theory developed in the present paper
is closely related to the theory developed in~\cite{Koh}.
 Without going into a detailed comparison, let us mention one obvious
difference between our approaches: no derived categories are
mentioned in~\cite{Koh}.
 The author of~\cite{Koh} avoids derived categories by considering
the sequence of functors Ext in each cohomological degree separately.
 So he has to consider the corresponding sequence of subquotient
categories of the category of admissible smooth $G$\+modules, and
have an involutive duality on each of these subquotient categories
separately.
 In this paper we obtain an involutive duality on a triangulated
category containing the whole abelian category of admissible smooth
$G$\+modules, but this achievement comes with a price: we have to
consider the bounded derived category $\sD^\b_\acp(G\smooth_k)$ of
complexes of nonadmissible smooth $G$\+modules with admissible
cohomology $G$\+modules.
\end{rem}


\bigskip
\begin{thebibliography}{99}
\smallskip

\bibitem{DSMS}
 J.~D.~Dixon, M.~P.~F.~du~Sautoy, A.~Mann, D.~Segal.
   Analytic pro-$p$ groups.
Second edition.  Cambridge Studies in Advanced Mathematics, 61.
Cambridge University Press, 1999--2003.

\bibitem{EM}
 S.~Eilenberg, J.~C.~Moore.
   Homology and fibrations~I.  Coalgebras, cotensor product and
its derived functors.
\textit{Commentarii Mathematici Helvetici} \textbf{40}, p.~199--236,
1966.

\bibitem{FMS}
 L.~Fiorot, F.~Mattiello, M.~Saor\'\i n.
   Derived equivalence induced by nonclassical tilting objects.
\textit{Proceedings of the American Math.\ Society} \textbf{145},
\#4, p.~1505--1514, 2017.  \texttt{arXiv:1511.06148 [math.RT]}

\bibitem{GTNT}
 J.~G\'omez-Torrecillas, C.~N\u ast\u asescu, B.~Torrecillas.
   Localization in coalgebras.  Applications to finiteness conditions.
\textit{Journ.\ of Algebra and its Appl.}\ \textbf{6}, \#2,
p.~233--243, 2007.  \texttt{arXiv:math.RA/0403248}

\bibitem{GW}
 K.~R.~Goodearl, R.~B.~Warfield.
   An introduction to noncommutative Noetherian rings.
Second edition.  London Mathematical Society Student Texts, 61.
Cambridge University Press, 2004.

\bibitem{Har}
 R.~Hartshorne.
   Residues and duality.
\textit{Lecture Notes in Math.} \textbf{20},
Springer-Verlag, Berlin--Heidelberg--New York, 1966.

\bibitem{KS}
 M.~Kashiwara, P.~Schapira.
   Categories and sheaves.
Grundlehren der mathematischen Wissenschaften~332,
Springer-Verlag, Berlin--Heidelberg, 2006.

\bibitem{Koh}
 J.~Kohlhaase.
   Smooth duality in natural characteristic.
\textit{Advances in Math.} \textbf{317}, p.~1--49, 2017.

\bibitem{NSZ}
 P.~Nicol\'as, M.~Saor\'\i n, A.~Zvonareva.
   Silting theory in triangulated categories with coproducts.
\textit{Journ.\ Pure Appl.\ Algebra} \textbf{223}, \#6, p.~2273--2319,
2019.  \texttt{arXiv:1512.04700 [math.RT]}

\bibitem{Pbogom}
 L.~Positselski.
   Koszul property and Bogomolov's conjecture.
\textit{Internat.\ Math.\ Research Notices} \textbf{2005}, \#31,
p.~1901--1936. \texttt{arXiv:1405.0965 [math.KT]}

\bibitem{Psemi}
 L.~Positselski.
   Homological algebra of semimodules and semicontramodules:
Semi-infinite homological algebra of associative algebraic structures.
 Appendix~C in collaboration with D.~Rumynin; Appendix~D in
collaboration with S.~Arkhipov.
 Monografie Matematyczne vol.~70, Birkh\"auser/Springer Basel, 2010. 
xxiv+349~pp. \texttt{arXiv:0708.3398 [math.CT]}

\bibitem{Pkoszul}
 L.~Positselski.
   Two kinds of derived categories, Koszul duality, and
comodule-contramodule correspondence.
\textit{Memoirs of the American Math.\ Society} \textbf{212},
\#996, 2011.  vi+133~pp.  \texttt{arXiv:0905.2621 [math.CT]}

\bibitem{Pweak}
 L.~Positselski.
   Weakly curved A${}_\infty$-algebras over a topological local ring.
\textit{M\'emoires de la Soc.\ Math.\ de France} \textbf{159}, 2018.
vi+206~pp.  \texttt{arXiv:1202.2697 [math.CT]}

\bibitem{Pcosh}
 L.~Positselski.
   Contraherent cosheaves.
Electronic preprint \texttt{arXiv:1209.2995 [math.CT]}.

\bibitem{Prev}
 L.~Positselski.
   Contramodules.
Electronic preprint \texttt{arXiv:1503.00991 [math.CT]}.

\bibitem{Pqf}
 L.~Positselski.
   Koszulity of cohomology $=$ $K(\pi,1)$\+ness $+$ quasi-formality.
\textit{Journ.\ of Algebra} \textbf{483}, p.~188--229, 2017.
\texttt{arXiv:1507.04691 [math.KT]}

\bibitem{PMat}
 L.~Positselski.
   Triangulated Matlis equivalence.
\textit{Journ.\ of Algebra and its Appl.}\ \textbf{17}, \#4,
article ID~1850067, 2018.  \texttt{arXiv:1605.08018 [math.CT]}

\bibitem{Pmc}
 L.~Positselski.
   Dedualizing complexes of bicomodules and MGM duality over
coalgebras.
\textit{Algebras and Represent.\ Theory} \textbf{21}, \#4,
p.~737--767, 2018.  \texttt{arXiv:1607.03066 [math.CT]}

\bibitem{Ppc}
 L.~Positselski.
   Pseudo-dualizing complexes of bicomodules and pairs of t\+structures.
Electronic preprint \texttt{arXiv:1907.03364 [math.CT]}.

\bibitem{PS}
 L.~Positselski, J.~\v St\!'ov\'\i\v cek.
   The tilting-cotilting correspondence.
\textit{Internat.\ Math.\ Research Notices}, published online at
\texttt{https://doi.org/10.1093/imrn/rnz116} in July~2019.
\texttt{arXiv:1710.02230 [math.CT]} {\hbadness=2625\par}

\bibitem{PV}
 L.~Positselski, A.~Vishik.
   Koszul duality and Galois cohomology.
\textit{Math.\ Research Letters} {\bf 2}, \#6, p.~771--781, 1995.
\texttt{arXiv:alg-geom/9507010}

\bibitem{Schn}
 P.~Schneider.
   Smooth representations and Hecke modules in characteristic~$p$.
\textit{Pacific Journ.\ of Math.}\ \textbf{279}, \#1--2,
p.~447--464, 2015.

\bibitem{Ser}
 J.-P.~Serre.
   Sur la dimension cohomologique des groupes profinis.
\textit{Topology} \textbf{3}, \#4, p.~413--420, 1965.

\bibitem{Swe}
 M.~E.~Sweedler.
   Hopf algebras.
Mathematics Lecture Note Series, W.~A.~Benjamin, Inc., New York, 1969.

\bibitem{Tak}
 M.~Takeuchi.
   Morita theorems for categories of comodules.
\textit{Journ. of the Faculty of Science, the Univ.\ of Tokyo.
Section~1\,A, Math.} \textup{24}, \#3, p.~629--644, 1977.

\bibitem{WW}
 M.~Wang, Z.~Wu.
   Conoetherian coalgebras.
\textit{Algebra Colloquium} \textbf{5}, \#1, p.~117--120, 1998.

\end{thebibliography}
\end{document}